\documentclass[11pt]{article}

\usepackage{graphicx,,psfrag}

\usepackage{amssymb,amsmath,amsthm,enumerate,bm}
\usepackage{fullpage}

\usepackage{tikz}

\renewcommand{\star}{{\mathrm{F}}}

\usepackage[colorlinks=true,linkcolor=blue,citecolor=red,urlcolor=blue,pdfborder={0 0 0},pagebackref=true]{hyperref}

\newtheorem{lemma}{Lemma}[section]
\newtheorem{remark}[lemma]{Remark}
\newtheorem{proposition}[lemma]{Proposition}
\newtheorem{definition}[lemma]{Definition}
\newtheorem{theorem}[lemma]{Theorem}
\newtheorem{corollary}[lemma]{Corollary}

\newcommand{\cA}{\mathcal{A}}

\newcommand{\EE}{{\mathbf{E}}}

\newcommand{\tr}{{\rm tr}}
\newcommand{\Tr}{{\rm Tr}}

\newcommand{\dE}{\mathbf {E}}
\newcommand{\dP}{\mathbf{P}}
\newcommand{\dF}{\mathrm{F}}

\newcommand{\dU}{\mathrm {U}}
\newcommand{\dS}{\mathrm{S}}

\newcommand{\dO}{\mathrm{O}}
\newcommand{\dSO}{\mathrm{SO}}
\newcommand{\dR}{\mathbb {R}}
\newcommand{\dC}{\mathbb {C}}

\newcommand{\cB}{\mathcal{B}}

\newcommand{\cW}{\mathcal {W}}
\newcommand{\cP}{\mathcal {P}}

\newcommand{\cM}{\mathcal {M}}

\newcommand{\FULL}{ \mathrm{full}}

\newcommand{\ABS}[1]{{{\left| #1 \right|}}} 

\newcommand{\BRA}[1]{{{\left\{#1\right\}}}}
\newcommand{\SBRA}[1]{{{\left[#1\right]}}} 
\newcommand{\NRM}[1]{{{\left\| #1\right\|}}} 
\newcommand{\PAR}[1]{{{\left(#1\right)}}} 

\newcommand{\AND}{\quad \mathrm{and} \quad}

\newcommand{\uB}{\underline{B}}

\newcommand{\prof}{\mathfrak{p}}

\newcommand{\1}{1\!\!{\sf I}}
\newcommand{\IND}{\1}
\newcommand{\veps}{\varepsilon}

\newcommand{\INT}[1]{{{\left[\hspace{-1.4pt}\left[#1\right]\hspace{-1.4pt}\right]}}}

\definecolor{darkred}{rgb}{0.9,0,0.3}

\begin{document}

\title{Norm of matrix-valued polynomials in random unitaries and permutations}
\author{Charles Bordenave and Beno\^\i{}t Collins}

\maketitle

\begin{abstract}
We consider a non-commutative polynomial in several independent $N$-dimensional random unitary matrices, uniformly distributed over the unitary, orthogonal or symmetric groups, and assume that the coefficients are $n$-dimensional matrices. 
The main purpose of this paper is to study the operator norm of this random non-commutative polynomial. We compare it with its counterpart, where the random unitary matrices are replaced by the unitary generators of the free group von Neumann algebra. 

Our first result concerns the upper bound of the sequence: it is bounded above by an arbitrarily small inflation of its free counterpart with an overwhelming probability 
 in the large $N$ limit, and this estimate is uniform over all matrix coefficients as long as $n \le\exp (N^\alpha)$ for
some explicit $\alpha >0$. 
Such results had been obtained by very different techniques for various regimes, all falling in the category $n\ll N$. Our result provides a new proof of the Peterson-Thom conjecture.

Our second result is a universal quantitative lower bound for the operator norm of polynomials in independent $N$-dimensional random unitary and permutation matrices with coefficients in an arbitrary $C^*$-algebra. A variant of this result for permutation matrices generalizes the Alon-Boppana lower bound in two directions: firstly, it applies to arbitrary polynomials and not only linear polynomials, and secondly, it applies to coefficients of an arbitrary $C^*$-algebra with non-negative joint moments and not only for non-negative real numbers. 

\end{abstract}

\section{Introduction}

\subsection{Motivation}
\label{subsec:motiv}

Fix an integer $d \geq 1$ and consider a sequence of $d$-tuple $(X_{N,1},\ldots X_{N,d})$ of matrices in $M_N(\dC)$ of growing dimension $N \to \infty$. For ease of exposition we remove the explicit dependence in $N$ and assume that the set $X = \{X_1,\ldots X_d\}$ is stable by conjugation, i.e., $X_i^* \in X$ for all $i \in \{1,\ldots,d\}$. To understand the algebra generated by these matrices 
in the large $N$ limit, it is necessary
to understand for any non-commutative polynomial $P \in \dC[\dF_d]$, the operator norm  of 
\begin{equation}\label{eq:PNintro}
P_N = P(X_1,\ldots,X_d),
\end{equation}
as the dimension $N$ goes to infinity (here and below $\dF_d$ is the free group with $d$ free generators). 
In operator algebraic random matrix theory, one says that $(X_1,\ldots,X_d)$ converges strongly toward $(x_1,\ldots,x_d)$, elements in a unital $C^*$-algebra $\cA \subset \cB(H)$ of bounded operators on some Hilbert set $H$, if for all non-commutative polynomials $P \in \dC[\dF_d]$, the operator norm of   $P(X_1,\ldots,X_d)$ converges toward the operator norm of $P(x_1,\ldots,x_d)$.

Thanks to a linearization trick, establishing the strong convergence of $(X_1,\ldots,X_d)$  is
equivalent to proving that for any integer $n \geq 1$ and matrices $(a_0,\ldots,a_d)$ in $M_n(\dC)$, setting $X_0 = 1_N$, the operator norm of 
\begin{equation}\label{eq:ANintro}
A_N =  \sum_{i=0}^d a_i \otimes X_i,
\end{equation}
converges toward the operator norm of 
\begin{equation*}
A  =  \sum_{i=0}^d a_i \otimes x_i,
\end{equation*}
with $x_0 = 1_H$ (more precisely, by a result of Pisier \cite{MR1401692} this is equivalent when the $X_i's$ are unitaries, see Subsection \ref{subsec:strong} below. In the general case, one needs to compare the spectrum of $A_N$ and $A$, we refer to \cite{MR3585560}). 

Establishing strong convergence in its form \eqref{eq:PNintro} or \eqref{eq:ANintro} for non-trivial sequences of matrices always has very powerful consequences both for the matrices themselves and for the $C^*$-algebra $\cA$. The latter was the motivation of the seminal work of Haagerup and Thorbj{\o}rnsen \cite{MR2183281}, which proves the strong convergence of $d$ independent GUE random matrices toward free semi-circular variables. Among other references, we may cite, for example, the Peterson-Thom conjecture \cite{MR2827095,MR4448584,belinschi2022strong},  matrix concentration inequalities \cite{bandeira2023matrix,brailovskaya2022universality}, the Alon's generalized spectral gap conjecture in random graphs \cite{MR4024563}, the cut-off phenomenon for finite Markov chains \cite{MR4476123}, Buser's conjecture on the largest possible spectral gap of the closed hyperbolic surface of high genus \cite{hide2023near,louder2023strongly}, representation theory of classical Lie groups \cite{BC2}, quantum information theory \cite{MR3452274,MR2995183} or operator space version of Grothendieck's inequality \cite{MR3273440}. However, this quest for strong convergence is only beginning. In all examples where strong convergence has been proved, the matrices $(X_1,\ldots,X_d)$ are taken random and independent, and they converge strongly toward free operators.

To our knowledge, no known non-trivial examples of strongly convergent deterministic sequences of matrices exist. Toward the objective of reducing the role of randomness and capturing some `intrinsic freeness' (to quote \cite{bandeira2023matrix}),  one may consider quantitative versions where, for $A_N$ in \eqref{eq:ANintro}, the matrices $(a_0,\ldots,a_d)$ are also of growing dimension $n \to \infty$. This question is also directly motivated by the above-mentioned applications in operator algebra as in \cite{MR3273440,MR4448584,belinschi2022strong}. This problem is the central matter in  \cite{MR3273440,MR4445344,MR4408503,bandeira2023matrix,brailovskaya2022universality,belinschi2022strong,parraud2023asymptotic}.

Let us give a partial account on this body of works for the matrix $A_N$ in \eqref{eq:ANintro}. As $N \to \infty$, for $n = O(1)$, Haagerup  and Thorbj{\o}rnsen 
 \cite{MR2183281} prove that $\|A_N\| = (1+o(1)) \| A\|$ with $(X_1,\ldots,X_d)$ independent GUE random matrices and $(x_1,\ldots,x_d)$ free semi-circular elements. This was extended for independent Haar unitaries by Collins and Male \cite{MR3205602} (with $(x_1,\ldots,x_d)$ free Haar unitaries) and for independent random permutation matrices by the present authors in \cite{MR4024563}. For GUE random matrices (actually Ginibre), Pisier \cite{MR4448584} proves that $\|A_N\| \leq (1+o(1)) \| A \|$ if $n = o(N^{1/4})$ (and $(x_1,\ldots,x_d)$ free circular elements). Among other refinements, Collins, Guionnet, and Parraud \cite{MR4445344} introduce a new strategy of proof and obtain $\|A_N\| \leq (1+o(1)) \| A\|$ for $n = o(N^{1/3})$ and Parraud \cite{MR4408503} proves the analogous result for independent Haar distributed unitary matrices. Taking a very different perspective, Bandeira, Boedihardjo, and van Handel \cite{bandeira2023matrix} establish, as a byproduct of a powerful general theorem, that $\|A_N\| \leq (1+o(1)) \| A \|$  for independent GUE matrices when $n = o (N / (\log N)^3)$ and for a specific random choice of matrices $(a_0,\ldots,a_d)$ which is dictated by Hayes' criterion \cite{MR4448584} for the Peterson-Thom conjecture (see Subsection \ref{subsec:PT} below). In \cite{parraud2023asymptotic}, Parraud improves his previous work on independent Haar unitaries and obtains the bound $\|A_N\| \leq (1+o(1)) \| A \|$ as soon as $n = o (N/(\log N)^{5/2})$ for the choice of matrices $(a_0,\ldots,a_d)$ dictated by Hayes' criterion. A few months earlier, this conjecture was, however, settled by Belinschi and Capitaine \cite{belinschi2022strong}: improving on  \cite{MR4445344}, they could reach $\|A_N\| \leq (1+o(1)) \| A \|$ for $n = O(N)$ and the specific random choice of matrices $(a_0,\ldots,a_d)$ identified by Hayes \cite{MR4448584}.

From this list of recent results on independent GUE matrices and Haar unitaries of dimension $N$, one might guess that $n = N^{1+o(1)}$ is a barrier for having a strong asymptotic freeness for all choices of matrices $(a_0,\ldots,a_d) \in M_n(\dC)$. On the other hand, inspired by  \cite{MR1710376},  Pisier \cite[Corollary 4.7]{MR3273440}  proves that, for independent Gaussian random matrices, $\| A_N \| \leq (2+o(1)) \| A \|$ (with $x_i$ free circular elements) as soon as $n =  \exp (o(N) )$: we trade a factor $2$ for an exponentially larger threshold in $n$ valid for all matrices $(a_0,\ldots,a_d)$ in $M_n(\dC)$! However, Pisier also shows that for any $C >1$ and $d$ large enough, $\|A_N\|$ cannot be smaller that  $ C \|A \|$ for a carefully designed family of matrices $(a_0,\ldots,a_d)$ of dimension $n =  \exp( \Theta(N^2))$. Hence from this perspective, the regime $n =  \exp( \Theta(N^\alpha))$ with $\alpha >0$ is the pertinent scaling.

The first objective of this paper is to fill this huge gap between the regimes $n = O(N)$ where strong asymptotic freeness has been proven for very specific choices of $(a_0,\ldots, a_d)$ and the regime $n = \exp ( \Theta(N^\alpha))$ identified by Pisier.  Theorem \ref{th:main1} proves strong asymptotic freeness in a very quantitative way for independent Haar unitaries for all $(a_0,\ldots,a_d)$ in $M_n (\dC)$ with $n \leq \exp ( N^{\alpha})$ and an explicit constant $\alpha >0$. Among other consequences, this result gives an alternative proof of Hayes criterion  \cite{MR4448584} for the Peterson-Thom conjecture, see Subsection \ref{subsec:PT}. In the more challenging case of independent random permutations, Theorem \ref{th:main3} proves a similar quantitative result when $n \leq N^{(\log N)^{1-o(1)}}$ which is already much larger than what was available for GUE or Haar unitary matrices.  This result gives the first quantitative version of the main result in \cite{MR4024563} and has an interesting consequence for random actions of the Cartesian product of free groups on finite sets, see Subsection \ref{subsec:PT}.

The second objective  of this paper is to provide also matching lower bounds of the form $\| A_N \| \geq (1-o(1)) \| A \|$. As pointed out in \cite{bandeira2023matrix}, even if the community has, with reason, focused on upper bounds on $\|A_N\|$, it is not straightforward to prove lower bounds when $n,N$ goes to infinity.  Our main finding is that there is a universal lower bound of the form $\| A_N \| \geq (1-o(1)) \| A \|$ which is uniform in $n$, see Theorem \ref{th:LBU1}. As we shall see, this is ultimately a consequence of the exactness of the free group reduced $C^*$-algebra. There is, thus, a sharp contrast between lower and upper bounds in this perspective. In the specific case of permutations, this lower bound can be interpreted as a universal Alon-Boppana lower bound, see Theorem \ref{th:AB}, which is of independent interest.

Finally, in Section \ref{sec:poly}, we will also explain how to transfer all these estimates on $\|A_N\|$ into estimates on $\|P_N\|$ with $P_N$ as in \eqref{eq:PNintro} by a quantitative version of the linearization trick. We also recall how we can easily deduce from this the convergence of the spectrum of $P_N$ to the spectrum of $P$ quantitatively when $P_N$ and $P$ are self-adjoint. 

\subsection{Models and main results}

\label{sec:model}

Let $\cA_1,\cA_2$ be unital $C^*$-algebras equipped with faithful normal   states $\tau_1,\tau_2$. We denote by $\cA_1 \otimes  \cA_2$ the minimal tensor product of $\cA_1$ and $\cA_2$.  The state $\tau = \tau_1 \otimes \tau_2$ is a faithful normal   state on $\cA_1 \otimes \cA_2$.

We are mainly interested in two specific examples for $\cA_2$. The first is $\cA_2 = M_N(\dC)$, $N \geq 1$ integer and $\tau_2 = \tr_N$ is the normalized trace $\tr_N = \frac 1 N \Tr$. The second example is $\cA_2 = \cA_{\star}$ the left group algebra of the free group $\dF_d$ with free generators $g_1, \ldots ,g_d$ and their inverses. The algebra $\cA_\star$ is the reduced $C^*$-algebra generated by the convolution operators $\lambda(g)$ on $\ell^2(\dF_d)$ defined for $g ,h \in \dF_d$ by
$$
\lambda(g) \delta_h = \delta_{gh},
$$
where $\delta_g$ is the unit coordinate vector at $g \in \dF_d$. Note that $\lambda(g)$ is unitary and $\lambda(g^{-1}) = \lambda(g)^*$. For  integer $i \in  \INT{d} = \{1,\ldots , d\}$, we set $i^* = i+d$ and for $d+1 \leq i \leq 2d$, we set $i^* = i-d$ and $g_{i} = g_{i - d}^{-1}$. By construction, $i \mapsto i^* $ is an involution of $\INT{2d}$.

In $\cA_1 \otimes \cA_\star$, we define 
\begin{equation}\label{eq:defA*}
A_\star = a_0 \otimes 1  + \sum_{i=1}^{2d} a_i \otimes \lambda(g_i),
\end{equation}
with $a_0, a_i \in \cA_1$ for all $i \in \INT{2d}$. The operator $A_\star$ is self-adjoint if the $a_i$'s are such that for all $i \in \INT{2d}$, 
\begin{equation}\label{eq:symai}
a_i ^ * = a_{i^*} \quad \hbox{ and } \quad a_0^* = a_0.
\end{equation}

Next, for any integer $N \geq 1$, let $(U_1,\ldots, U_d)$ be independent and Haar-distributed elements on $\dU_N$ or $\dO_N$, we define the operator in $\cA_1 \otimes M_N(\dC)$,
\begin{equation}\label{eq:defA}
A_N = a_0 \otimes 1  + \sum_{i=1}^{2d} a_i \otimes U_i,
\end{equation}
where for $d +1 \leq i \leq 2d$, we have set $U_{i} = U_{i^*}^*$ and $a_i \in \cA_1$ are as above. Under the condition  \eqref{eq:symai}, the operator $A_N$ is self-adjoint. We will also consider below the case when $(U_1,\ldots,U_d)$ are independent uniform permutation matrices of size $N$.

As mentioned in the introduction, the central objective of this paper is to establish that, under rather weak conditions, on $\cA_1$, the norm of $A_N$ is close to the norm of $A_\star$ up to a factor close to $1$.  For $p \geq 1$, we set, with $T \in \cA_1 \otimes \cA_2$ and $|T| = \sqrt{ T T^*}$,  
$$
\| T \|_p = \PAR{ \tau (|T|^{p} ) }^{\frac 1 {p}}.
$$
Gelfand's theorem implies that for any $T \in \cA_1$, $\| T \|_p$ converges  as $p \to \infty$ to the operator norm denoted by $\| T \|$. Haagerup's inequality on the free group algebra or, more generally, the rapid decay property quantifies this convergence; see \S \ref{subsubsec:haagerup} below. If $\cA_1 = M_n(\dC)$, $\cA_{2} = M_N(\dC)$, we have 
\begin{equation}\label{eq:TtoTp}
\| T \| \leq (nN)^{1/p} \| T \|_p.
\end{equation}

Our first main result is the following universal upper bound on $\| A_N \|_p$. 
\begin{theorem}
\label{th:main1}
Let $(a_0,\ldots,a_{2d})$ in $\cA_1$ and $A_\star$, $A_N$ be as in \eqref{eq:defA*}-\eqref{eq:defA} with $d^{70} \leq N$ and  $(U_1,\ldots, U_d)$ be independent and Haar-distributed elements on $\dU_N$ or $\dO_N$. There exists a numerical constant $c >0$ such that for  $p = N^{1/(16d + 80)}$,
$$
  \dE \|A_N \|_p  \leq   \| A_\star \| \PAR{ 1 + \frac{c}{ \sqrt N}}.
$$
\end{theorem}

When $\cA_1 = M_n(\dC)$, we may use \eqref{eq:TtoTp} and deduce, for example, the following corollary.
\begin{corollary}\label{cor:main1}
With the assumptions of Theorem \ref{th:main1}, assume further that $\cA_1 = M_n(\dC)$ with $n \leq \exp ( N^{1/(32d + 160)} )$, then, for some numerical constant $c >0$, 
$$
\dE \| A_N \|  \leq \| A_\star \| \PAR{ 1 + c N^{-\frac{1}{32 d + 160} }}.
$$
\end{corollary}

The constants and exponents in Theorem \ref{th:main1} are obviously not optimal. 
As explained in introduction, for $\cA_1 = M_n(\dC)$, the order of magnitude in $\log n$ is correct in Corollary \ref{cor:main1}.
A more precise bound will be given in Theorem \ref{th:FKBlm} below. For example, for $p = N^{o(1)}$ and $d = O(1)$, we can replace $c/\sqrt N$ by $N^{-1+o(1)}$. It will also imply that $\|A_N\|_p$ is upper bounded with high probability. The dependency in $d$ in the expression of $p$ in Theorem \ref{th:main1} can be improved under some assumptions on the elements $(a_0,\ldots,a_{2d})$. The next theorem illustrates two such examples.  

\begin{theorem}
\label{th:main2}
Let $(a_0,\ldots,a_{2d})$ in $\cA_1$ and $A_\star$, $A_N$ be as in \eqref{eq:defA*}-\eqref{eq:defA} with  $(U_1,\ldots, U_d)$ be independent and Haar-distributed elements on $\dU_N$ or $\dO_N$. We assume either that (i) the elements $(a_0,\ldots,a_{2d})$ are bistochastic matrices in $M_n(\dR_+)$ or (ii) $\cA_1 = M_m(\dC) \otimes \cA_0$ and for all $ i \in \INT{ 2d}$, $a_i = b_i \otimes u_i$ with $b_i \in M_m(\dC)$ and $u_i \in \cA_0$ unitary. In case (i), we set $\theta =1$, in case (ii), we set $\theta = m$. We assume that $d^{40} \theta^{30} \leq N$. There exists a numerical constant $c >0$ such that for  $p =  N^{1/80}$,
$$
\dE \| A_N \|_p  \leq  \| A_\star \| \PAR{ 1 + \frac{c}{\sqrt N}}.
$$
\end{theorem}

We next consider an extension of the previous results to the more delicate case when the random Haar unitary or orthogonal matrices are replaced by random Haar permutations: $(U_1,\ldots,U_d)$ are independent and uniformly distributed permutation matrices in $\dS_N$. The operator, $A_N$, is defined as in \eqref{eq:defA}. If $\cA_1 \subset \cB(H_1)$ is a $C^*$-algebra of bounded operators on a vector space $H_1$ and $\IND_N \in \dC^N$ is the unit vector with all coordinates equal to $1/\sqrt N$, the vector space $H_1 \otimes \dC \IND_N$ is invariant by $A_N$ and its adjoint. The restriction of $A_N$ to $H_1 \otimes \dC \IND_N$ is isomorphic to 
$$A_1 = a_0 + \sum_{i=1}^{2d} a_i  \in \cA_1.$$
We are then interested in the norm of $A_N \Pi_N = \Pi_N A_N$ where 
\begin{equation}\label{eq:defPiN}
\Pi_N = 1_{H_1} \otimes \PAR{1_{\dC^N} -  \IND_N \otimes \IND_N}
\end{equation}
is the orthogonal projection onto $H_1 \otimes \IND_N^\perp$, the orthogonal of $H_1 \otimes \dC \IND_N$. 

\begin{theorem}
\label{th:main3}
Let $N \geq 3$, $(a_0,\ldots,a_{2d})$ in $\cA_1$ and $A_\star$, $A_N$ be as in \eqref{eq:defA*}-\eqref{eq:defA} with $2 \leq d \leq \log ( N) / (20 \log \log N)$ and $(U_1,\ldots, U_d)$ be independent and uniform permutation matrices on $\dS_N$. For any $0 < \veps < 1$, there exists a constant $c(\veps) >0$ such that with probability at least $1 - c(\veps) N^{\veps -1}$, for any $1 \leq p  \leq (\log ( N)) ^2 / (c(\veps) (\log d ) \log \log N)$, 
$$
 \| A_N\Pi_N \|_{p}  \leq  \| A_\star \| \PAR{ 1 + c(\veps)\frac{(d \log d) (\log \log N) }{\log N}}.
$$
\end{theorem}

For $\cA_1 = M_n(\dC)$, we have, for example, the following corollary of Theorem \ref{th:main3}.
\begin{corollary}\label{cor:main3}
In the setting of Theorem \ref{th:main3}, assume $\cA_1 = M_n(\dC)$. For any $0 < \veps < 1$, there exists $c(\veps) >0$ such that if  $n  \leq N^{\sqrt{ \log N}}$ and $2 \leq d \leq \sqrt{\log ( N)}$ then,  with probability at least $1 - c(\veps) N^{\veps -1}$, 
$$
 \| A_N \Pi_N \| \leq \| A_\star \| \PAR{ 1 +  \frac{c(\veps)}{(\log N)^{1/4}}}.
$$
\end{corollary}

This corollary improves on the main result of \cite{MR4024563} in two ways: it is quantitative and allows $n$ to grow super-polynomially in $N$ (while $n$ was fixed in \cite{MR4024563}). Note that the error bound is far from optimal: in the simplest case, $n=1$ and $a_i = 1$ for all $i \in \INT{2d}$, then polynomial error bounds are proved in the impressive works \cite{MR4135670,HuangYau}. 

There is also be an analog of Theorem \ref{th:main2} for random permutations, which allows to improve the dependency in $d$. We do not state it for the sake of brevity. 

We conclude the results of this section by lower bounds on $\|A_N\|$. It is possible to obtain lower bounds on $\| A_N \|$ when $\cA_1 = M_n(\dC)$ and $n$ is not too large (say $n \leq N^{\alpha}$ with $0 < \alpha <1$). However, it is not straightforward to obtain a universal lower bound in the sense that it does not depend on $\cA_1$. It turns out that it is possible to compute such bounds thanks to the exactness of $\cA_\star$.

 We start with the case of permutations. Let $\sigma = (\sigma_1,\ldots, \sigma_d)$ be permutations in $\dS_N$ with permutation matrices $(U_1,\ldots, U_d)$. Let $G_\sigma$ be the Schreier graph associated to $\sigma$: it is  the graph on $\INT{N}$ whose adjacency matrix is given by $\sum_{i=1}^{2d} U_i$ with $U_{i+d} = U_i^*$ for $i\in \INT{d}$. For $x \in \INT{N}$ and $h>0$, let $(G_\sigma ,x)_h$ be the sub-graph of $G_\sigma$ spanned by vertices at distance at most $h$ from $x$. 

\begin{theorem}\label{th:LBP}
Let $(a_0,\ldots,a_{2d})$ in $\cA_1$ and $A_\star$, $A_N$ be as in \eqref{eq:defA*}-\eqref{eq:defA} with  $U_1,\ldots, U_d$ permutation matrices of $\sigma = (\sigma_1,\ldots, \sigma_d)$ in $\dS_N$. Let $p \geq 1$ be an integer and assume that there exists $x,y \in \INT{N}$ such that  $(G_\sigma,x)_p$ and $(G_\sigma,y)_p$ are disjoint and without cycle. Then, for some numerical constant $c >0$,
$$
\| A_N \Pi_N \| \geq \| A_\star \| \PAR{ 1- c \frac{ \sqrt{ \log (2d)}}{\sqrt p}    }.
$$
\end{theorem}

If $\sigma$ is uniform on $\dS_N^d$, it is standard that the conditions of the lemma are fulfilled with $p$ of order $\log_{2d-1} N$ with high probability (see, for example, \cite[Lemma 23]{MR4024563}).  A variant of the proof of Theorem \ref{th:LBP} will also give a universal Alon-Boppana lower bound; see \cite{MR1124768,MR1978881} for background. We say that $(a_0,\ldots,a_{2d})$ in $\cA_1$ have {\em non-negative joint moments}  if $\tau_1 ( a^{\veps_1}_{i_1}\cdots a^{\veps_n}_{i_n} ) \geq 0$ for all $n$ and $(i_1,\ldots,i_n) \in \{0,\ldots,2d \}^n$ and $a_{i_k}^{\veps_k} \in \{a_{i_k},a_{i_k}^*\}$. This property was notably studied in \cite{lehnerthesis}. The iconic example is $\cA_1 = M_n(\dC)$ and $a_i \in M_n (\dR_+)$.

\begin{theorem}\label{th:AB}
Let $(a_0,\ldots,a_{2d})$ in $\cA_1$ with non-negative joint moments and $A_\star$, $A_N$ be as in \eqref{eq:defA*}-\eqref{eq:defA} with  $(U_1,\ldots, U_d)$ permutation matrices. Then, for some numerical constant $c >0$,
$$
\| A_N \Pi_N \| \geq \| A_\star \| \PAR{ 1- c \frac{ \log (2d) }{\sqrt{\log N}} }.
$$
\end{theorem}

The remarkable fact is that the statement does not depend on the permutation matrices (apart from their dimension), and it does not depend on the $C^*$-algebra $\cA_1$. The classical Alon-Boppana bound is for $\cA_1 = \dC$ and $a_i = 1$; better bounds are known in this case, see \cite{MR2679612} (or more generally when $\cA_1 = M_n(\dC)$ and $n$ is not too large thanks to Haagerup inequality, see Lemma \ref{le:Haagerup} below). If $\cA_{1} = M_n(\dC)$, Theorem \ref{th:AB} can for example be used for $N$-lifts of a base graph on $n$ vertices and $d$ edges (see \cite{MR1883559,MR1978881,MR4024563}). In this case, the bound can be improved by taking into consideration the maximal degree, say $\Delta$ of the base graph: the term $\log d / \sqrt{\log N}$ can be replaced by  $\sqrt{  \log d \log \Delta} / \sqrt{\log N}$. A similar remark holds for Theorem \ref{th:LBP}.

For Haar unitaries, we will prove the following universal lower bound. 

\begin{theorem}\label{th:LBU1}
Let   $(a_0,\ldots,a_{2d})$ in $\cA_1$ and $A_\star$, $A_N$ be as in \eqref{eq:defA*}-\eqref{eq:defA} with  $(U_1,\ldots, U_d)$ are independent Haar-distributed elements on $\dU_N$. There exists a numerical constant $c >0$ such that, 
$$
\dE \| A_N \| \geq \|A_\star\|  \PAR{ 1 -  \frac{c  \log (2d)}{\log N}}.
$$
\end{theorem}
We note that the random variable $\| A_N \|$ is tightly concentrated around its mean (see Lemma \ref{le: GM}). For the case of the orthogonal group or for an improvement of the lower bound in Theorem \ref{th:LBU1}, see Remark \ref{rk:HaarON} (if $\cA_1 = M_n(\dC)$ and $n \leq N^{\alpha}$, $0 < \alpha < 1$, it is possible to extract polynomial error bounds from this work).

As explained in \S \ref{subsec:motiv}, all these results can be extended in a quantitative to arbitrary polynomials in the unitaries $(U_1,\ldots,U_d)$. We postpone this discussion to Section \ref{sec:poly}.

\subsection*{Comparing this manuscript with \cite{MR4024563,BC2}}
This paper is a continuation of a joint effort of the authors to understand the operator norm of matrix models build with independent unitary matrices  in large dimension through non-backtracking techniques. 
While some results of this paper solve important problems in operator algebras (the Peterson Thom conjecture and Pisier's exponential estimates), as well as in graph theory  (a universal Alon Boppana bound), we stress that from the technological point of view, this paper is not an incremental refinement of  \cite{MR4024563,BC2}. Rather, many new theoretical tools are developed, mostly on the operator algebraic side, which we enumerate below:
\begin{enumerate}
\item The non-backtracking equation (see section \ref{sec:nonback}) is completely new. Unlike the previous versions, it is completely combinatorial and relies on little to no complex analysis, and therefore allows for very efficient estimates.
\item A new tensor version of a Cauchy-Schwarz inequality is needed to overcome problems intrinsic to the high non-commutativity of high dimensional coefficients. 
\item A completely new summation method is introduced to control the moments
\item A new effective version of the linearization trick is introduced. 
\item An effective operator theoretic characterization of exactness in the case of the reduced group $C*$-algebra is introduced and used in a crucial way to ensure control on the lower bound of the matrix models.  
\end{enumerate}
We expect that these new tools will again prove useful in subsequent studies of the spectrum of matrices with tensor structure 

\subsection*{Acknolowedgements}

The authors are very grateful to Narutaka Ozawa for teaching us some operator space theoretic aspects of exactness. They thank Gilles Pisier for pointing the reference \cite{lehnerthesis}. 
They also thank Michael Magee and Will Hide for constructive comments on the first version of the paper.
BC would like to thank Wangjun Yuan, Ben Hayes, F\'elix Parraud for multiple discussions on the Peterson-Thom conjecture.  

This work was completed while CB was a visiting professor at Kyoto University in the context of a one-year RIMS program, and he acknowledges the hospitality of Kyoto University.  BC was supported by JSPS KAKENHI 17K18734 and 17H04823.

\section{Overview of the proofs for the upper bounds}

\subsection{Overview of Theorem \ref{th:main1} and Theorem \ref{th:main2}}

By a standard linearization trick, it is sufficient to prove Theorem \ref{th:main1} and Theorem \ref{th:main2} under the symmetry assumption \eqref{eq:symai}. Under this assumption, $A_N$ and $A_\star$ are self-adjoint and thus for any even integer $\ell \geq 2$, 
\begin{equation*}
(\dE \| A_N \|_\ell ) ^\ell \leq  \dE  ( \| A_N \|^\ell_\ell) = \dE \tau \PAR{  A_N^\ell }.
\end{equation*}
The goal is then to show that for $\ell$ as large as possible, 
\begin{equation}\label{eq:Ostep0}
\dE \tau \PAR{  A_N^\ell } \leq  (1 + o(1))^\ell \| A_\star\|^\ell,
\end{equation}
where $o(1)$ is an explicit vanishing bound as $N \to \infty$. This is basically the expected high trace method introduced by F\H{u}redi and Koml\'os \cite{MR637828}. For lighter notation, we now remove the explicit dependency in $N$ and write $A$ instead of $A_N$.

\paragraph{Step 1: a new non-backtracking decomposition. } Non-backtracking operators have proven to be a very powerful tool in conjunction with the expected trace method; see \cite{Friedman1996,MR2377835,MR2647136,MR4116720,BLM,FrKo,MR4203039} to cite a few related references. Roughly speaking, non-backtracking operators can be thought as discrete geodesic flows and, as in Selberg's trace formulas, there are spectral correspondences between diffusive operators and non-backtracking operators. This correspondence is usually called Ihara-Bass identity. Operator-valued non-backtracking operators were introduced in \cite{MR4024563,BC2} to study operators of the form \eqref{eq:defA}. These operators have, however, some important drawbacks: they are not self-adjoint, the spectral correspondence is established only for $\cA_1 = M_n(\dC)$, and it is difficult to extract quantitative information from this correspondence.

In Section \ref{sec:nonback} of the present paper, we overcome these limitations and introduce a new family of non-backtracking operators with two integer parameters $B^{(\ell,m)}$ which are self-adjoint and such that 
\begin{equation}\label{eq:-1decompA}
A^\ell = \sum_{m=0}^\ell B^{(\ell,m)}.
\end{equation}
The parameter $m$ encodes the number of non-backtracking steps. Indeed, specified to the free group $\dF_d$ and on the operator $A_\star$ defined in \eqref{eq:defA*} this decomposition simply amounts to writing
\begin{equation}\label{eq:OdecompA}
A_\star^\ell = \sum_{g \in \dF_d} a(\ell,g) \otimes \lambda(g)  = \sum_{m=0}^{\ell} \PAR{ \sum_{g \in S_m} a(\ell,g) \otimes \lambda(g)}  = \sum_{m=0}^\ell B_\star^{(\ell,m)},
\end{equation}
where $a(\ell,g) \in \cA_1 $, $S_0 = \{\o\}$ is the unit of $\dF_d$, and for $m \geq 1$, $S_m$ is the set of elements $g \in \dF_d$ written in reduced form as $g = g_{i_m}g_{i_{m-1}} \cdots g_{i_1}$ with $i_{t+1}^* \ne i_t$ (that is the set of elements at word distance exactly $m$ from the unit). Viewing $A_\star$ as an $\cA_1$-operator-valued operator on $\ell^2(\dF_d)$, $a(\ell,g)$ is equal to the entry $(g,\o)$ of $A_\star^\ell$. 

As a byproduct, using the linearity of the trace and the expectation, we find in \eqref{eq:Ostep0}, 
\begin{equation}\label{eq:Ostep1}
\dE \tau \PAR{  A^\ell }  = \sum_{m=0}^\ell \dE  \tau \PAR{  B^{(\ell,m)} }. 
\end{equation}
We have thus reduced our task to controlling from above the trace of an operator-valued self-adjoint non-backtracking operator.

\paragraph{Step 2: expected high trace method. } 

In Section \ref{sec:FK}, after expanding the trace  $\tau = \tau_1 \otimes \tr_N$ of $B^{(\ell,m)}$  $\tau = \tau_1 \otimes \tr_N$ over $\tr_N$, and some manipulations, we encounter the expression, for $1 \leq m \leq \ell$,
$$
\dE \tau \PAR{  B^{(\ell,m)} } =   \frac 1 N  \tau_1 \PAR{ \sum_{\gamma = (g,x)} a(\ell,g) w(\gamma)  },
$$
where the sum is over $\gamma = (g,x) \in S_m \times \INT{N}^{m+1}$, $a(\ell,g) \in \cA_1$ as in \eqref{eq:OdecompA} and, if $g = (g_{i_m}g_{i_{m-1}} \cdots g_{i_1})$ and $x = (x_0,\ldots,x_m)$, $x_0 = x_m$,
$$w(\gamma ) = \dE \prod_{t=1}^{m} (U_{i_t})_{x_{t-1} , x_{t}}$$
is an expected product of entries of the random unitary matrices. An element $\gamma$ can be thought of as a closed non-backtracking path. 

We partition the elements $\gamma$ in terms of equivalence classes $\pi$ on which $w(\gamma)$ is constant.
Calling $p(\pi)$ this common value, we get the upper bound: 
$$
\dE \tau \PAR{  B^{(\ell,m)} } \leq  \frac 1 N  \sum_{\pi} |p(\pi)| \NRM{ \sum_{\gamma \in \pi} a(\ell,g)   },
$$
where $\NRM{\cdot}$ is the norm on $\cA_1$ and the first sum is over all equivalence classes $\pi$.
The probabilistic weight $w(\gamma)$ is estimated using Weingarten calculus. This work was already done in \cite{BC2}. Our central technical contribution (Lemma \ref{le:sumagamma}) is an upper bound of the form for some explicit multiplicative factor $L (\pi)$,
\begin{equation}\label{eq:Osumagamma}
\NRM{\sum_{\gamma \in \pi} a(\ell,g)} \leq L (\pi)\| A_\star \|^\ell.
\end{equation}
We have thus obtained the upper bound
$$
\dE \tau \PAR{  B^{(\ell,m)} } \leq \PAR{\frac 1 N  \sum_{\pi} |p(\pi)| L(\pi) } \| A_\star \|^\ell.
$$
We are then able to get an estimate of the form \eqref{eq:Ostep0}. This final estimate depends on the number of equivalence classes with a given upper bound on $L(\pi)$ and $p(\pi)$. In all these bounds, it is very useful to see the elements $\gamma = (g,x)$ as paths of length $m$ on the graph spanned by the vertices $(x_0,\ldots,x_m)$ and where the edges are decorated by the generators of the free group.

The proof of \eqref{eq:Osumagamma} differs very significantly from previous related lemmas established in \cite{MR4024563,BC2}. The major difference is that in these references, 
the quantity to be controlled is
$$
\sum_{\gamma \in \pi} \NRM{ a(\ell,g)}.
$$
With the norm inside the sum, this last expression can be much larger than the left-hand side of \eqref{eq:Osumagamma}. For example, if $\cA_1 = M_n(\dC)$, there could be a factor as worth as $\sqrt n$ between the two expressions. This turns out to be crucial. By keeping the norm outside the sum in \eqref{eq:Osumagamma}, we can derive a sharp quantitative upper bound valid on any $C^*$-algebra $\cA_1$.

\paragraph{Step 3: a non-commutative Cauchy-Schwarz inequality. } At the core of the proof of \eqref{eq:Osumagamma} there is a non-commutative inequality proven in Section \ref{sec:NCCS}. This inequality is a non-commutative version of the Cauchy-Schwarz inequality written as: 
\begin{equation}\label{eq:OCS}
\ABS{ \sum_{\vec j = (j_1,\ldots,j_k) \in \INT{n}^k } x_{\vec j} \prod_{i=1}^k y_{i,j_i} z_{i,j_i}   } \leq \max_{\vec j} |x_{\vec j} |  \cdot \prod_{i=1}^k \sqrt{ \sum_{j} | y_{i,j}|^2  \cdot  \sum_{j} | z_{i,j}|^2  }, 
\end{equation}
with $x_{\vec j},y_{i,j},z_{i,j}$ in $\dC$. In Section \ref{sec:NCCS}, the variables $y_{i,j}$, $z_{i,j}$ are now in an arbitrary $C^*$-algebra $\cA$, the variables $x_{\vec j}$ are replaced by a finite collection $(x_{1,\vec j},\ldots,x_{l,\vec j})$ of elements in $\cA$, we take a product of all these elements according to some ordering, and finally we replace $\ABS{\cdot}$ by a norm. Under an assumption on the variables $x_{i,\vec j}$ and their positions in the product, we prove in Theorem \ref{th:NCCS}  that the analog of \eqref{eq:OCS} holds. In Subsection \ref{subsec:NCCStheta}, we identify situations where a form of \eqref{eq:OCS} holds irrespectively of the assumption on the variables $x_{i,\vec j}$. This improvement will ultimately lead to the improvement given by Theorem \ref{th:main2} on Theorem \ref{th:main1} when the $a_i$'s have special structures.

Interestingly, this non-commutative Cauchy-Schwarz inequality implies the non-commutative Kintchine inequality due to Lust-Piquard and Pisier \cite{MR1150376} with the optimal constant found by Buchholz \cite{MR1812816}.

\subsection{Overview of Theorem \ref{th:main3}}

The proof of Theorem \ref{th:main3} follows 
the same strategy as the proof of Theorem \ref{th:main1}. There is however the well-known caveat that $\dE \| A_N \Pi_N \|_p^p$ cannot be $(1+o(1))^p$ close to $\| A_\star  \|^p$ for $p \gg \log N$ even when $\cA_1 = \dC$ (at least for non-trivial choices of the $a_i$'s).   This obstruction is due to what Friedman \cite{MR2437174} called {\em tangles}. We refer to \cite[Subsection 4.1]{MR4024563} for a more detailed discussion. To overcome this difficulty, we follow a strategy similar to the one developed among others, in
\cite{M13,BLM,MR3792625,MR4024563} and then combine it to the proof of Theorem \ref{th:main3}.

Let us describe the steps of proof more precisely. We remove the dependency in $N$ and set $\Pi= \Pi_N$. 

\paragraph{Step 1: Powers of non-backtracking decompositions. } We may assume that the symmetry condition \eqref{eq:symai} holds.  We  compute an estimate on $\|A \Pi \|_p$ by composing $p/\ell$ steps of $\ell$:  using \eqref{eq:-1decompA} and $\Pi A \Pi = A\Pi = \Pi A$, we get for integers $\ell,q \geq 1$, $q$ even, 
$$
\| A \Pi\|_{\ell q} ^\ell = \| A^\ell \Pi\|_{ q} ^\ell = \NRM{ \sum_{m=0}^{\ell}   B^{(\ell,m)} \Pi }_q \leq \sum_{m=0}^{\ell}  \| B^{(\ell,m)} \Pi \|_q  .
$$
The parameter $\ell$ will be of order $\log_{2d -1} N$ and $q$  as large as $\log N / (\log d \log \log N)$. The goal is then to upper-bound individually each $\| B^{(\ell,m)} \Pi \|_q$. 
In contrast to the proof of Theorem \ref{th:main1}, we consider the powers of the non-backtracking operators here. 

\paragraph{Step 2: remove the tangles. } In Subsection \ref{subsec:PDP}, the next step is to identify an operator $\widetilde B^{(\ell,m)}$ which coincides with $B^{(\ell,m)}$ on an event $E$ with high probability and such that $\widetilde B^{(\ell,m)}$ when written as a sum over non-backtracking paths of length $m$ cancels on paths which visits more than one cycle (tangle-free paths). To enjoy such property $m \leq \ell$ cannot  be of size larger than $O (\log_{2d -1} (N) ) $. 

\paragraph{Step 3: orthogonal projection. } We then compute an expression for $\widetilde B^{(\ell,m)} \Pi$ in terms of the centered matrices 
$$\underline U_i = U_i - \dE U_i = U_i - \IND_N \otimes \IND_N.$$ 
This is straightforward for $ B^{(\ell,m)} \Pi$, but since we have removed the tangled paths in $\widetilde B^{(\ell,m)}$, this is now more delicate to do it efficiently. Nevertheless, this step proves that $\widetilde B^{(\ell,m)} \Pi$ is close to the operator that we would have obtained from $ B^{(\ell,m)} \Pi$ after removing the tangled paths (even though the formal operations ``removing the tangles" and ``projecting by $\Pi$" do not commute). This is achieved with Lemma \ref{le:decomp} in Subsection \ref{subsec:PDP}. A similar step was already done in \cite{M13,BLM,MR3792625,MR4024563}. 

\paragraph{Step 4: expected high trace methods. } Up to an extra control of some remainder terms, we are now
in a position to apply the strategy of proof developed for Theorem \ref{th:main1} and compute an upper bound for $\dE [ \| \widetilde B^{(\ell,m)} \Pi \|_q ^q]$. The necessary estimates on random permutations were already available from \cite{MR3792625,MR4024563}.

\section{Non-backtracking identities}
\label{sec:nonback}

In this section, we consider the general setting for operators of the form \eqref{eq:defA*}-\eqref{eq:defA}. Let $\cA_1$ and $\cA_2$ be two unital $C^*$-algebras of bounded operators on the Hilbert spaces $H_1$ and $H_2$. We define the operator in the minimal tensor product $\cA_1 \otimes \cA_2$,
\begin{equation}\label{eq:defAH}
A = a_0 \otimes 1_{H_2}  + \sum_{i=1}^{2d} a_i \otimes u_i,
\end{equation}
with $u_i \in \cA_2$ unitaries such that for all $i \in \INT{2d}$, $u_{i^*} = u_i^*$ and $a_i \in \cA_1$. Under the condition  \eqref{eq:symai}, the operator $A$ is self-adjoint.

We  prove a general identity that can be rephrased, in some cases, in an analytic identity (see Theorem  \ref{th:resIB} ).

We start by constructing an homomorphism $u$ between two $*$-algebras.
Consider the algebraic products $\cM = \cA_1 \odot \cA_2 \odot M_{2d} (\dC)$ and $\cM_\star = \cA_1 \odot \cA_\star \odot M_{2d} (\dC)$. 
To clarify, the algebraic tensor products are a $*$-subalgebra of the respective min
tensor products $\cA_1 \otimes \cA_2 \otimes M_{2d} (\dC)$ and $\cA_1 \otimes \cA_\star \otimes M_{2d} (\dC)$, consisting of finite rank tensors. They are dense. 
Identifying $H_1 \otimes H_2$ as a subspace of $H_1 \otimes H_2 \otimes \dC^{2d}$, we have $A , B, P \in \cM$. 
The unitaries $(u_1, \ldots, u_d)$ define uniquely a group  homomorphism $u $ from $\dF_d$ onto the unitaries of $\cA_2$, by setting for $u(\o) = 1_{H_2}$ and for $g \in \dF_d$, $g \ne \o$, 
\begin{equation}\label{eq-def-u}
u(g)= u_{i_k} \cdots u_{i_1},  
\end{equation}
if $g = g_{i_k} \cdots g_{i_1}$ in reduced form. The homomorphism $u$ extends to $\cM_\star \to \cM$ by setting  $u (a \otimes \lambda (g) \otimes b ) = a \otimes u(g) \otimes b$. We have by construction: 
$$
u(A_\star) = A,  \quad u (B_\star(z)) = B(z) \quad \hbox{and} \quad u(P_\star) = P ,
$$
where $B_\star(z), P_\star \in \cM_\star$ are the operators $B(z),P$ when $H = \ell^{2} (\dF_d)$ and $u_i = \lambda(g_i)$. 

For integers $0 \leq m \leq k$, we consider the operator in $\cA_1 \otimes \cA_2$,
\begin{equation}\label{eq:defBkmbis}
B^{(k,m)} = \sum_{g \in S_m} (A^k_\star)_{g\o} \otimes u(g),
\end{equation}
where $S_m \subset \dF_d$ is the subset of elements written in reduced form as a word of length $m$ (for $m=0$, $S_0 = \{\o\}$) and $u(g)$ is the homomorphism defined in Equation \eqref{eq-def-u}
We have the decomposition:

\begin{lemma}\label{th:powerIB}
For all integers $k \geq 0$, we have 
$$
A^k =  \sum_{m=0}^k B^{(k,m)}.
$$
\end{lemma}
\begin{proof}
By construction, 
$$
A_\star ^k = \sum_{g \in \dF_d} (A^k_\star)_{g\o} \otimes \lambda(g) = \sum_{m=0}^k \sum_{g \in S_m} (A^k_\star)_{g\o} \otimes \lambda(g) = \sum_{m=0}^k B^{(k,m)}_\star, 
$$
where $B^{(k,m)}_\star$ is defined by \eqref{eq:defBkmbis} with $u(g) = \lambda(g)$ (that is $\cA_2 = \cA_\star$).
Next, using the homomorphism $u$ from $\dF_d$ onto the unitaries of $\cA_2$,  defined in Equation \eqref{eq-def-u}
extended to $u : \cA_1 \odot \cA_\star \to \cA_1 \odot \cA_2$, we have $ A^k  = u (A^k_\star))$ and $B^{(k,m)} = u (B^{(k,m)}_\star)$. The conclusion follows. 
\end{proof}

We now introduce a companion operator in $\cA_1 \otimes \cA_\star$ which will play a significant role in the sequel. For integer $k\geq 1$, the set $W_k$ of words of length $k$ is the set of $w = (w_1,\ldots,w_k) \in \{0,\ldots, 2d\}^k$. Set $g_0 = \o$. If $w \in W_k$ and $t \in \INT{k}$, we set $(w)_t = (w_1,\ldots,w_t)$. We also define $a(w) = a_{w_k} \cdots a_{w_1}$, $[w] = g_{w_k} \cdots g_{w_1}$ and $[w]_t = [(w)_t]$. With these notation, we may write
\begin{equation}\label{eq:Akgo0}
A_\star^{k} = \sum_{w \in W_k} a(w) \otimes \lambda([w]).
\end{equation}
For $i \in \INT{2d}$, we introduce the operator in $\cA_1 \otimes \cA_{\star} \subset \cB (H_1 \otimes \ell^2(\dF_d))$ defined by 
\begin{equation}\label{eq:defCki}
C_\star^{(k,i)} =  \sum_{w \in W_k}  \IND (  \hbox{for all $t\in \INT{k}$, $[w]_t \ne \o$} , w_1 = i) \,  a(w) \otimes  \lambda ( [w]).
\end{equation}
In other words, for all $g,h \in \dF_d$, 
$$
(C_\star^{(k,i)})_{gh} = \sum_{w \in W_k} \IND ( \hbox{for all $t\in \INT{k}$, $[w]_t \ne \o$}, w_1 = i)  \IND ( g = [w] h)\, a (w).
$$ 

We also define $A_\star^{\o}$ as the restriction of the operator $A_\star$ to $H_1 \otimes \ell^2 (V)$ where $V =  \dF_d\backslash \{\o\}$. We shall need the following decomposition lemma. Let
$V_i \subset \dF_d$  be the set of $g \in \dF_d$, $g \ne \o$ which are written in reduced form as $g = g_{i_m}\cdots g_{i_1}$ with $i_1 = i$.  

\begin{lemma}\label{le:pathdecomFd}
Let $m \geq 1$ and $g \in S_m \subset \dF_d$ which is written as a reduced word as $g = h_r  \cdots h_1 h_0$ with  $h_j \in S_{m_j}\cap V_{i_j}$, $m_j \geq 1$ and $m = \sum_j m_j$. Then, for any integer $k \geq 0$,
$$
(A_\star^k)_{g \o} = \sum_{ (k_0,\cdots ,k_r)}  (C_\star^{(k_r,i_r)})_{h_r \o } \cdots (C^{(k_1,i_1)}_\star)_{h_1 \o} (A_\star)^{k_0}_{h_0 \o},
$$
where the sum is over all $(k_0,\ldots,k_r)$ with $\sum k_j  = k$, $k_j \geq 1$ for $j \geq 1$. Moreover, if $g \in V_i \subset \dF_d$, $i \in \INT{2d}$, we have for any $ k \geq 1$,
\begin{equation*}
(C_\star^{(k,i)})_{g \o}   = ((A^{\o}_\star)^{k-1} )_{g g_i} a_i . 
\end{equation*}
\end{lemma}
\begin{proof}
We start with the second statement. Since $A_\star^{\o}$ is the restriction of the operator $A_\star$ to $H_1 \otimes \ell^2 (V)$ where $V =  \dF_d\backslash \{\o\}$, we have, for $g \in V_i$ and integer $k \geq 0$,
\begin{equation*}\label{eq:Bkmgo}
((A^{\o}_\star)^{k} )_{g g_i} = \sum_{w \in W_k^i(g)} a(w)
\end{equation*}
where the sum is over all sequences $w \in W_k$  such that $g_i, [w]_1 g_i, \ldots  , [w]_k g_i$ is a walk in $V_i$ from $g_i$ to $g$.  On the other hand, for $k \geq 1$, a walk $[w]_0,\ldots, [w]_k$ with $w = (w_1,\ldots,w_k) \in W_k$ from $\o$ to $g \in V_i$ such that $[w]_t \ne \o$ for $t \geq 1$ is characterized by $w_1 = i$ and $w = (w_2,\ldots,w_k) \in W_{k-1}^i (g)$.  The identity $(C_\star^{(k,i)})_{g \o}   = ((A^{\o}_\star)^{k-1} )_{g g_i} a_i $ follows. 

Similarly, the first identity is proved by decomposing a path of length $k$ from $\o$ to $g$ along its last passage times at $h_0$, $h_1 h_0$, \ldots \,. Assuming $h_t = V_{i_t}$ for $t \in\INT{r}$, such decomposition is made of a path of length $k_0$ from $\o$ to $h_0$, followed $g_{i_1}$ and a path in $V_{i_1}$ of length $k_1-1$ from $g_{i_1}$ to $h_1g_{i_1}$ and so on. 
\end{proof}

We conclude this section with a bound on the norm of $C_\star^{(k,i)}$ in terms of the norm $ A_\star$ in the self-adjoint case.

\begin{lemma}\label{le:normCk}
Assume that the symmetry condition \eqref{eq:symai} holds. Then, for any integer $k \geq 1$, $i \in \INT{2d}$ and $g \in \dF_d$, we have
$$
\| (C_\star ^{(k,i)})_{g \o} \| \leq \|  A_\star\|^{k}\AND  \| ((C_\star ^{(k,i)})^* C_\star^{(k,i)})_{\o \o} \|  \leq \|A_\star\|^{2k}.
$$
Moreover,
$$
\| C_\star ^{(k,i)} \| \leq k \|  A_\star\|^{k}.
$$
\end{lemma}
\begin{proof}
We remove the subscript $\star$ for ease of notation. We start from the first inequality. We may assume $g \in V_i$ otherwise this is trivial since $C ^{(k,i)}_{g \o}  = 0$. By Lemma \ref{le:pathdecomFd}, since the operator norm is sub-multiplicative, we have for $g \in V_i$,  
$$
\| C  ^{(k,i)}_{g \o} \| \leq \| ((A^{\o})^{k-1} )_{g g_i} \| \| a_i \| \leq \| A^{\o} \|^{k-1} \| a_i \| . 
$$
Next, since $A$ and $A^{\o}$ are self-adjoint and $A^{\o}$ is the restriction of $A$ to $\ell^2(\dF_d \backslash \{\o\})$, we also have $\| A^{\o}\| \leq \| A\|$. The bound $\| C ^{(k,i)}_{g \o} \| \leq \|  A\|^{k}$ follows. 

Similarly, from Lemma \ref{le:pathdecomFd},
$$
 ((C ^{(k,i)})^* C^{(k,i)})_{\o \o} = \sum_{g \in F_d}  ((C ^{(k,i)})^*)_{\o g} C^{(k,i)}_{g \o} =  \sum_{g \in V_i} a_i^* ((A^{\o})^{k-1} )^*_{g g_i} ((A^{\o})^{k-1} )_{g g_i} a_i ,
$$
where we have use we have use that $(T^*)_{\o g} = (T_{g \o } )^*$ for $T \in \cB ( H_1 \otimes \ell^2 (\dF_d))$. 
Hence, 
\begin{eqnarray*}
\| ((C ^{(k,i)})^* C^{(k,i)})_{\o \o} \| & \leq &  \| a_i \|^2  \| \sum_{g \in V_i} ((A^{\o})^{k-1} )^*_{g g_i} ((A^{\o})^{k-1} )_{g g_i} \| \\
&  =  &   \| a_i \|^2  \| (((A^{\o})^{k-1} )^* (A^{\o})^{k-1} )_{g_i g_i} \| \\
& \leq &   \| a_i \|^2   \| A^{\o} \|^{2 (k-1)}.
\end{eqnarray*}
Using again $\| A^{\o}\| \leq \| A\|$, we obtain the second inequality. 
For the last inequality, we find the recursion, for any integer $k \geq 1$,
$$
A C^{(k,i)} = C^{(k +1,i)} +  \sum_{w \in W_{k-1}} \IND( [w] = \o)  \IND ( \hbox{for all } t\in \INT{k-1}, [w]_t g_i \ne \o ) \, a_i^* a(w) a_i \otimes 1_{\dF_d}.
$$
Indeed, if $w \in W_{k-1}$,  $g_j [w] g_i = \o$ and for all for all $ t\in \INT{k-1}, [w]_t g_i \ne \o$ then $j = i^*$ and thus $[w] = \o$.  We get 
$$
A C^{(k,i)} = C^{(k+1,i)} +  a^*_i C  ^{(k,i)}_{g_i \o} \otimes 1_{\dF_d}.
$$
Similarly, for $k =1$, we have $C^{(1,i)} = a_i \otimes \lambda(g_i)$. We deduce that $\| C^{(1,i)} \| = \| a_i \|  \leq  \| A \|$ and
\begin{eqnarray*}
\|C^{(k+1,i)} \| & \leq  & \|A \| \| C^{(k,i)} \| + \|  a_i^* C  ^{(k,i)}_{g_i \o}\| \\
& \leq & \|A \| \|C^{(k,i)}\| + \|A \|  \|C  ^{(k,i)}_{g_i \o} \|.
\end{eqnarray*}
Now, from first inequality, $\|C  ^{(k,i)}_{g_i \o} \|   \leq \| A \|^{k}$. This implies by recursion that $\| C^{(k,i)} \| \leq  k \| A \|^k$ as requested (we note that this bound improves on Haagerup's inequality \cite{MR520930}).
\end{proof}

\section{A non-commutative Cauchy-Schwarz inequality}

\label{sec:NCCS}

\subsection{Main inequality}

Let $r \geq k \geq 1$ be integers and $\pi$ be a partition of a subset $S \subseteq \INT{r}$, and $T = \INT{r} \backslash S$. We assume that $\pi$ partitions $S$ into $k$ blocks $\{B_1,\ldots,B_k\}$ and each block is either a singleton or a pair of elements in $S$. For $i \in S$, let $l_i$ be the index of the block containing $i$, that is $i \in B_{l		_i}$. 

Let $m\geq 1$ be an integer and $X= ( (X_{i,\vec j})_{i\in \INT{r},\vec j\in \INT{m}^k})$ be a collection of elements in a $C^*$-algebra $\cA$. We assume that for $i \in S$ and $\vec j = (j_1,\ldots,j_k)$, $X_{i,\vec j}$ depends only on $(i,j_{l_i})$ (more precisely, for $i \in S$, $X_{i,\vec j} = X_{i,\vec j'}$ if $j_{l_i} = j'_{l_i}$). With a slight abuse of notation, for $i \in S$ and $j \in \INT{m}$, we will sometimes write $X_{i,j}$ in place of $X_{i,\vec j}$ where $\vec j \in \INT{m}^k$ is any element such that $j_{l_i} = j$. 

With these notations, we define
$$
X_{\pi}=\sum_{\vec j  \in \INT{m}^k} X_{1,\vec j}\cdots X_{r,\vec j}.
$$

For example, if $r = 8$, $k=3$ and $\pi = (\{1,5\},\{2,7\},\{4,8\})$, then
\begin{equation}\label{eq:exNCCS}
X_{\pi} = \sum_{1 \leq j_1,j_2 ,j_3 \leq m} X_{1,j_1}  X_{2,j_2} X_{3,(j_1,j_2,j_3)} X_{4,j_3} X_{5,j_1}  X_{6,(j_1,j_2,j_3)}X_{7,j_2}   X_{8,j_3}. 
\end{equation}

Our goal is to estimate the norm of $X_{\pi}$. Our bound will require a compatibility assumption between $X$ and the partition $\pi$, that we first describe.  For $i \in T$ and $l \in \INT{k}$, we say that the block $B_l$ is {\em open} at $i$ if $\min (i' \in B_l) < i < \max (i' \in B_l)$ (thinking of a block $B_l$ as the extremal points of a discrete interval of $\INT{r}$, $B_l$ is open at $i$, if $i$ is inside this interval). We say $X$ is {\em open} for $\pi$ if for all $i \in T$, the map $\vec j  \to X_{i,\vec j}$ depends only on the blocks open at $i$ (that is, if $\vec j$ and $\vec j'$ are equal on the coordinates $j_l$ with $B_l$ open at $i$ then $X_{i,\vec j} = X_{i,\vec j'}$).  For example, in \eqref{eq:exNCCS}, $X$ will be open for $\pi$, if $X_{3,(j_1,j_2,j_3)}$ depends only on $(j_1,j_2)$ and if $ X_{6,(j_1,j_2,j_3)}$ depends only on $(j_2,j_3)$. 

To estimate the norm of $X_{\pi}$, we introduce a last notation. 
For $i\in \INT{r}$, we define $Q_{\pi,i} \in \dR_+$ as follows:
\begin{enumerate}
\item If $i \in T $,
$$Q_{\pi,i}=\max_{\vec j\in \INT{m}^k} \NRM{ X_{i,\vec j}}.$$
\item If $i \in S$ is the leftmost element of a block with two elements, 
$$Q_{\pi,i}=\NRM{ \sum_{j\in \INT{m}} X_{i,j}X_{i,j}^*}^{1/2}.$$
\item If $i \in S$ is the rightmost element of a block with two elements, 
$$Q_{\pi,i}=\NRM{ \sum_{j\in \INT{m}} X_{i,j}^*X_{i,j}}^{1/2}.$$

\item If $i \in S$ is an element of a singleton block, then 
$$Q_{\pi,i}= \NRM{\sum_{j\in \INT{m}} \sqrt{X_{i,j}X_{i,j}^*} }^{1/2} \cdot\NRM{\sum_{j\in \INT{m}} \sqrt{X_{i,j}^*X_{i,j}} }^{1/2}.$$
\end{enumerate}

\begin{theorem}\label{th:NCCS}
If $X$ is open for $\pi$ then we have 
$$
\NRM{X_{\pi}} \le \prod_{i\in \INT{r}} Q_{\pi,i}.
$$
\end{theorem}

\begin{proof}
We first treat the case where $\pi$ has no singletons.
As usual, for $j,j' \in \INT{m}$, $E_{jj'} \in M_m(\dC)$ denotes the matrix with all entries equal to $0$ except entry $(j,j')$ set to $1$. In $M_m(\mathbb{C})^{\otimes k}\otimes \cA$, we introduce the  elements
$\tilde{X}_i$ as follows
\begin{enumerate}
\item If $i \in S$ is the leftmost element of its block, 
$$\tilde{X}_i=\sum_{j\in \INT{m}} 1_m^{\otimes l_i-1}\otimes E_{1j}\otimes 1_m^{\otimes k-l_i}\otimes X_{i,j}.$$
\item If $i \in S$ is the rightmost element of its block, 
$$\tilde{X}_i=\sum_{j\in \INT{m}} 1_m^{\otimes l_i-1}\otimes E_{j1}\otimes 1_m^{\otimes k-l_i}\otimes X_{i,j}.$$
\item If $i\in T$ and $L_i \subset \INT{k}$ are the indices of blocks open at $i$,
$$\tilde{X}_i=\sum_{\vec  j\in \INT{m}^{L_i}} K_{p_{1,i}(\vec j)} \otimes \cdots \otimes K_{p_{r,i}(\vec j)} \otimes X_{i,\vec j},$$
where $K_{p_{l,i} (\vec j)} \in \{ 1_m, E_{j_l j_l} \}$ is $E_{j_l j_l}$ if $l \in L_i$ and is $1_m$ otherwise. Also, $X_{i,\vec j}$ is a slight abuse of notation (since $X_{i,(j_1,\ldots,j_k)}$ depends only on $(j_{l})_{l \in L_i}$).
\end{enumerate}

The key computation is the following simple verification: 
\begin{equation}\label{eq:trick}
E_{11}^{\otimes k}\otimes X_{\pi}=\prod_{i\in \INT{r}}\tilde{X}_i.
\end{equation}

Let us check that this formula is correct in the example \eqref{eq:exNCCS} assuming that $X$ is open for $\pi$: for shorter notation, we simply write $X_{3,(j_1,j_2,j_3)} = X_{3,(j_1,j_2)}$ and $X_{6,(j_1,j_2,j_3)} = X_{6,(j_2,j_3)}$. We have 
{\small \begin{align*}
\prod_{i\in \INT{r}}\tilde{X}_i  = & \PAR{ \sum_{j_1} E_{1j_1} \otimes 1 \otimes 1 \otimes X_{1,j_1}   }\PAR{ \sum_{j_2} 1 \otimes E_{1 j_2} \otimes 1 \otimes X_{2,j_2}  } \PAR{ \sum_{j_3,j_4} E_{j_3 j_3} \otimes E_{j_4 j_4} \otimes 1 \otimes X_{3,(j_3,j_4)}  } \\
&  \cdot  \PAR{ \sum_{j_5} 1  \otimes 1 \otimes E_{1 j_5} \otimes X_{4,j_5}  }  \PAR{ \sum_{j_6} E_{j_61} \otimes 1 \otimes 1  \otimes X_{5,j_6}  }\PAR{ \sum_{j_7,j_8} 1 \otimes E_{j_7 j_7} \otimes E_{j_8 j_8} \otimes X_{6,(j_7,j_8)}  } \\
& \cdot  \PAR{ \sum_{j_9} 1  \otimes 1 \otimes E_{1 j_9}\otimes 1 \otimes X_{7,j_{9}}  }  \PAR{ \sum_{j_{10}} 1 \otimes 1 \otimes E_{j_{10} 1} X_{8,j_{10}}  }.
\end{align*}}
Expanding the products, this forces $j_1 = j_3 = j_6$, $j_2 = j_4 = j_7 = j_9$ and $j_5 = j_8 = j_{10}$. We obtain the claimed formula \eqref{eq:trick}.

Next, the conclusion follows from the sub-multiplicativity of the operator norm on $M_m(\mathbb{C})^{\otimes k}\otimes \cA$.
Indeed:

(1) Since $E_{11}^{\otimes k}$ is a selfadjoint projection, it is of operator norm $1$ and it follows that 
$$\|E_{11}^{\otimes k}\otimes X_{\pi}\|=\| X_{\pi}\|.$$

(2) It follows from the $C^*$ norm axiom that $$\|\tilde{X}_i\|=Q_{\pi,i}.$$

This concludes the proof when  $\pi$ has no singletons. In the general case, we can handle singletons by doubling them. 
For $i \in S$ in a singleton block, we write a polar decomposition $X_{ij}=H_{ij}U_{ij}$ where $H_{ij}$ positive, $U_{ij}$ unitary.
Then we write $X_{ij}=(\sqrt{H_{ij}}) (\sqrt{H_{ij}}U_{ij})$ and we are back to the case without singletons by increasing by the value of $r$ by the number of singletons of $\pi$. It remains to observe that $H_{ij} H_{ij}^* = H_{ij}^2 = \sqrt{X_{ij}X_{ij}^*} $ and $(\sqrt{H_{ij}}U_{ij})^* (\sqrt{H_{ij}}U_{ij}) = \sqrt{X_{ij}^*X_{ij}}$. \end{proof}

\begin{remark}\label{rk:NCCS}To apply Theorem \ref{th:NCCS} in concrete applications, we observe that it is possible to consider the following (seemingly) more general setting: the vectors $\vec j$ take values in $J = J_1 \times \cdots \times J_k$ where for $l \in \INT{k}$, $J_l $ is a finite set in place of $\INT{m}^k$: 
$$
X_\pi = \sum_{\vec j  \in J} X_{1,\vec j}\cdots X_{r,\vec j}.
$$
The definitions $X$ open for $\pi$ and $Q_{\pi,i}$ remain unchanged up to the modification of the index set $\INT{m}$ by $J_l$. In this setting, the conclusion of Theorem \ref{th:NCCS} still holds. Indeed, we may assume without loss of generality that $J_l = \INT{m_l}$. We set $m = \max_l m_l$ and define the elements in $\cA$, for $i \in \INT{k}$ and $\vec j \in \INT{m}^k$, $X'_{i,\vec j} = X_{i,\vec j}$  if $\vec j \in J$ and is $0$ otherwise. We then have $X'_\pi = X_{\pi}$.
\end{remark}

\subsection{Stronger forms of Theorem \ref{th:NCCS} }
\label{subsec:NCCStheta}
We do not know whether the conclusion of Theorem \ref{th:NCCS} remains true without the assumption that $X$ is open for $\pi$ (we are not aware of a counterexample). In this subsection, we will explore a few cases where a potential improvement of this inequality can be proved. To that end, we introduce two definitions. 

\begin{definition}\label{def:Deltacontrol}
Let $\theta \geq 1$. 
We say $(X,\pi)$ is $\theta$-controlled, if for any $\delta = (\delta_{\vec j}) \in [-1,1]^{\INT{m}^k}$
$$
\NRM{\sum_{\vec j  \in \INT{m}^k} \delta_{\vec j} X_{1,\vec j}\cdots X_{r,\vec j}}\leq \theta^k\prod_{i=1}^r Q_{\pi,i}.
$$
We say $(X,\pi)$ is strongly $\theta$-controlled on the left (respectively on the right) if 
$$
\NRM{ \sum_{\vec j  \in \INT{m}^k} |X_{1,\vec j}\cdots X_{r,\vec j}| } \leq \theta^k \prod_{i=1}^r Q_{\pi,i},
$$
where $|  A |  = \sqrt{A A^*}$ (respectively $|  A |  = \sqrt{A^* A}$).
\end{definition} 
Note that as its name suggests, strong $\theta$-control implies $\theta$-control. In the remainder of this section, we give necessary conditions for $(X,\pi)$ to be strongly $\theta$-controlled

\begin{lemma}\label{le:NCCSuni}
Let $n \geq 1$ be an integer and assume that $\cA = M_n(\dC) \otimes \cA_0$ where $\cA_0$ is a $C^*$-algebra. We assume that for all $i \in S$, $X_{i,j} = x_{i,j} \otimes u_{i,j}$ where $x_{i,j} \in M_n(\dC)$ and $u_{i,j} \in \cA_0$ is unitary. Then $(X,\pi)$ is strongly $n$-controlled on the left and on the right.
\end{lemma}

\begin{proof}
From the triangle inequality for the operator norm and its sub-multiplicativity, we get 
$$
\| X_{\pi} \| \leq \sum_{\vec j} \|X_{1,\vec j}\| \cdots \| X_{r,\vec j} \|.
$$
We deduce from the classical Cauchy-Schwarz inequality that 
$$
\| X_{\pi} \| \leq \prod_{i=1}^r \tilde Q_{\pi,i},
$$
where for $i\in T$, $\tilde Q_{\pi,i} =  Q_{\pi,i}$ and for $i \in S$ in a block of two elements,
$$
\tilde Q_{\pi,i} =  \PAR{ \sum_{j\in \INT{m}} \| X_{i,j} \|^2 }^{1/2},
$$ 
and for for $i \in S$ in a singleton block,
$$
\tilde Q_{\pi,i} =  \sum_{j\in \INT{m}} \| X_{i,j} \| ,
$$

Then, the lemma follows from the claim $\tilde Q_{\pi,i} \leq \sqrt n  Q_{\pi,i} $  for $i$ in a block of two elements and $\tilde Q_{\pi,i} \leq n  Q_{\pi,i} $ for $i$ in a singleton block (recall that $k$ is the number of blocks of $\pi$). 

We check the claim. From our assumption, $\|X_{i,j}\| =\|x_{i,j}\| \leq \sqrt{\Tr (x_{i, j} x_{i,j}^* })$. From the linearity of the trace, it follows that, for $i$ in a block of two elements, 
$$
\tilde Q_{\pi,i}  \leq \PAR{\Tr   \sum_j x_{i,j}^* x_{i,j} }^{1/2}  \leq \sqrt n \NRM{  \sum_j x_{i,j}^* x_{i,j} }^{1/2} = \sqrt n \NRM{  \sum_j X_{i,j}^* X_{i,j} }^{1/2},
$$
where the last inequality follows from the fact that $u^*_{ij} u_{ij} = 1$ and $\| T \otimes 1 \| = \| T \|$. The same bound holds with $X_{i,j}^*X_{i,j}$. We deduce that $\tilde Q_{\pi,i} \leq \sqrt n Q_{\pi,i}$, as claimed. For a singleton block, the proof is identical.
\end{proof}

The bistochastic matrices give another simple example of matrices that can be $\theta$-controlled.

\begin{lemma}
\label{le:discr4}
Assume that $\cA = M_n(\dC)$ and that there exists  a unit vector $f \in \dC^n$ such that for all $i \in S$ and $j \in \INT{m}$,  $X_{i,j} f = X_{i,j}^* f = \| X_{i,j} \| f$.  Then $(X,\pi)$ is strongly $1$-controlled on the left and on the right.
\end{lemma} 

\begin{proof}
With $\tilde Q_{\pi,i}$ as in the proof of Lemma  \ref{le:NCCSuni}, it is sufficient to check that $\tilde Q_{\pi,i} \leq Q_{\pi,i}$ (in fact there will be equal). Let us check the claim for $i$ in a block of two elements. Using the assumption, we have 
$$
\NRM{ \sum_j X_{i,j}^* X_{i,j} } \geq \NRM{ \PAR{ \sum_j X_{i,j}^* X_{i,j} } f }_2 = \tilde Q^2_{\pi,i},$$
and similarly with $X_{i,j} X_{i,j}^*$.
For a singleton block, the proof is identical.\end{proof}

\subsection{Relation to non-commutative Kintchine inequalities}

(This subsection is a digression independent of the rest of the article). Theorem \ref{th:NCCS} is connected to the non-commutative Kintchine inequalities.  Indeed, as a corollary of the proof of Theorem \ref{th:NCCS}, we obtain the following version of a classical non-commutative Kintchine-inequality due to Lust-Piquard and Pisier \cite{MR1150376} with the optimal constant found by Buchholz \cite{MR1812816}. 

\begin{corollary}
Let $(\gamma_1,\ldots,\gamma_m)$ be independent standard Gaussian variables and $(X_1,\ldots,X_m)$ be  elements in a $C^*$-algebra with faithful   state $\tau$. For any integer $k \geq 1$, we have
$$
\dE \tau \BRA{ \ABS{ \sum_j \gamma_j X_j } ^{2k}} \leq  \frac{(2k)! }{ 2^k k!}    \max\PAR{  \tau \BRA{ \PAR{ \sum_j X_j X_j^* }^k}, \tau \BRA{ \PAR{\sum_j X_j^*X_j }^k}},
$$
where $|x| = \sqrt{ x x^*}$ or $|x| = \sqrt{x^*x}$.
\end{corollary}

\begin{proof}
We write
\begin{eqnarray*}
\dE \tau   \ABS{ \sum_j \gamma_j X_j }^{2k} &  = &\dE  \tau  \PAR{ \PAR{ \sum_j \gamma_j X_j }\PAR{ \sum_j \gamma_j X^*_j } }^{k}  \\
& \leq & \sum_{(j_1,\ldots,j_{2k})} \tau \BRA{ X_{j_1} X_{j_2}^* \cdots  X_{j_{2k-1}}  X_{j_{2k}}^* }\dE \prod_{i=1}^{2k} \gamma_{j_i}.
\end{eqnarray*}
From Wick's formula, we get
\begin{eqnarray*}
\dE \tau   \ABS{ \sum_j \gamma_j X_j }^{2k}& = & \sum_{\pi} \tau \BRA{ X_{\pi} },
\end{eqnarray*}
where the sum is over all pair partitions $\pi$ of $\INT{2k}$ and $X_\pi$ is as in the previous subsections with $r =2k$, $X = (X_{1,\vec j},\ldots,X_{2k,\vec j})$  and $X_{2i-1,\vec j} = X_{j_{l_{2i-1}}}$, $X_{2i,\vec j} = X^*_{j_{l_{2i}}}$ where $l_i$ is the index of the block of $\pi$ where $i$ belongs. 

Since there $(2k)! / (2^k k!)$ pair partitions, from Equation \eqref{eq:trick} in Theorem \ref{th:NCCS}, we deduce that 
\begin{eqnarray*}
\dE \tau   \ABS{ \sum_j \gamma_j X_j }^{2k} & \leq &    \frac{(2k)! }{ (2^k k!)} \max_{\pi} \tilde \tau \{  \prod_{i\in \INT{2k}}\tilde{X}_i \},
\end{eqnarray*}
where $\tilde \tau  = \Tr_{M_{m}(\dC)^k} \otimes \tau$ and $\tilde X_i$ was defined above \eqref{eq:trick} with $X_{i,j} \in \{X_j ,X_j^*\}$ depending on the parity of $i$. Observe that 
$$
 \tilde X_i  \tilde X_i^*  =  \sum_{j\in \INT{m}} 1_m^{\otimes l_i-1}\otimes E_{11} \otimes 1_m^{\otimes k-l_i}\otimes X_{i,j}  X_{i,j}^*.
$$
Hence $$\tilde \tau \{ ( \tilde X_i  \tilde X_i^* )^k \} = \tau \{  ( \sum_j X_{i,j} X_{i,j}^*)^k \} \leq \max\PAR{  \tau \{ ( \sum_j X_j X_j^* )^k\}, \tau \{ (\sum_j X_j^*X_j )^k\}} .$$
We then apply H\"older's inequality:
$$
 \tilde \tau \{  \prod_{i\in \INT{2k}}\tilde{X}_i \}\leq \prod_{i \in \INT{2k}}  \PAR{  \tilde \tau \{  (\tilde{X}_i\tilde{X}_i^*)^k \} }^{1/(2k)}.
$$
The conclusion follows.
 \end{proof}

\section{Expected high trace method}
\label{sec:FK}

\subsection{Main technical result: proof of Theorem \ref{th:main1}}

We now fix a sequence $(a_0,a_1, \ldots, a_{2d})$  in $\cA_1$ satisfying the symmetry condition \eqref{eq:symai}. Recall the definition of $A_\star$ in \eqref{eq:defA*} and let 
$$
\rho = \| A_\star\|. 
$$

We fix an integer $N \geq 1$ and consider $(U_1,\ldots, U_d)$ be independent Haar distributed unitary matrices in $\dU_N$ or $\dO_N$ and set $U_{i+d}  = U_{i^*} = U_i^*$ for $i \in \INT{d}$. We then consider the operator $A \in \cA_1 \otimes M_N(\dC)$ as in \eqref{eq:defA} and the corresponding operators $B^{(k,m)}$ defined by \eqref{eq:defBkmbis} (we omit the dependency in $N$). We denote by $$\tau = \tau_1 \otimes \tr_N$$ the trace on $\cA_1 \otimes M_N(\dC)$ with $\tr_N = \frac 1 N \tr$.

We fix an even integer $\ell \geq 1$. From Lemma \ref{th:powerIB}, we have 
\begin{equation*}
 \| A \|_{\ell}^{\ell}  =   \tau \PAR{ A^{\ell}}  = \sum_{m=0}^{\ell} \tau \PAR{ B^{(\ell,m)} }.
\end{equation*}
For $m=0$, from \eqref{eq:defBkmbis}, we find
\begin{equation*}
\tau \PAR{ B^{(\ell,0)} }=  \tau_1 ( (A_\star^{\ell})_{\o \o} )  = \| A_\star\|_\ell^\ell.
\end{equation*}
Hence,
\begin{equation}\label{eq:normAltr}
\ABS{ \dE \SBRA{\| A \|_\ell^\ell} - \| A_\star \|_\ell^\ell }  \leq \sum_{m=1}^\ell \ABS { \tau \PAR{ B^{(\ell,m)} }}.
\end{equation}

From now on, we assume $1 \leq m \leq \ell$.  Our aim in this section is to upper bound the expected value $\dE [ \tau( B^{(\ell,m)} ) ]$ by $\rho^{\ell}$ up to an effective multiplicative factor as long as $\ell$ is not too large. More precisely, we will prove the following upper bound.

\begin{theorem}\label{th:FKBlm}
Assume that $(a_0,a_1,\ldots,a_{2d})$ satisfy  the symmetry condition \eqref{eq:symai}. Let $\ell_0$ be the largest integer $\ell$ such that 
$$
2 (2d)^{7/2} \ell^{3d+ 14} \leq N^{1/4}.
$$
Then, for all $1 \leq m \leq \ell \leq \ell_0$, for some numerical constant $c >0$, we have
$$
\ABS{ \dE \SBRA{\tau \PAR{ B^{(\ell,m)} } }} \leq  \frac{ c d^4  m^{4d +4} \ell^{11}}{N}  \rho^\ell.
$$
and, if $\ell$ is even,
$$
\ABS{ \dE \SBRA{\| A \|_\ell^\ell} - \| A_\star \|_\ell^\ell }  \leq   \frac{  c  d^4  \ell^{4 d + 16}}{ N} \rho^{\ell} \leq \frac{  c }{\sqrt N} \rho^{\ell}.
$$
\end{theorem}

Assuming Theorem \ref{th:FKBlm}, Theorem \ref{th:main1} is immediate. 

\begin{proof}[Proof of Theorem \ref{th:main1}] We omit the dependency in $N$. 
We first assume that the symmetry condition \eqref{eq:symai} is met. We use $\| A_\star\|_p \leq \| A_\star\| = \rho$ for any $p \geq 1$.  From the assumption $d^{70} \leq N$, we find easily that $\ell = 2 \lceil p /2 \rceil = 2 \lceil  N^{1/(16d + 80)} /2 \rceil \leq \ell_0$ where $\ell_0$ is as in Theorem \ref{th:FKBlm}. From Jensen's inequality and Theorem \ref{th:FKBlm}, we deduce that $\dE  \SBRA{\| A \|_\ell} \leq ( 1+ c / \sqrt{N} )^{1/\ell} \rho \leq ( 1 + c' / (\sqrt N \ell) ) \rho$. The conclusion follows when the symmetry condition \eqref{eq:symai} holds. 

In the general case, we may use the standard linearization trick. Let $\check \cA_1 = M_2(\dC) \otimes \cA_1$. We consider the operator $\check A \in \check \cA_1 \otimes M_N(\dC)$ defined by \eqref{eq:defA} with $a_i$ replaced by 
$$
\check a_i = \begin{pmatrix} 0 & a_i \\ a^*_{i^*} & 0 \end{pmatrix}.
$$
In other words 
$$\check A =    \begin{pmatrix} A & 0 \\  0  & A^*\end{pmatrix} \cdot  ( u \otimes 1), \; \hbox{ with }\,  u =  \begin{pmatrix} 0 & 1  \\ 1 & 0 \end{pmatrix}.$$
The operator $\check A_\star$ in  $ \check \cA_1 \otimes \cA_\star$ is defined similarly. Since $u$ is unitary and $\| A^* \|_p = \| A \|_p$, we get $\| \check A_\star \| = \| A_\star\| $ and  $\| \check A \|_p = \| A \|_p $ for any $p \geq 1$.  We finally observe that the symmetry condition \eqref{eq:symai} is met for $(\check a_0,\ldots,\check a_{2d})$. 
\end{proof}

\subsection{Proof of Theorem \ref{th:FKBlm}}\label{subsec:th7}
From  \eqref{eq:defBkmbis}, we may write
$$
\tau \PAR{ B^{(\ell,m)} } =  \tau_1 \PAR{   \sum_{g \in S_m}  (A_\star^\ell)_{g , \o} \otimes \tr_N U (g) },
$$
with $U(g) = U_{i_1} \cdots U_{i_m}$ if $g = g_{i_m}\cdots g_{i_1} \in S_m$. 
Next, we expand the partial trace $\tr_N$ in terms of the entries of the unitary matrix $U(g)$. If $g = g_{i_m} \cdots g_{i_1}$, we find
$$
\tr_N U(g) = \frac 1 N  \sum_{x_0,\ldots,x_m} \prod_{t=1}^m (U_{i_t})_{x_{t-1} x_t},
$$
where the sum is over all $x_t \in \INT{N}$ with $x_0 = x_m$. 
Note that $S_m$ is in bijection with the set $(i_1,\ldots,i_m) \in \INT{2d}^m$ such that $i_{t} \ne i^*_{t+1}$ for all $t \in \INT{m-1}$ (if $m = 1$, this last constraint is empty). Let  $P_m$ be the set of  $\gamma = (x_0,i_1,x_1,\ldots,i_m, x_m) \in \INT{N} \times \INT{2d} \times \ldots \times \INT{2d} \times \INT{N}$ with $x_0 = x_m$ and $i_{t} \ne i^*_{t+1}$ for all $t \in \INT{m-1}$. We might think about an element of $P_m$ as a path on $\INT{N}$ of length colored by the indices in $\INT{2d}$.  We now use the linearity of the expectation and sum over all $g \in S_m$. We get
\begin{eqnarray}
N \ABS{ \dE \tau \PAR{ B^{(\ell,m)} }  } = \ABS{  \tau_1  \PAR{ \sum_{\gamma  \in P_m}   a(\ell,\gamma)  w(\gamma) } } \leq    \NRM{  \sum_{\gamma  \in P_m}   a(\ell,\gamma)  w(\gamma) },\label{eq:tr1}
\end{eqnarray}
where we have introduced the operator and probabilistic weights: 
\begin{equation}\label{eq:defapgamma}
a(\ell,\gamma) = (A_\star^\ell)_{g(\gamma) , \o} \AND w(\gamma) =   \dE \prod_{t=1}^m (U_{i_t})_{x_{t-1} x_t},
\end{equation}
with, for $\gamma = (x_0,i_1,x_1,\ldots,i_m, x_m) $,  $g(\gamma) = g_{i_m}\cdots g_{i_1} \in S_m$. 
Note that $a(\ell,\gamma)$ does not depend on the $x_t$'s; only $w(\gamma)$ does.

In the proof of Theorem \ref{th:FKBlm}, it is crucial that, in contrast to previous arguments of matrix-valued trace methods \cite{MR4024563,BC2}, we do not use the triangle inequality and upper bound \eqref{eq:tr1} by  $\sum_{\gamma  \in P_m}  \NRM{  a(\ell,\gamma)}  | w(\gamma) |$. This is why Theorem \ref{th:NCCS} will come into play.

We organize the terms on the right-hand side of \eqref{eq:tr1} by introducing a first equivalence class on elements of 
$P_{m}$. 
We write $\gamma \sim \gamma'$ if there exist a permutation $\tau$ on $\INT{N}$ and a family of permutations $(\beta_x)_{x \in \INT{n}}$ on $\INT{2d}$ such that the image of $\gamma = (x_{0},i_{1},\ldots,x_m)$ by these permutations is $\gamma' = (x_{0}',i'_{1},\ldots, x'_m)$. More precisely, $\gamma \sim \gamma'$ if for all $t \in \INT{m}$, $\tau(x_t) = x'_t$ and $\beta_{x_{t-1}}(i_{t}) = i_t' = \beta_{x_{t}} (i_t^*)^*$.  
In a combinatorial language, if 
$P_{m}$ is seen as a collection of labeled paths where the labels are the vertex entries in $\INT{N}$ and the edge colors in $\INT{2d}$, an equivalence class is the corresponding unlabeled path.
We will use colored graphs defined formally as follows. 

\begin{definition}[Colored graph]\label{def:colorgraph}
Let $X$ be a countable set and $C$ a countable set with an involution $* : C \to C$. A {\em colored edge} $[x,i,y]  $ is an equivalence 
class of $X \times C \times X$ equipped with the equivalence relation $(x,i,y) \sim (y,i^*,x)$. A {\em colored graph} $G = (V,E)$ is the 
pair formed by a vertex set $V \subset X$ and a set of colored edges $[x,i,y]$ with $x,y \in V$. The {\em degree} of $v \in V$ is
 $\sum_{e = [x,i,y] \in E}  ( \IND ( v = x )  + \IND ( v = y) )$.
\end{definition}

To each element $\gamma  = (x_{0},i_{1},\ldots,x_m) \in P_{m}$, we define the colored graph $G_\gamma = (V_\gamma,E_\gamma)$ on the color set $\INT{2d}$ as follows. We set $V_\gamma = \{ x_{t} : t \in \INT{m} \} \subseteq \INT{N}$ and 
$E_\gamma = \{ [ x_{t-1} ,i_t, x_{t}] :   t \in \INT{m}\}$. 
The {\em multiplicity} of an edge $f = [x,i,y] \in E_\gamma$ is defined as 
$$
m(f) = \sum_{t= 1}^m \IND (  [ x_{t-1} ,i_t, x_{t}] = f ).
$$
By definition
$$
\sum_{f \in E_\gamma} m(f) = m.
$$
If  $e = |E_\gamma|$, $e_1$ is the number of edges of multiplicity one and $e_{\geq 2}$ is the number of edges of multiplicity 
at least two, we deduce
$$
e = e_1 + e_{\geq 2} \AND e_1 + 2 e_{\geq 2} \leq m . 
$$
Consequently
\begin{equation}\label{eq:bdee1}
e \leq (m + e_1 )/2. 
\end{equation}
Moreover, from the boundary condition $x_0 = x_m$, $\gamma = (x_0,i_1, \ldots,x_m)$ is a closed non-backtracking path in $G_\gamma$ it contains al least one cycle.  Setting $v =|V_\gamma|$, this implies that
\begin{equation}\label{eq:genus0}
v \leq e.
\end{equation}
In particular, combining \eqref{eq:bdee1}-\eqref{eq:genus0}, for any $\gamma \in P_{m}$, we find
\begin{equation}\label{eq:genus}
\chi = m/2 + e_1/2 - v \geq 0.
\end{equation}

For integers $v, e_1$ such that \eqref{eq:genus} holds,  let $P_{m} (v,e_1)$ be the set of elements in $P_{m}$  with $v$ vertices and $e_1$ edges of multiplicity one.  Note that if $\gamma \sim \gamma'$, the number of vertices and the number of 
edges with a given multiplicity are equal. We may thus define $\cP_{m} (v,e_1)$ as the set of equivalence classes 
with $v$ vertices and $e_1$ edges of multiplicity one.  Our first lemma is a rough bound on $\cP_{m} (v,e_1)$. 

\begin{lemma}\label{le:Plclass}
If $\chi  = m/2 + e_1/2 - v $, we have, 
$$
|\cP_{m} (v,e_1) |\leq m  ^{ 6 \chi +4 }.  
$$
\end{lemma}
\begin{proof}
We define a canonical element 
in each equivalence class  by saying that $\gamma = (x_0,i_1,\ldots,x_m) \in P_{m}$ is {\em canonical} if it is minimal 
for the lexicographic order in its equivalence class. For example, if $\gamma \in P_{m}(v, e_1)$ is canonical then 
$x_0 = x_m = 1$, $i_1 = 1$, $V_\gamma = \INT{v}$ and the vertices appear for the first time in the sequence $(\gamma_t)$ in order ($2$ before $3$ and so on).  
Our goal is then to give an upper bound on the number of canonical elements in $P_{m} (v,e_1)$. The case $m=1$ is trivial. We assume $m \geq 2$ in the sequel.

Fix such canonical element $\gamma$ in $P_{m} (v,e_1)$. Set $m_0 = m-1$. We think of $t \in \INT{m_0}$ as a time and we explore iteratively the sequence  $(x_{0},\ldots, ,x_t,i_{t+1})$ $t \in \INT{m_0}$. We build a non-decreasing sequence of colored trees $(T_t)$ as follows. Initially, for $t=0$, $T_0$ is the graph with no edge, and vertex set $\{1 \}$. At time $t \in \INT{m_0}$, if the addition to $T_{t-1}$ of the edge 
$[x_{t-1},i_t,x_{t}]$ does not create a cycle, then $T_t$ is the tree spanned by $T_t$ and $[x_{t-1},i_t,x_{t}]$, we then say that 
$[x_{t-1},i_t,x_{t}]$ is a {\em tree edge}. Otherwise $T_t = T_{t-1}$. By construction, the vertex set of $T_t$ is $ \{x_0,\ldots,x_t\} = \INT{u}$ where $u$ is the number of visited vertices so far (since $\gamma$ is canonical). In particular, $T_{m_0}$ is a spanning subtree of $\INT{v}$, and the number of tree edges is $f = v-1$. 

Next, we build a partition of $\INT{m_0}$. Let us say that $t \in \INT{m_0}$ is a {\em first time} if
$x_{t} \ne \{x_0,\ldots,x_{t-1}\}$, that is $x_t$ had not been seen so far. 
We say that 
$t \in \INT{m_0}$ is a {\em tree time} if  $[x_{t-1},i_t,x_{t}]$ is a tree edge which has already been visited. Finally, the other times $t \in \INT{m_0}$ are 
called {\em important times}. We have thus partitioned $\INT{m_0}$ into first times, tree times, and important times. By construction, $x_t \in \{x_0,\ldots,x_{t-1}\}$ for important and tree times.

We also note that all edges are visited at least twice except $e_1$ of them. We deduce that the number of important times is at most 
$$
m-1 - 2 f + e_1 \leq m - 2 v + e_1 +1 = 2 \chi +1. 
$$

Due to the non-backtracking constraint $i_{t+1} \ne i_t^*$, a first time cannot be directly followed by a tree time. The sequence $\gamma$ can thus be decomposed as 
the successive repetitions of: $(i)$ a sequence of first times (possibly empty), $(ii)$ an important time, 
$(iii)$ a sequence of tree times (possibly empty).

We mark each important time by the vector $(i_t,x_{t}, x_{\tau-1},i_{\tau})$ where $\tau > t$ is the next time which is not a tree time (if this time exists, otherwise $t$ is the last important time and we put the mark $(i_t,x_t)$).
We claim that the canonical sequence $\gamma$ is 
uniquely determined by the positions of the important times, their marks, and possibly the value of $i_m$. Indeed, the sequence $(x_{t},i_{t+1}, \ldots, x_{\tau})$ is the unique non-backtracking path in $T_t$ from $x_{t}$ to $x_\tau$ (there is a unique non-backtracking path between two vertices of a tree). Moreover, if $\sigma \geq \tau$ is the next important time or $\sigma = m$ if $t$ is the last important time, $(\tau, \ldots, \sigma-1)$ is a sequence of first times (if $\tau = \sigma$ this sequence is empty). It follows that if $u = |\{x_0,\ldots,x_{t-1}\}|$ vertices had been seen so far, we have $x_{\tau + k-1} = u+k$, for $k = 1, \cdots, \sigma- \tau$, by the minimality of canonical paths. Similarly, if $i_{\tau} \ne 1^* = d+1$, then $i_{\tau + k} = 1$ for  $k = 1, \cdots, \sigma- \tau-1$, while if $i_{\tau} = d+1$, then $i_{\tau + 1} = 2$, $i_{\tau + k} = 1$ for $k = 2, \cdots, \sigma- \tau-1$.

There are at most $m-1$ possibilities for the position of an important time and $m^2$ possibilities for its mark 
(there are at most $m$ edges in $G_\gamma$). In particular, the number of distinct canonical paths in $P_{m} (v,e_1)$ 
is generously bounded by 
$$
m (m^3)^{2 \chi+1} = m^{6\chi + 4},
$$
where the first factor $m$ accounts for the possible values of $i_m$.  \end{proof}

Let us now estimate the probabilistic weight $w(\gamma)$.  For that purpose, we rely on a result from \cite{BC2}. Let us introduce some notation. For $w \in \dC$ and $\veps \in \{\cdot,-\}$, we denote by $w^\veps$ either $w$ or $\bar w$ (if $\veps = \cdot$ or $\veps = -$). For integer $k \geq 1$ consider a sequence $x,y \in \INT{N}^k$ and $\veps \in \{\cdot , -\}^k$. For $U \in \dU_N$ Haar unitary, we are interested in estimating the average product
$$
 \EE \PAR{\prod_{t=1}^k  U_{x_{t} y_{t}}^{\varepsilon_{t}}}. 
$$
We say that the triple $(x,y,\veps)$ is {\em balanced}, if for any $u \in \INT{n}$, $\sum_{t} \IND ( \veps_t = \cdot ) \IND ( x_t = u) = \sum_{t} \IND ( \veps_t = - ) \IND ( x_t = u)$ and $\sum_{t} \IND ( \veps_t = \cdot ) \IND ( y_t = u) = \sum_{t} \IND ( \veps_t = - ) \IND ( y_t = u)$ (that is if the left and right indices in the product $\prod_t U_{x_{t} y_{t}}^{\varepsilon_{t}}$ are balanced with respect to complex conjugation). Similarly, for $\dO_N$, we say that the sequence is balanced if for  any $u \in \INT{n}$, $\sum_{t}  \IND ( x_t = u) $ and $\sum_t \IND( y_t = u)$ are both even. As above the multiplicity of a pair $(u,v) \in \INT{N}^2$ is $\sum_{t} \IND ( (x_t,y_t) = (u,v) ) $.  
 
\begin{lemma}\label{cor:WG2} For $N \geq 1$, let $U$ be a Haar unitary on $\dU_N$ or $\dO_N$. 
 Let $k$ be even with $2 k^{7/2} \leq  N^{2}$, $\veps \in \{\cdot,-\}^k$ and $x,y$ in  $\INT{N}^k$. For some numerical 
 constant $c>0$, 
$$\left| \EE \PAR{\prod_{t=1}^k  U_{x_{t} y_{t}}^{\varepsilon_{t}} }  \right| \leq c\, e^{k\eta} N^{-k/2}  \eta^{ e_1 } k^{m_4/2},$$
where $\eta =  k^{} N^{-1/4}$, $e_1$ is the number of pairs of multiplicity $1$,
$m_{4}$ is the sum of multiplicities of pairs with multiplicity at least $4$. Moreover, the above expectation is zero unless the triple $(x,y,\veps)$ is balanced.
\end{lemma}

\begin{proof}
We start with the unitary group $\dU_N$. The last statement is an immediate consequence of the Weingarten formula; see \cite[Theorem 4]{BC2}. Let $\delta = 3 k ^{7/2} N^{-2}$ and $\eta = k N^{-1/4}$ be the constants appearing in \cite[Theorem 11]{BC2}. From the proof of \cite[Corollary 14]{BC2}  we have 
$$
I = N^{k/2} (1+\delta)^{-1} \left| \EE \PAR{\prod_{t=1}^k  U_{x_{i} y_{i}}^{\varepsilon_{i}} }  \right| \leq k^{ m_4/2} (1+ \eta)^k \eta^{e_1}.
$$
It remains to use $(1+a)^b \leq e^{ab}$ and adjust the constant $c$. For the orthogonal group $\dO_N$, the same argument works thanks to \cite[Theorem 16]{BC2} (see discussion below  \cite[Theorem 16]{BC2}).
\end{proof}

As a corollary, we obtain the following bound on $w(\gamma)$ defined in \eqref{eq:defapgamma}.
\begin{corollary}\label{le:pgamma} Assume $2 m^{7/2} \leq  N^{2}$.
Let $\gamma \in P_{m} (v, e_1)$ and $\chi = m/2 + e_1/2 - v$. If $v > m/2$ then $w(\gamma ) = 0$. Otherwise, 
there exists a numerical constant $c >0$ such that 
$$
w(\gamma) \leq c e^{m \eta} N^{ -  m/2}  \eta^{e_1} m  ^{ 2\chi  },
$$
with $\eta =  m N^{-1/4}$. 
\end{corollary}

\begin{proof}
We treat the case of $\dU_N$. Using the independence of the matrices $(U_1,\ldots,U_d)$, $w(\gamma ) =  |\dE \prod_{t=1}^m (U_{i_t})_{x_{t-1} x_t}|$ can be decomposed as a product of $d$ terms of the form given by Lemma \ref{cor:WG2}. 
It follows from Lemma \ref{cor:WG2} that
$w(\gamma) = 0$ unless for each vertex $x \in V_\gamma$ and each $i \in \INT{2d}$, $$\sum_{t = 1}^m \IND (i_t = i) \IND(x_{t-1} = x) = \sum_{t = 1}^m \IND (i_t = i^*) \IND (x_{t} = x).$$ Since $i_{t} \ne i^*_{t+1}$, each vertex on $G_\gamma$ different from the starting vertex $x_0$ is at least twice equal to $x_t$ with $t \in \INT{m}$ and at least once for $x_0$ (since $x_0 =x_m$). This
last condition implies that $2 (v -1) +1 \leq m $ or equivalently $v \leq m/2$ (since $m$ is even). This proves the first claim. The same argument works for $\dO_N$.

Let us now prove the second claim. By Corollary \ref{cor:WG2}, it suffices to prove that 
\begin{equation}\label{eq:m3bd}
m_{\geq 4} \leq  4 \chi ,
\end{equation}
where $m_{\geq 4} = \sum_{e} \IND(m_e \geq 4) m_e$ is the sum of multiplicities of edges of $G_\gamma$ with multiplicity at least $4$. 
Indeed, let $e_{23}$ be the number of edges of multiplicity $2$ or $3$. We have 
$$
e_1 + 2 e_{23} + m_{\geq 4} \leq m  \AND e_1 + e_{23} + \frac{m_{\geq 4}}{ 4} \geq e \geq v ,
$$
where $e = |E_\gamma|$ is the number of edges and where we have used \eqref{eq:genus0}. 
We may cancel $e_{23}$:
$$
-e_1 + \frac{m_{\geq 4}} { 2} \leq m - v. 
$$
Therefore, \eqref{eq:m3bd} is satisfied, and the proof is complete. \end{proof}

To understand better the right-hand side of \eqref{eq:tr1}, we introduce a new colored graph $\widehat G_\gamma$ associated with a path $\gamma = (x_0,i_1,\ldots, x_m) \in \cP_m(v,e_1)$. In words, $ \widehat G_\gamma = (\widehat V_\gamma , \widehat E_\gamma)$ is obtained from $G_\gamma =( V_\gamma, E_\gamma) $ by gluing together all degree $2$ vertices (different from $x_0$). Such a graph is often called the {\em kernel} of a base graph (here $G_\gamma$). 

More precisely, let $V_{\geq 3} \subseteq V_\gamma$ be the subset of vertices of degree at least $3$. We set $\widehat V_\gamma = V_{\geq 3} \cup \{x_0 \}$. Note that $V_\gamma \backslash \widehat V_\gamma$ contains only vertices of degree $2$ (since $i_{t} \ne i_{t-1}^*$ all vertices different from $x_0$ have degree at least $2$). Let us now define the edges of $\widehat G_\gamma$. We say that $\pi = (y_0,j_1,\ldots,y_k)$ with $y_t \in V_\gamma$ and $j_t \in \INT{2d}$ is a {\em path} on $G_\gamma$ of length $k$ if for all $t \in \INT{k}$, $j_{t} \ne j_{t-1}^*$ and $[y_{t-1}, j_t, y_{t}] \in E_\gamma$. We say that the path is {\em maximal} if $y_0 , y_k \in \widehat V_\gamma$ and $y_t \notin \widehat V_\gamma$ for all $t \in \INT{k-1}$.  Finally, we equip the set of maximal paths of $G_\gamma$, say $C$, with the involution $\pi^* = (y_k,j_{k}^*,y_{k-1},\ldots, j_{1}^*,y_0)$. In words, this involution maps a path to its reversed path. We may then define $\widehat G_\gamma = (\widehat V_\gamma, \widehat E_\gamma)$ as the colored graph on $C$ (as per Definition \ref{def:colorgraph}) with colored edges $[y_0, \pi, y_k ]$ with $\pi = (y_0,j_1,\ldots,y_k) \in C$ maximal path. We also note that since vertices in $V_\gamma \backslash \widehat V_\gamma$ have degree $2$, the maximal paths in $C$ may only intersect at their end vertices in $\widehat V_\gamma$. In particular, each vertex in $V_\gamma \backslash \widehat V_\gamma$ is contained in exactly one edge of $\widehat G_\gamma$.
 
Let us now introduce a finer equivalence for paths. This is motivated by the fact that $w(\gamma)$ is not constant over each equivalence class (even if we have a common upper bound given by \eqref{le:pgamma}). Consider for example $d = 5$ and paths which starts with $\pi_1 = (x_0,1,x_1,2,x_2,3,x_3,4,x_4,5,x_0)$ and $\pi_2 = (x_0,1,x_1,2,x_2,3,x_3,4,x_4,2,x_0)$ with all $x_i$'s distinct. For $i = 1,2$, let $\gamma_i$ be the concatenation of $\pi_i$ and  $\pi_i' = (x_0,1^*,x_1,\ldots)$. We have $\gamma_1 \sim \gamma_2$, but the multiplicities of colored edges are different. We may then check by hand that $w(\gamma_1) \ne w(\gamma_2)$. 

Accordingly, we introduce a finer notion of path equivalence denoted by $\hat \sim$ such that $\gamma \hat \sim  \gamma$  does imply that $\gamma \sim \gamma'$ and $w(\gamma) = w(\gamma')$. This notion of equivalence is tuned so that in the forthcoming Lemma \ref{le:sumagamma}, we have a sharp upper bound on the operator norm of $\sum_{\gamma'} a(\ell,\gamma')$ where, for a fixed $\gamma \in \cP_{m}(v,e_1)$ the sum is over all $\gamma'$ such that $\gamma' \hat \sim \gamma$ 
(recall that $a(\ell,\gamma)$ was defined in \eqref{eq:defapgamma}). This has to be done carefully to guarantee that the number of equivalence classes of $\hat \sim$ do not grow too fast. This whole procedure in the main technical step toward the proof of Theorem \ref{th:FKBlm}.

\begin{definition}[Profile of a path]\label{def:profile} Let $\gamma \in P_m$ with kernel graph $\widehat G_\gamma = (\widehat V_\gamma,\widehat E_\gamma)$.
To each edge $f = [y_0,\pi,y_k]  \in \widehat E_\gamma$, $\pi = (y_0,j_1,y_1,\ldots,j_k, y_k)$ we associate its {\em profile} as the collection of integers $(j_1,j_k,(\prof_i)_{i \in \INT{2d}}),$ where $\prof_i$ is the number of $t \in \INT{k}$ such that $j_t = i$. 
\end{definition}
 
We remark that if $\gamma' \sim \gamma$ is the image of $\gamma$ by the permutation $\tau$ on $\INT{n}$ and the family of permutations $(\beta_x)_{x \in \INT{n}}$ on $\INT{2d}$, then $\widehat G_{\gamma'}$ is the image of $\widehat G_\gamma$ by these permutations.  We may then define our new path equivalence. We say that $\gamma \hat \sim \gamma'$ if $\gamma \sim \gamma'$ and the corresponding profiles of edges are equal. 

\begin{lemma}\label{le:newclass}
If $\gamma \hat \sim \gamma'$ then 
$
w(\gamma) = w(\gamma').
$
\end{lemma}
\begin{proof}
From the independence of $U_i$, $i \in \INT{d}$, $w(\gamma)$ is equal to a product of the expectations of the form $p (x,y,\veps) = \EE \PAR{\prod_{t=1}^k  U_{x_{t} y_{t}}^{\varepsilon_{t}}}$ with $U$ Haar distributed of $\dU_N$ and $\veps_t = \{\cdot,-\}^k$, $x,y \in \INT{N}^k$. For any permutation matrices, $S,T$ of size $N$, $S U T $, and $U$ have the same distribution. Hence, we have $p(x,y,\veps) = p(\sigma(x),\tau(y),\veps)$ if $\sigma,\tau$ are any permutations on $\INT{N}$. We observe that if $\gamma \hat \sim \gamma'$, there exists a bijection between the vertices of $\widehat G_\gamma$ and  $\widehat G_{\gamma'}$ and since the profiles on edges coincide, there is also a pair of bijections between the left and right internal vertices adjacent to an edge of $E_\gamma$ and $E_{\gamma'}$ with color $i$ or $i^*$ (note that due to the constraint $i_t \ne i_{t-1}^*$ and the fact that internal vertices have degree $2$ no internal vertex can be adjacent to both $i$ and $i^*$). The claim follows. 
\end{proof}

We denote by $\widehat \cP_m(v,e_1)$ the set of equivalence classes for $\hat \sim$ with $v$ vertices and $e_1$ edges of multiplicity one. The next lemma asserts that $\widehat \cP_m(v,e_1)$ is not much bigger than $\cP_m(v,e_1)$. 

\begin{lemma}\label{le:Plclass2}
If $\chi = m/2 + e_1/2 - v$, we have
$$
| \widehat \cP_m(v,e_1) | \leq (2d m^{d})^{6\chi + 4} | \cP_m(v,e_1) |.
$$
\end{lemma}

\begin{proof}
Let $\gamma \in \cP_{m}(v,e_1)$. Assume that $G_\gamma$ has $v$ vertices, $e$ edges, and $e_1$ edges of multiplicity one. Let us call $\hat v$ its number of vertices and $\hat e$ its number of edges.
By definition of $\hat \sim$, the equivalence class of $\gamma \in  \cP_{m}(v,e_1)$ is subdivided into at most 
$
((2d)^2 m^{2d})^{\hat e}
$ classes in $\widehat \cP_{m}(v,e_1)$.
Note that this estimate is very rough: if the profile of an edge is $(j_1,j_k, (\prof_i)_{i \in \INT{2d}})$, there are at most $2d$ possibilities for $j_1$ and $j_k$ and $m$ possibilities for $\prof_i$. Thus, it suffices to prove that  $\hat e \leq 3\chi +2$ and the lemma follows.

 To this end, we first notice that 
\begin{equation}\label{eq:e-v}
\hat e - \hat v =  e- v.
\end{equation}
Indeed, a maximal path $\pi = (y_0,j_1,\ldots,y_k) \in C$ has $k-1$ vertices in $V_\gamma \backslash \widehat V_\gamma$ and contains $k$ edges of $E_\gamma$. 

Next, let $u = |V_\gamma \backslash \widehat V_\gamma|$ and for $x \in V_\gamma$, let $\mathrm{deg}(x)$ be the degree of $x$ in $G_\gamma$. The vertex $x_0 \in \widehat V_\gamma$ has a degree at least $1$, all other elements of $\widehat V_\gamma$ have a degree at least $3$, and vertices of $V_\gamma \backslash \widehat V_\gamma$ have degree $2$. We deduce that
$$\hat v + u  = v \quad \hbox{ and } \quad 3 (\hat v -1) + 2 u  + 1 \leq \sum_{x \in V_\gamma} \mathrm{deg}(x) = 2 e.$$ 
By subtracting twice the first identity from the inequality, we cancel $u$ and arrive at 
 $$
 \hat v \leq 2 e - 2 v +2.
 $$
 Using  \eqref{eq:bdee1}-\eqref{eq:e-v}, we thus obtain the following bound on $\hat e$:
\begin{equation}\label{eq:bdhate}
\hat e = \hat v + (e - v) \le 3 e - 3 v + 2 \leq 3 \chi + 2.
\end{equation}
The conclusion follows.
\end{proof}

Our final lemma is the estimation of $\sum a(\ell,\gamma)$ where the sum is over an equivalence class and $a(\ell,\gamma)$ was defined in \eqref{eq:defapgamma}. 
Its proof relies crucially on Theorem \ref{th:NCCS}.

\begin{lemma}\label{le:sumagamma}
Let $\gamma \in \widehat \cP_m(v,e_1)$ and $\chi = m/2 + e_1/2 - v$. We have 
$$
\rho (\ell,\gamma) = \NRM{ \sum_{ \gamma' : \gamma' \hat \sim \gamma}  a(\ell,\gamma') }  \leq N^{v} \rho^\ell (2d)^{e_1/2}  \ell^{18 \chi +11}.
$$
\end{lemma}

\begin{proof} 
 We use $\widehat G_\gamma$ as a backbone of the closed path $\gamma$ to upper bound $\rho(\ell,\gamma)$. Recall  the multiplicity introduced below Definition \ref{def:colorgraph}: if $[x,i,y] \in E_\gamma$, $m([x,i,y])$, is the number of times that $[x_{t-1}, i_{t}, x_{t} ] = [x,i,y]$. Let $e_t$  be the number of edges of $E_\gamma$ of multiplicity equal to $t$. We have
$
\sum_t e_t = e$ and $\sum_t t  e_t = m$.
In the sequel, we will use that
\begin{equation} \label{eq:bounde3}
\sum_{t} ( t -2 )_+ e_t = \sum_{t} ( t -2 ) e_t  + e_1 =  m - 2 e +e_1.
\end{equation}
We may also unambiguously define the  multiplicity $m([x,\pi,y])$ of an edge $[x,\pi,y]  \in \widehat E_\gamma$ with $\pi  =  (y_0,j_1,y_1, \ldots, j_{k}, y_k)$, $y_0 = x$, $y_k = y$. Indeed, since $y_s$, $s \in \INT{k-1}$, has degree $2$ and since $\gamma$ is non-backtracking ($i_t \ne i_{t-1}^*)$, all edges $[y_{s-1},i_s,y_{s}]$, $s \in \INT{k}$, have the same multiplicity.  The multiplicity $m([x,\pi,y])$ is defined as the common multiplicity $m([y_{s-1},i_s,y_{s}])$, $s \in \INT{k}$. 

Since the path $\gamma$ is non-backtracking, we may decompose it into successive visits of the edges of $\widehat E_\gamma$. 
More precisely, we decompose $\gamma $ as $ (p_{1}, p_{2}, \ldots , p_{r})$ with $$p_t =   (y_{t,0},j_{t,1},y_{t,1},\ldots, j_{t,k_t},y_{t,k_t}),$$  where either (a) $p_{t} $ 
follows an edge of $\widehat E_\gamma$ which is visited for the first or last time, or (b) $p_{t}$ follows a sequence  of edges of 
$\widehat E_\gamma$ which are not visited for the first or last time. By construction, there are at most $2\hat e$ subpaths $p_{t}$ of type (a), and thus at most $ \hat e$ subpaths of type (b) (any subpath of type (b) is preceded by a subpath of (a)). Hence, 
\begin{equation}\label{eq:ubr}
r \leq 3 \hat e.
\end{equation}
 By construction, we have 
$$
\sum_t k_t = m.
$$
This decomposition into subpaths is equivariant in each equivalence class: if $\gamma' \hat \sim \gamma$ is the image of $\gamma$ by the permutation $\tau$ on $\INT{n}$ and the family of permutations $(\beta_x)_{x \in \INT{n}}$ on $\INT{2d}$ then $\widehat G_{\gamma'}$ is the image of $\widehat G_\gamma$ by these permutations, and $\gamma'$ is decomposed into subpaths $\gamma '= (p'_1,\ldots,p'_r)$ which are the images of $(p_1,\ldots,p_r)$.

Next, we decompose the value of $a(\ell,\gamma)$ over each subpath. Using Lemma \ref{le:pathdecomFd}, we may write 
\begin{equation}\label{eq:decompalgamma}
a(\ell,\gamma) =  \sum_{(\lambda_1, \ldots ,\lambda_r) } a_{1}(\lambda_r,p_r) \cdots  a_{1}(\lambda_2,p_2) a(\lambda_1,p_1),
\end{equation}
where the sum is over all $(\lambda_1,\ldots,\lambda_r)$ such that $\sum_{t} \lambda_t = \ell$, $\lambda_t \geq 1$ and, for $p = (y_0,j_1,\ldots,y_k)$, $a_1(\lambda,p)$ is defined as
\begin{equation}\label{eq:defAi} 
a_1 (\lambda,p) =  (C_\star^{(\lambda,j_1)} )_{g(p) , \o },
\end{equation}
where $C^{(\lambda,j)}_\star$ was defined in \eqref{eq:defCki}, $g(p) = g_{j_k}\cdots g_{j_1}$ and $g(p)$ is in reduced form since $j_t \ne j_{t-1}^*$.   By Lemma \ref{le:normCk},
\begin{equation}\label{eq:defAinrm} 
\| a_1 (\lambda,p) \| \leq \rho^{\lambda}.
\end{equation} 

Next, we sum over all $\gamma'$ with $\gamma' \hat \sim \gamma$. As pointed out above, they are decomposed into subpaths $\gamma '= (p'_1,\ldots,p'_r)$. Note that the indices $j_{1,1},j_{2,1}, \ldots, j_{r,1}$ and $j_{1,k_1},j_{2,k_2},\ldots,j_{r,k_{r}}$ are common to all $\gamma \hat \sim \gamma$ since these indices belong to the profile of $\gamma$ as defined in Definition \ref{def:profile}. Using the triangle inequality in \eqref{eq:decompalgamma}, we obtain:
\begin{eqnarray}
\rho(\ell,\gamma) = \NRM{  \sum_{\gamma' \hat \sim \gamma} a(\ell,\gamma')  }  \leq \sum_{(\lambda_1, \ldots ,\lambda_r) } \NRM{ \sum_{\gamma' \hat \sim \gamma}   a_{\delta_r}(\lambda_r,p'_r)  \cdots  a_{\delta_1}(\lambda_1,p'_1) },\label{eq:nrmsumg}
\end{eqnarray}
where, we have set 
 $\delta_1 = \emptyset$, for $2 \leq t \leq r$, $\delta_t = 1$ and $a_\emptyset (\lambda,p) =a (\lambda,p)$.

We decompose the above product over paths over types (a) and (b). Let $S \subset \INT{r}$ be the indices $t$ such that $p_t$ is of type (a). The set $T = \INT{r} \backslash S$ is thus the set of $t$'s such that $p_t$ is of type (b). By construction, the edge set $\widehat E_\gamma$ defines a partition of $S$ where each block has $1$ or $2$ elements, depending on whether $f = [x,\pi,y] \in \widehat E_\gamma$ is visited once or at least twice. Indeed, for $t \in S$, $p_t$ visits an edge of $\widehat E_\gamma$ and thus $p_t \in \{\pi,\pi^*\}$ for some $f = [x,\pi,y]\in \widehat E_\gamma$.

The sum over all $\gamma' \hat \sim \gamma$ can be decomposed over the $N(N-1)\cdots (N-v+1)$ possibilities for the distinct $x_t \in \INT{N}$ and over  the product set 
$$
J = \bigotimes_{ f\in \widehat E_\gamma } J_f, 
$$
where for $f = [y_0,\pi,y_k]$ with $\pi = (y_0,j_1,\ldots,j_k,y_k)$ and with profile $(j_1,j_k,(\prof_i)_{i \in \INT{2d}})$, $J_f$ accounts for the set of possibilities for the indices $\bm j = (\bm j_1,\ldots,\bm j_k)$ such that $\bm j_t \ne \bm j_{t-1}^*$, $\bm j_ 1= j_1$, $\bm j_k = j_k$ and, for each $i \in \INT{2d}$, there are precisely $\prof_i$ $\bm j_t$'s equal to $i$.  Note that the summand in \eqref{eq:nrmsumg} does not depend on the choice of the distinct $x_t$'s. Also, by construction for $t \in T$, $p_t$ depends only on the choice of the $J_f$'s which have been visited once and which will be visited at a later time in $S$. We are thus in a position to apply Theorem \ref{th:NCCS} (see Remark \ref{rk:NCCS}). We get 
\begin{equation}\label{eq:bdwQt}
 \NRM{ \sum_{\gamma' \hat \sim \gamma}   a_{\delta_r}(\lambda_r,p'_r)  \cdots  a_{\delta_1}(\lambda_1,p'_1) } \leq N^v \prod_{t \in \INT{r}} Q_t,
\end{equation}
where $Q_t$ are defined as follows. 

\noindent{\em - $Q_t$ for $t \in T$. } For $t \in T$, i.e. $p_t$ of type (b), from \eqref{eq:defAinrm},
$$
Q_t = \max_{p' \sim p_t}\NRM{ a_{\delta_t} (\lambda_t,p')} \leq \rho^{\lambda_t}.
$$

\noindent{\em - $Q_t$ for $t \in S$ in a block of size $2$. }
For $t \in S$ which visits the edge $f = [y_0,\pi,y_k]$, $\pi = (y_0,j_1,\ldots,j_k,y_k)$, visited at least twice, we have, by Theorem \ref{th:NCCS},
$$
Q_t = \max \left( \NRM{ \sum_{\bm{j} \in J_f}  a_{\delta_t} (\lambda_t , p(\bm j)) a_{\delta_t} (\lambda_t , p(\bm j))^*}, \NRM{\sum_{\bm j \in J_f} a_{\delta_t} (\lambda_t , p(\bm j))^* a_{\delta_t} (\lambda_t , p(\bm j))} \right)^{1/2},
$$
where $\bm j = (\bm j_1, \ldots, \bm j_k)$ and $p(\bm j) = (y_0,\bm j_1, y_1,\ldots, \bm j_k,y_k)$. In the above sums, the summand is positive semi-definite. Hence, we may only increase its norm by adding more positive semi-definite elements. From \eqref{eq:defAi}, we deduce that, if $\delta_t = 1$ and $j_1 = i$, 
\begin{eqnarray*}\label{eq:Qta}
Q_{t}  & \leq & \max \left( \NRM{ \sum_{ g  \in S_{i,k} } (C^{(\lambda_t,i)}_\star) _{ g \o }     (C^{(\lambda_t,i)}_\star)^* _{ g \o } }  ,  \NRM{ \sum_{ g  \in S_{i,k} }    (C^{(\lambda_t,i)}_\star)^* _{ g \o }    (C^{(\lambda_t,i)}_\star) _{ g \o }  }   \right)^{1/2}\\
& \leq & \max \left( \NRM{ \sum_{ g  \in \dF_d } (C^{(\lambda_t,i)}_\star) _{ g \o }     (C^{(\lambda_t,i)}_\star)^* _{ g \o } }  ,  \NRM{ \sum_{ g  \in \dF_d }    (C^{(\lambda_t,i)}_\star)^* _{ g \o }    (C^{(\lambda_t,i)}_\star) _{ g \o }  }   \right)^{1/2}\\
& =  & \max \left( \NRM{  (C^{(\lambda_t,i)}_\star (C^{(\lambda_t,i)}_\star)^*)  _{ \o \o }   }  ,  \NRM{ ((C^{(\lambda_t,i)}_\star)^* C^{(\lambda_t,i)}_\star )  _{ \o \o }  }   \right)^{1/2} ,
\end{eqnarray*}
where at the first line, $S_{i,k}$ is the set of elements of the free group $\dF_d$ which are written in reduced form as $g = g_{\bm j_k}\cdots g_{\bm j_1}$ with $\bm j_1 = i$. In the second line, we use the positivity. In the last line, we have used for all $x,y \in \dF_d$, $(M_{xy})^* = (M^*)_{yx}$ for $M \in \cB(H_1 \otimes \ell^2 (\dF_d))$ and, if, in addition $M \in \cA_1 \times \cA_\star$, we have $M_{xg,yg} = M_{xy}$ for all $g \in \dF_d$ and thus
$$
\sum_{g \in \dF_d} M_{g \o} (M^*)_{\o g} = \sum_{g \in \dF_d} M_{\o g^{-1}} (M^*)_{g^{-1} \o} = \sum_{g \in \dF_d} M_{\o g} (M^*)_{g \o}  = (M M ^*)_{\o \o},
$$
(at the second step, we use that $g \to g^{-1}$ is a bijection of $\dF_d$).  By Lemma \ref{le:normCk}, we deduce finally that 
$$
Q_t \leq \lambda_t \rho^{\lambda_t}.
$$
Similarly, if $\delta_t = \emptyset$, the same conclusion holds. 

\noindent{\em - $Q_t$ for $t \in S$ in a block of size $1$. } For $t \in S$ which visits the edge $f = [y_0,\pi,y_k]$, $\pi = (y_0,j_1,\ldots,j_k,y_k)$, visited once, we have, by Theorem \ref{th:NCCS}, 
$$
Q_t = \max \left( \NRM{ \sum_{\bm{j} \in J_f}  \sqrt{ a_{\delta_t} (\lambda_t , p(\bm j)) a_{\delta_t} (\lambda_t , p(\bm j))^*}}, \NRM{\sum_{\bm j \in J_f} \sqrt{a_{\delta_t} (\lambda_t , p(\bm j))^* a_{\delta_t} (\lambda_t , p(\bm j))}} \right),
$$
where $\bm j = (\bm j_1, \ldots, \bm j_k)$ and $p(\bm j) = (y_0,\bm j_1, y_1,\ldots, \bm j_k,y_k)$. Arguing as above, we find, for  $\delta_t = 1$ and $j_1 = i$, 
$$
Q_t \leq \max \left( \NRM{ \sum_{ g  \in S_{i,k} } \sqrt{   (C^{(\lambda_t,i)}_\star) _{ g \o }     (C^{(\lambda_t,i)}_\star)^* _{ g \o }  }} ,  \NRM{ \sum_{ g  \in S_{i,k} }   \sqrt{(C^{(\lambda_t,i)}_\star)^* _{ g \o }    (C^{(\lambda_t,i)}_\star) _{ g \o }  } } \right).
$$
There are $(2d-1)^{k} \leq (2d)^{k}$ elements in $S_{k,i}$. Using Theorem \ref{th:NCCS} again (for $r = 1$, $k = 2$), we get 
$$
Q_t \leq (2d)^{k/2}\max \left( \NRM{ \sum_{ g  \in S_{i,k} } \sqrt{   (C^{(\lambda_t,i)}_\star) _{ g \o }     (C^{(\lambda_t,i)}_\star)^* _{ g \o }  }} ,  \NRM{ \sum_{ g  \in S_{i,k} }   \sqrt{(C^{(\lambda_t,i)}_\star)^* _{ g \o }    (C^{(\lambda_t,i)}_\star) _{ g \o }  } } \right)^{1/2}.
$$
We are then back to the previous case. We get 
$$
Q_t \leq (2d)^{k/2} \lambda_t \rho^{\lambda_t}.
$$
Similarly, if $\delta_t = \emptyset$, the same conclusion holds. 

We may now use the various upper bounds of $Q_t$ in \eqref{eq:bdwQt}.  By construction $e_1$, the sum of edges of $E_\gamma$ of multiplicity one is also the sum of the length of edges of $\widehat E_\gamma$ of multiplicity one. We thus have obtained the upper bound:
$$
 \NRM{ \sum_{\gamma' \hat \sim \gamma}   a_{\delta_r}(\lambda_r,p'_r)  \cdots  a_{\delta_1}(\lambda_1,p'_1) } \leq N^v (2d)^{e_1/2} \prod_{t=1}^r \lambda_t \rho^{\lambda_t}.
$$
From \eqref{eq:nrmsumg}, we find
$$
\rho(\ell,\gamma) \leq N^v (2d)^{e_1/2} \sum_{(\lambda_1, \ldots ,\lambda_r) } \lambda_t \rho^{\lambda_t}.
$$
By assumption, $\lambda_1+ \cdots + \lambda_r = \ell$, $1 \leq \lambda_t \leq \ell$ and thus 
$$
\sum_{(\lambda_1, \ldots ,\lambda_r)} \prod_t \lambda_t \leq  \sum_{(\lambda_1, \ldots ,\lambda_r)} \ell^{r} \leq \ell^{2r-1} \leq \ell^{6 \hat e-1}.
$$
Finally, the lemma follows from \eqref{eq:bdhate}. \end{proof}

All ingredients are gathered to prove Theorem \ref{th:FKBlm}. 

\begin{proof}[Proof of Theorem \ref{th:FKBlm}]
From \eqref{eq:tr1} and Lemma \ref{le:newclass}, with $\rho(\ell,\gamma)$ as in Lemma \ref{le:sumagamma} and $\widehat \cP_{m} (v,e_1)$ as in Lemma \ref{le:Plclass2}, we find 
\begin{equation}\label{eq:trBlmf}
N \ABS{ \dE \SBRA{ \tau (B^{(\ell,m)} )}} \leq   \sum_{v =1}^{m/2} \sum_{e = 1}^{m} | \widehat \cP_{m} (v,e_1)| \max_{\gamma  \in \widehat \cP_{m} (v,e_1)}  \rho(\ell,\gamma) | w(\gamma)|,
\end{equation}
where we have used \eqref{eq:genus0} and Corollary \ref{le:pgamma} $e_1 \leq m$ and $v \leq m/2$ (unless the summand is zero). Next, we use Corollary \ref{le:pgamma} and Lemma \ref{le:sumagamma}:  
\begin{equation*}
\rho(\ell,\gamma) | w(\gamma) | \leq c e^{m \eta} N^{ -  m/2}  \eta^{e_1} m  ^{ 2\chi  } N^{v} \rho^\ell (2d)^{e_1/2} \ell^{18 \chi +11},
\end{equation*}
with $\eta = m N^{-1/4}$ and $\chi = m/2 +e_1/2 - v$. We isolate the dependency in $e_1$ and set $r = m/2 - v$. We obtain 
$$
\rho(\ell,\gamma) | w(\gamma) | \leq c e^{m \eta}  \ell^{11} \rho^\ell  \Delta_1^{e_1} \Delta_2 ^r
$$
with 
$$
\Delta_1 = \eta m \sqrt{2d}   \ell^{9} \AND \Delta_2 = N^{-1} m^2  \ell^{18}.
$$
By Lemma \ref{le:Plclass} and Lemma \ref{le:Plclass2}, we find
$$
|\widehat \cP_{m} (v,e_1) |  \max_{\gamma  \in \widehat \cP_{m}(v,e_1)}  \rho(\ell,\gamma) | w(\gamma)| \leq   (2 d m^{d +1})^4 c e^{m \eta}  \ell^{11} \rho^\ell  ((2 d m^{d +1})^3 \Delta_1)^{e_1} ( (2 d m^{d +1}) ^6 \Delta_2) ^r. 
$$ 
Since $m \leq \ell$, we have, for our choice of $\ell \leq \ell_0$, 
$$
(2 d m^{d +1})^3 \Delta_1 \leq (2d)^{7/2}\ell^{3d + 5} \ell^{9} N^{-1/4} \leq 1/2. 
$$
Similarly, 
$$
(2 d m^{d +1}) ^6 \Delta_2 \leq (2d)^6 \ell^{6d + 8}\ell^{18} N^{-1} \leq 1/2.
$$
Hence,
\begin{eqnarray*}
N \ABS{ \dE \SBRA{ \tau (B^{(\ell,m)} )} } & \leq &   (2 d m^{d +1})^4 c e^{\ell^2 N^{-1/4}}  \ell^{11} \rho^\ell \sum_{r\geq 0} ((2 d m^{d +1}) ^6 \Delta_2)^{r} \sum_{e_1 \geq 0} ((2 d m^{d +1})^3  \Delta_1)^{e_1}\\
& \leq & 4 e c   (2 d m^{d +1})^4  \ell^{11} \rho^\ell.
\end{eqnarray*}
(since $\ell^2 \leq N^{1/4}$ for our choice of $\ell \leq \ell_0$). The first statement of the theorem follows.

For the second statement, we use \eqref{eq:normAltr} and get, for even $\ell$,
$$
\ABS{ \dE \SBRA{\| A \|_\ell^\ell} - \| A_\star \|_\ell^\ell } \leq  c \rho^\ell   N^{-1} d^4  \ell^{4 d + 16} .
$$
as requested. 
\end{proof}

\subsection{A variant with $\theta$-control: proof of Theorem \ref{th:main2}}

We now present a possible improvement over Theorem \ref{th:FKBlm}. 
We fix a sequence $(a_0,a_1, \ldots, a_{2d})$ satisfying the symmetry condition \eqref{eq:symai}. 
We introduce a new definition. We say that the  {\em paths generated by $(a_0,a_1, \ldots, a_{2d})$ are $\theta$-controlled} if for all $r \geq k \geq 1$, all integers $m \geq 1$, all partitions $\pi$ as in Definition \ref{def:Deltacontrol} and all families $X = (X_{i,\vec j}), i \in \INT{k}, \vec j \in \INT{m}^k$ in $\cA$, depending on $(a_0,\ldots,a_{2d})$ as described after, we have that $(X,\pi)$ is $\theta$-controlled as per Definition \ref{def:Deltacontrol}. We require that for each $i,\vec j$, $X_{i,\vec j}$ is a product of elements $Y_1\cdots Y_l$ where for $t \in \INT{l}$, $Y_t = ((A^{\o}_\star)^{\lambda_t} )_{g  g_{\delta_t} , g_{\delta_t}}$ or $Y_t = (A_{\star}^{\lambda_t})_{g_t \o}$ for some $\delta_t \in \INT{2d}$, $g_t \in \dF_d$ and $\lambda_t$ integer.

By Lemma \ref{le:discr4}, the paths generated by  $(a_0,a_1, \ldots, a_{2d})$ are $1$-controlled if all $(a_0,a_1, \ldots, a_{2d})$ are bistochatic matrices. Similarly, by Lemma \ref{le:NCCSuni}, if $\cA_1 = M_n(\dC) \otimes \cA_0$ and $a_i = b_i \otimes u_i$ with $b_i $ unitary, then  the paths generated by  $(a_0,a_1, \ldots, a_{2d})$ are $n$-controlled.

For controlled paths, we have the following variant of Theorem \ref{th:FKBlm}, which implies Theorem \ref{th:main2} as explained in the proof of Theorem \ref{th:main1}.

\begin{theorem}\label{th:FKBlmtheta}
Let $\theta \geq 1$. Assume that $(a_0,a_1,\ldots,a_{2d})$ satisfy  the symmetry condition \eqref{eq:symai} and that their paths are $\theta$-controlled. Let $\ell_0$ be the largest integer $\ell$ such that 
$$
2 (2d)^{2} \ell^{14} \theta^{3/2} \leq N^{1/4}. 
$$
Then, for all $1 \leq m \leq \ell \leq \ell_0$, for some numerical constant $c >0$, we have
$$
\ABS{ \dE \SBRA{\tau \PAR{ B^{(\ell,m)} } }}\leq c  \frac{d^2  m^{4} \ell^{11}  \theta^2 }{N} \rho^\ell.
$$
and, if $\ell$ is even,
$$
\ABS{ \dE \SBRA{\| A \|_\ell^\ell}  - \| A_\star \|_\ell^\ell } \leq   \frac{c}{\sqrt N}  \rho^{\ell}.
$$
\end{theorem}

\begin{proof}
The idea is to use the $\theta$-control to modify the equivalence class $\hat \sim$ in a coarser partition of $\sim$.  Let $\gamma \in P_m$ with kernel graph $\widehat G_\gamma = (\widehat V_\gamma,\widehat E_\gamma)$.
The {\em terminal color} of the edge $f = [y_0,\pi,y_k]  \in \widehat E_\gamma$, $\pi = (y_0,j_1,y_1,\ldots,j_k, y_k)$ is the integer $j_k$. For $\gamma, \gamma'$ in $P_m$, We then say that $\gamma \check \sim \gamma'$ if $\gamma \sim \gamma'$ and the terminal colors of their kernel graphs coincide.

Let $\check \cP_m(v,e_1)$ be the set of equivalence classes for $\check \sim$ with $v$ vertices and $e_1$ edges of multiplicity one. Since there are at most $2d$ choices a terminal color, the proof of Lemma \ref{le:Plclass2} gives 
$$
| \check \cP_m(v,e_1) | \leq (2d)^{3\chi + 2} | \cP_m(v,e_1) |,
$$
with $\chi = m/2 +e_1/2 -v$. For $\gamma \in \check \cP_m(v,e_1)$, we define 
$$
\delta(\gamma) = \frac{w(\gamma)}{\max_{\gamma' \in  \cP_m(v,e_1)} |w(\gamma')|} \in [-1,1],
$$
(if the denominator is $0$, we set $\delta(\gamma) = 1$).

The proof of Lemma \ref{le:sumagamma} implies that for any  $\gamma \in \widehat \cP_m(v,e_1)$, we have 
$$
\check \rho (\ell,\gamma) = \NRM{ \sum_{ \gamma' : \gamma' \check \sim \gamma}  a(\ell,\gamma')  \delta(\gamma')}  \leq N^{v} \theta^{\hat e}\rho^\ell (2d)^{e_1/2} \ell^{18 \chi +11},
$$
where $\hat e$ is the number of edges of $\widehat G_\gamma$ (this corresponds to the number of blocks when we apply Theorem \ref{th:NCCS} - Definition \ref{def:Deltacontrol}). From \eqref{eq:bdhate}, we get 
$$
\check \rho(\ell,\gamma) \leq N^{v} \rho^\ell (2d)^{e_1/2} \theta^2 \ell^{11} (\theta^3 \ell^{18} ) ^{ \chi} .
$$
From \eqref{eq:trBlmf} we then get:
$$
N \ABS{ \dE \SBRA{ \tau (B^{\ell,m} )}} \leq    \sum_{v =1}^{m/2} \sum_{e_1 = 0}^m | \check \cP_{m} (v,e_1)| \max_{\gamma  \in \check \cP_{m}(v,e_1)}  \check \rho(\ell,\gamma) \max_{\gamma  \in  \cP_{m}(v,e_1)}|w(\gamma)| ,
$$
It remains to repeat the computations below \eqref{eq:trBlmf} with our slightly modified upper bounds $| \check \cP_{m} (v,e_1)|$ and $\check \rho(\ell,\gamma)$. With $r = m/2 - v$, we get this time
$$
\max_{\gamma  \in \check \cP_{m}(v,e_1)}  \check \rho(\ell,\gamma) \max_{\gamma  \in  \cP_{m}(v,e_1)}|w(\gamma)| \leq c e^{m\eta} \theta^2 \ell^{11}  \rho^\ell \Delta_1^ {e_1 } \Delta_2 ^r,
$$
and 
$$
|\check \cP_{m} (v,e_1) |  \max_{\gamma  \in \check \cP_{m}(v,e_1)}  \check \rho(\ell,\gamma)  \max_{\gamma  \in \check \cP_{m}(v,e_1)}  w(\gamma) \leq   (2 d) ^2  m^4 c e^{m \eta}   \theta^2 \ell^{11}  \rho^\ell  ((2 d)^{3/2} m^3 \Delta_1)^{e_1} ( (2 d)^3 m^6\Delta_2) ^r. 
$$
where 
$$
\Delta_1 = \eta m \sqrt{2d}  \theta^{3/2}  \ell^{9} \AND \Delta_2 = N^{-1} m^2   \theta^3 \ell^{18}.
$$
The conclusion follows.
\end{proof}

\section{Extension to random permutations}

\label{sec:perm}

\subsection{Main technical result: proof of Theorem \ref{th:main3}}
We consider the setting of Section \ref{sec:FK}.  
We fix a sequence $(a_0,a_1, \ldots, a_{2d})$  in $\cA_1 \subset \cB(H)$ satisfying the symmetry condition \eqref{eq:symai}.
 Recall the definition of $A_\star$ in \eqref{eq:defA*} and set 
$
\rho = \| A_\star\|. 
$
We fix an integer $N \geq 3$ and consider $(U_1,\ldots, U_d)$ be independent uniformly distributed permutation matrices in $\dS_N$ and set $U_{i+d}  = U_{i^*} = U_i^*$ for $i \in \INT{d}$. We then consider the operator $A \in \cA_1 \otimes M_N(\dC)$ as in \eqref{eq:defA} and the corresponding operators $B^{(k,m)}$ defined by \eqref{eq:defBkmbis} (we omit the dependency in $N$). We denote by $\tau = \tau_1 \otimes \tr_N$ the trace on $\cA_1 \otimes M_N(\dC)$ with $\tr_N = \frac 1 {N} \tr$. 

Recall the definition of the projection operator $\Pi = \Pi_N$ defined in \eqref{eq:defPiN}. We have that $A \Pi = \Pi A = \Pi A \Pi$. Hence, for $\ell, p \geq 1$ integers,  we have $(A \Pi)^{\ell p} = (A^{\ell} \Pi)^p$.  From Lemma \ref{th:powerIB} and the triangle inequality, we deduce, if $p$ is even,
\begin{equation}\label{eq:normAltrP}
 \| A \Pi\|_{\ell p }^{\ell}  \leq \sum_{m=0}^{\ell} \|  B^{(\ell,m)} \Pi \|_p.
\end{equation}
For $m=0$, from \eqref{eq:defBkmbis}, we find
\begin{equation}\label{eq:Bl0P}
\| B^{(\ell,0)} \Pi \|_p  \leq  \rho^\ell.
\end{equation}

We are able to upper bound $\|  B^{(\ell,m)} \Pi \|_p$ and $ \| A \Pi\|_{\ell p}^{\ell}$ on an event of large probability $E$ which we will define in the next subsection. The main result of this section is the following upper bound.

\begin{theorem}\label{th:FKBlmP}
Fix $0 < \veps < 1$. Assume $d \leq \log N / (20 \log \log N)$. There exists $c(\veps) >0$ and an event $E = E_N( \veps, U_1,\ldots,U_d)$ of probability at least $1- c(\veps) N^{-1+\veps}$ such that the following holds for some numerical constant $c >0$. 
Assume that $(a_0,a_1,\ldots,a_{2d})$ satisfy  the symmetry condition \eqref{eq:symai}.  
Let 
\begin{equation}\label{eq:defh}
\ell = \lfloor \frac{ \veps}{5} \log _{2d -1} (N) \rfloor 
\end{equation}
Then, for all $1 \leq m \leq \ell$ and even integer $2 \leq q \leq \log N / (20 \log \log N)$, 
$$
\dE \SBRA{\|  B^{(\ell,m)} \Pi  \|_{q} ^{q}  \IND_{E} }\leq (\ell+1)^{6q}   d^{5q}  ( c qm)^{(4d +2)q} \rho^{q \ell} ,
$$
and, 
$$
\dE \SBRA{ \| A \Pi\|_{q\ell}^{q\ell } \IND_E } \leq   d^{5q}  ( c q \ell)^{(4d +10)q} \rho^{q \ell}.
$$
\end{theorem}

Theorem \ref{th:main3} is an immediate corollary of Theorem  \ref{th:FKBlmP}. 

\begin{proof}[Proof of Theorem \ref{th:main3}]
Fix $0 < \veps < 1$. We wake $\ell$ as in \eqref{eq:defh}, $q = 2  \lfloor \log N / (40 \log \log N)\rfloor$ so that $q \ell = ( 1+o(1)) (\log N)^2 / (100 \veps (\log (2d -1) \log \log N)$ where $o(1)$ depends on $\veps$.  For some $ t  > 0$ to be defined later, we set $\theta = \exp ( 2t d \log \log N / \ell) = 1 + 2t (1 + o(1)) d \log \log N / \ell$.  Let $E = E(\veps)$ be the event of Theorem \ref{th:FKBlmP}. We write
$
\dP ( \| A \Pi\|_{q\ell} \geq \theta \rho ) \leq \dP ( E^c) + \dP ( \| A \Pi\|_{q\ell} \IND_{E} \geq \theta \rho ). 
$
From Markov inequality and Theorem \ref{th:FKBlmP}, we have 
$$
\dP ( \| A \Pi\|_{q\ell} \IND_{E} \geq \theta \rho ) \leq \frac{\dE \SBRA{ \| A \Pi\|_{q\ell}^{q\ell } \IND_E } }{ (\theta \rho ) ^{q\ell}} \leq  d^{5q}  ( c q \ell)^{(4d +10)q} e^{ - 2t d q \log \log N } .
$$
Since $c q\ell \leq (\log N)^2$ for $N$ large enough and $2 \leq d \leq \log N$, we get, for $ t  = 30$, 
$$
d^{5q}  ( c q \ell)^{(4d +10)q} e^{ - 2t d q \log \log N }  \leq e^{(dq \log \log N) ( 5/d  + 8 + 20 / d - 2 t )}  \leq e^{(dq \log \log N) ( 5/2  +  8 + 10 - 2 t )}\leq  e^{ - dq t \log \log N},
$$ 
for $N$ large enough. For our choice of $q$, the latter is at most $N^{-td (1+o(1))/20 } \leq N^{-t ( 1+ o(1)) / 10}  =  N^{-3( 1 + o(1))}$.  Adjusting the constant $c(\veps)$, the conclusion follows.
\end{proof}

\subsection{Path decomposition}
\label{subsec:PDP}

In this subsection, we define the event $E$ introduced in Theorem \ref{th:FKBlmP} and give an expression for $\| B^{(\ell,m)} \Pi \|_p$ when this event holds. We denote by $\sigma = (\sigma_1,\ldots,\sigma_d)$ the family of independent random uniform permutations in $\dS_N$ such that for all $i \in \INT{2d}$,
$$
(U_i)_{x,y} = \IND (\sigma_i (x) = y),
$$
with $\sigma_i = \sigma_{i^*}^{-1}$ for all $i \in \{d+1,\ldots,2d \}$. For $x,y \in \INT{N}$, the computation leading to \eqref{eq:tr1} gives 
\begin{equation}\label{eq:exBlmxy}
B^{(\ell,m)}_{x,y} =  \sum_{\gamma \in Q_{m,x,y}} a(\ell,\gamma) \prod_{t=1}^{m} (U_{i_t})_{x_{t-1} x_t} ,
\end{equation}
where $Q_{m,x,y}$ is the set of $\gamma = (x_0,i_1,\ldots,i_m,x_m)  \in \INT{N} \times \INT{2d} \times \ldots \times \INT{2d} \times \INT{N}$ with $x_0 = x$, $x_m = y$ and $i_{t} \ne i^*_{t+1}$ for all $t \in \INT{m-1}$ (if $m = 1$, this last constraint is empty). The operator $a(\ell,\gamma) \in \cA_1$ was defined in \eqref{eq:defapgamma}. We define $Q_m$ as the disjoint union of all $Q_{m,x,y}$, $x,y \in \INT{N}$ (in Section \ref{sec:FK}, $P_m$ is the disjoint union of all $Q_{m,x,x}$, $x \in \INT{N}$). Elements of $Q_m$ will be called \emph{paths}.

We are going to restrict the sum of paths in \eqref{eq:exBlmxy} to a subset of $Q_{m,x,y}$ (called $F_{m,x,y }$ below). To each element $\gamma \in Q_m$, we associate its colored graph $G_\gamma = (V_\gamma,E_\gamma)$ defined below Definition \ref{def:colorgraph} (there were defined for $\gamma \in P_m$ but the definition extends to $Q_m$). Finally, if $G = (V,E)$ is a graph, $x \in V$ is a vertex, and $h \geq 0$ real, we denote by $(G,x)_h$ the subgraph of $G$ spanned by vertices at graph distance at most $h$ from $x$.

\begin{definition}[Tangles]\label{def2}
Let $h \geq 0$. A graph $G$ is {\em tangle-free} if it contains at most one cycle, $G$ is {\em $h$-tangle-free} if for any vertex $x$, $(G,x)_{h}$ 
contains at most one cycle. We say that $\gamma  = (x_0,i_1,\ldots,x_m)\in Q_m$ is tangle-free
if $G_{\gamma}$ is. We say it is tangled otherwise.
\end{definition}

We consider the colored graph $G_\sigma = (\INT{N},E_\sigma)$ with colored edges $E _\sigma = \{ [x,i,\sigma_i(x)] :  x \in \INT{N}, i \in \INT{2d}\}$. In other words, $G_\sigma$ is the Schreier graph associated to the permutations $\sigma = (\sigma_i)_{i \in \INT{d}}$. Then \cite[Lemma 23]{MR4024563} implies the following.  
\begin{lemma}\label{le:probtf}
For any $0 < \veps < 1$, there exists a constant $c = c(\veps)$ such that if $d \geq 2$ and $h =  ( \veps / 5)\log _{2d -1} (N) $, then $G_\sigma$ is $h$-tangled-free with probability at least $1  - c n^{1-\veps}$.   
\end{lemma}

Let us now define an operator $\widetilde B^{(\ell,m)} \in \cA_1 \otimes M_N(\dC)$ for $1 \leq m \leq \ell$.  For all $x,y \in \INT{N}$, let 
\begin{equation}\label{eq:defwtBlm}
\widetilde B^{(\ell,m)}_{x,y} =  \sum_{\gamma \in F_{m,x,y}} a(\ell,\gamma) \prod_{t=1}^{m} (U_{i_t})_{x_{t-1} x_t} ,
\end{equation}
where 
$
F_{m,x,y} = \{ \gamma \in Q_{m,x,y} : \gamma \hbox{ tangle-free}\}.
$
Note that $F_{m,x,y}$ is deterministic (it does not depend on the permutations $\sigma$). We denote by $F_m \subset Q_m$ the disjoint union of all $F_{m,x,y}$, $x,y \in \INT{N}$. For $m=0$, we set $\widetilde B^{(\ell,0)} = B^{(\ell,0)} = a(\ell) \otimes 1_N$.

If $G_\sigma$ is $\ell$-tangle-free then 
\begin{equation}\label{eq:wttoP}
\widetilde B^{(\ell,m)}  =  B^{(\ell,m)}.
\end{equation}
Indeed, for any $\gamma \in Q_m \backslash F_m$, then 
$$\prod_{t=1}^{m} (U_{i_t})_{x_{t-1} x_t} = \prod_{t=1}^{m} \IND ( \sigma_{i_t}(x_{t-1}) = x_t) = 0,$$
because $G_\gamma$ has two cycles in a ball of radius $h$. Since  $G_\sigma$ is $\ell$-tangle-free $G_\gamma$ cannot be a subgraph of $G_\sigma$ and thus  $\prod_{t=1}^{m} \IND ( \sigma_{i_t}(x_{t-1}) = x_t) = 0$.

Our next goal is to obtain an expression for $\widetilde B^{(\ell,m)} \Pi$. For $i \in \INT{2d}$, we define 
\begin{equation}\label{eq:defS}
\underline U_{i }  = U_{i } (1_N - \IND_N \otimes \IND_N)  = U_{i}- \IND_N \otimes \IND_N ,
\end{equation}
where $\IND_N$ is the unit vector with all coordinates equal to $1/\sqrt N$. Note that $\underline U_i$ is the orthogonal projection of $U_i$ onto $\IND_N^\perp$ and that $ \dE \underline U_i = 0$. 

We introduce the operator $\underline B^{(\ell,m)}$ in $\cA_1 \otimes M_N(\dC)$ defined by, for all $x,y \in \INT{N}$,
\begin{equation}\label{eq:defdefub}
(\underline B^{(\ell,m)})_{x,y} = \sum_{\gamma \in F_{m,x,y}}  a(\ell,\gamma)  \prod_{t =1}^m ( \underline U_{i_t} )_{x_{t-1} x_{t} }.
\end{equation}
For $m=0$, we set $\underline B^{(\ell,0)} =  B^{(\ell,0)} = a(\ell) \otimes 1_N$.
Let $b_t,c_t$, $t \in \INT{m}$, be elements in a ring. We have the identity
$$
\prod_{t=1}^m b_t = \prod_{t=1}^m c_t + \sum_{k=1}^{m}\PAR{ \prod_{t=1}^{k-1} b_t } ( b_k - c_k) \PAR{\prod_{t = k+1}^{m} c_t},
$$
where by convention, a product over an empty set is equal to the unit of the ring. We apply this last identity to each $\gamma \in F_{m,x,y}$ in \eqref{eq:defwtBlm} with $b_t = (U_{i_t})_{x_{t-1} x_t}$, $c_t = (\underline U_{i_t})_{x_{t-1} x_t}$ and $b_t - c_t = 1/N$. We get 
\begin{equation*}
(\widetilde B^{(\ell,m)})_{x,y} = ( \underline B^{(\ell,m)})_{x,y} + \sum_{k=1}^m \sum_{\gamma \in F_{m,x,y}}  a(\ell,\gamma) \PAR{ \prod_{t=1}^{k-1} ( \underline U_{i_t} )_{x_{t-1} x_{t} } } \PAR{\frac 1 N } \PAR{\prod_{t = k+1}^{m} ( U_{i_t} )_{x_{t-1} x_{t} }}.
\end{equation*}
Recall from \eqref{eq:defapgamma} that $a(\ell,\gamma) =  (A_\star^\ell)_{g(\gamma) , \o}$, where $g(\gamma) = g_{i_m} \cdots g_{i_1} \in S_m$ is an element of $\dF_d$ of reduced length $m$. If $F_{g,x,y} = \{ \vec x = (x_0,\ldots, x_m) \in \INT{N}^{m+1} : (x_0,i_1,\ldots,i_{m},x_{m}) \in F_{m,x,y}\}$, we get 
\begin{equation}\label{eq:decomp1}
\widetilde B^{(\ell,m)} =  \underline B^{(\ell,m)} + \sum_{k=1}^m \sum_{g \in S_m} (A_\star^\ell)_{g\o} \otimes C^{(k)}_g,
\end{equation}
where  $C^{(k)}_g$ in $M_N(\dC)$ is defined for $x,y$ by 
$$
(C^{(k)}_g)_{x,y} =  \sum_{\vec x \in F_{g,x,y}} \PAR{ \prod_{t=1}^{k-1} ( \underline U_{i_t} )_{x_{t-1} x_{t} } } \PAR{\frac 1 N } \PAR{\prod_{t = k+1}^{m} ( U_{i_t} )_{x_{t-1} x_{t} }}.
$$
 For fixed $g = g_{i_m} \cdots g_{i_1} \in S_m$ and $k \in \INT{m}$, we consider the matrices in $M_N(\dC)$, whose entry $x,y$ are
\begin{align*}
(\underline U_{g}^{(k)} )_{x,y}= \sum_{\vec x  \in F^{k,1}_{g,x,y}}\prod_{t=1}^{k-1} ( \underline U_{i_t} )_{x_{t-1} x_{t} }  \AND (\widetilde U_{g}^{(k)} )_{x,y}= \sum_{\vec x \in {F}^{k,2}_{g,x,y}}\prod_{t=1}^{m-k} (  U_{i_t} )_{x_{t-1} x_{t} }, 
\end{align*}
where $F^{k,1}_{g,x,y} = \{ \vec x = (x_0,\ldots ,x_{k-1} )  : (x_0,i_1,\ldots,i_{k-1},x_{k-1}) \in F_{k-1,x,y} \}$, ${F}^{k,2}_{g,x,y} = \{ \vec x  = (x_0,\ldots,x_{m-k}) : (x_0,i_{k+1},\ldots,i_{m},x_{m-k}) \in F_{m-k,x,y}  \}$. For $\veps = 1,2$, we denote by $F^{k,\veps}_g$ the disjoint union of all $F^{k,\veps}_{g,x,y}$, $x,y \in \INT{N}$. In words, $F_g^{k,\veps}$ are tangle-free paths compatible with the beginning and end of $g$. By the definition of matrix products, 
$$
(C^{(k)}_g) = \underline U_{g}^{(k)} \IND_N \otimes \IND_N \widetilde U_g^{(k)} - \frac {D^{(k)}_g} N 
$$
 where 
 $$
( D^{(k)}_g)_{x,y} = \sum_{\vec x  \in T^{k}_{g,x,y}}  \PAR{ \prod_{t=1}^{k-1} ( \underline U_{i_t} )_{x_{t-1} x_{t} } } \PAR{\prod_{t = k+1}^{m} ( U_{i_t} )_{x_{t-1} x_{t} }},
 $$
and $T^{k}_{g,x,y}$ is the set of $(x_0,\ldots,x_m)$ such that, with $\gamma = (x_0,i_1,\ldots,x_m) = (\gamma_1,i_k,\gamma_2)$, $\gamma_\veps \in F_g^{k,\veps}$, $\veps = 1,2$ and $\gamma \in Q_{m,x,y} \backslash F_{g,x,y}$. In words, $T_{g,x,y}^{k}$ are tangled paths from $x$ to $y$, which are a union of tangled-free paths and which are compatible with $g$, see Figure \ref{fig:Gamma3} for the $3$ topological possibilities.

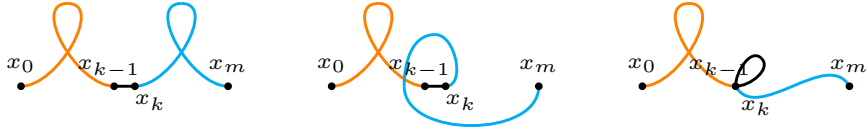
\begin{figure}[htb]
\begin{center}  
\resizebox{12cm}{!}{
\begin{tikzpicture}[main node/.style={circle,fill , text = black, thick}]

\draw[orange , -,thick] (0.0,0)  [out = 0, in = 0] to (0.45,0.8)  ; 
\draw[orange , -,thick] (0.45,0.8) [out = 180 , in = 180]   to (0.9,0)   ; 
\draw[ -,thick] (0.9,0) to (1.1,0)  ; 
\draw[ cyan,  -,thick] (1.1,0) [out = 0, in = 0]  to (1.55,0.8)  ; 
\draw[ cyan , -,thick] (1.55,0.8) [out = 180 , in = 180]  to (2,0) ;  

\node[above]  at (0,0) (1) {\tiny{$x_0$}} ;
\node[above]  at (2,0) (2)   {\tiny{$x_m$}} ;
\node[above]  at (0.85,-0.05) (3)  {\tiny{$ x_{k-1}$}} ;
\node[below]  at (1.25,0.05) (4) {\tiny{$x_{k}$}} ;

\draw [fill] (0,0) circle [radius=0.03] ;
\draw [fill] (0.9,0) circle [radius=0.03] ;
\draw [fill] (1.1,0) circle [radius=0.03] ;
\draw [fill] (2,0) circle [radius=0.03] ;

\draw[orange , -,thick] (3.0,0)  [out = 0, in = 0] to (3.45,0.8)  ; 
\draw[orange , -,thick] (3.45,0.8) [out = 180 , in = 180]   to (3.9,0)   ; 
\draw[ -,thick] (3.9,0) to (4.1,0)  ; 
\draw[ cyan,  -,thick] (4.1,0) [out = 0, in = 0]  to (4,0.5)  ; 
\draw[ cyan , -,thick] (4,0.5) [out = 180 , in = 90]  to (3.7,0) ;  
\draw[ cyan , -,thick] (3.7,0) [out = -90 , in = -90]  to (5,0) ;  

\node[above]  at (3,0) {\tiny{$x_0$}} ;
\node[above]  at (5,0)  {\tiny{$x_m$}} ;
\node[above]  at (3.85,-0.05)  {\tiny{$ x_{k-1}$}} ;
\node[below]  at (4.25,0.05)  {\tiny{$x_{k}$}} ;

\draw [fill] (3,0) circle [radius=0.03] ;
\draw [fill] (3.9,0) circle [radius=0.03] ;
\draw [fill] (4.1,0) circle [radius=0.03] ;
\draw [fill] (5,0) circle [radius=0.03] ;

\draw[orange , -,thick] (6.0,0)  [out = 0, in = 0] to (6.45,0.8)  ; 
\draw[orange , -,thick] (6.45,0.8) [out = 180 , in = 180]   to (6.9,0)   ; 
\draw[ -,thick] (6.9,0)[out = 0, in = -45]  to (7.15,0.27)  ;
 \draw[ -,thick] (7.15,0.27)[out = 135, in = 90]  to (6.9,0)  ;
\draw[ cyan,  -,thick] (6.9,0) [out = -60, in = -240]  to (8,0) ;  

\node[above]  at (6,0)  {\tiny{$x_0$}} ;
\node[above]  at (8,0)  {\tiny{$x_m$}} ;
\node[above]  at (6.75,-0.05)  {\tiny{$ x_{k-1}$}} ;
\node[below]  at (7.1,0.0)  {\tiny{$x_{k}$}} ;

\draw [fill] (6,0) circle [radius=0.03] ;
\draw [fill] (6.9,0) circle [radius=0.03] ;
\draw [fill] (8,0) circle [radius=0.03] ;

\end{tikzpicture}}
\caption{Tangle-free paths whose union is tangled.} \label{fig:Gamma3}
\end{center}\end{figure}

\vspace{-10pt}

We then use the definition of tangles again. If $G_\sigma$ is $\ell$-tangle free then 
$$
 \widetilde U_{g}^{(k)}  = U_{g}^{(k)}, 
\, \hbox{ with } \; ( U_{g}^{(k)} )_{x,y}=  \sum_{(x_1,\ldots,x_{m-1})\in \INT{N}^{m-1} }\prod_{t=1}^{k-1} (  U_{i_t} )_{x_{t-1} x_{t} } = ( U_{i_{k+1}} \cdots U_{i_m} )_{x,y} , 
$$
with $x_0 = x$, $x_m= y$. 
We have $U_g^{(k)} \IND_N = (U_g^{(k)})^* \IND_N = \IND_N$. We deduce that if $G_\sigma$ is $\ell$-tangle free then 
$$
C^{(k)}_g ( 1- \IND_N \otimes \IND_N ) = - \frac {D^{(k)}_g} N   ( 1- \IND_N \otimes \IND_N ).
$$

Since $\Pi = 1_{H} \otimes (1- \IND_N \otimes \IND_N )$, from \eqref{eq:decomp1}, we arrive at the following identity, 
$$
B^{(\ell,m)} \Pi = \underline B^{(\ell,m)} \Pi - \frac 1 N \sum_{k=1}^m R_k^{(\ell,m)}  \Pi, 
$$
with $R_k^{(\ell,m)} \in \cA_1 \otimes M_N(\dC)$ is defined for all $x,y \in \INT{N}$ by 
\begin{equation}\label{eq:defRk}
(R_k^{(\ell,m)})_{x,y} = \sum_{\gamma \in T^k_{m,x,y}} a(\ell,\gamma) \PAR{ \prod_{t=1}^{k-1} ( \underline U_{i_t} )_{x_{t-1} x_{t} } } \PAR{\prod_{t = k+1}^{m} ( U_{i_t} )_{x_{t-1} x_{t} }},
\end{equation}
and 
$
T^k_{m,x,y} = \{ \gamma = (x_0,i_1,\ldots,i_m,x_m) = (\gamma_1,i_k,\gamma_2) \in Q_{m,x,y} \backslash F_{m,x,y} : \gamma_1 \in F_{k-1}, \gamma_2 \in F_{k-m}\}.
$

Using the triangle inequality and the fact that $\Pi$ is a contraction, we get the following lemma which is the conclusion of this subsection.
\begin{lemma}\label{le:decomp}
If $G_\sigma$ is the $\ell$-tangle-free, for any $1 \leq m \leq \ell$ and $p \geq 1$, we have 
$$
\| B^{(\ell,m)} \Pi \|_p  \leq \| \underline B^{(\ell,m)} \|_p  + \frac 1 N \sum_{k=1}^m \| R_k^{(\ell,m)}  \|_p.$$
\end{lemma}

In view of Lemma \ref{le:decomp}, we will now upper bound $\dE \| \uB^{(\ell,m)}\|^p_p$ and  $\dE \| R_k^{(\ell,m)}\|^p_p$ for $p$ large enough. This can be done by adapting to random uniform permutations the strategy followed for the proof of Theorem \ref{th:FKBlm} in Subsection \ref{subsec:th7} for random Haar unitaries. 

\subsection{Norm of $\underline B^{(\ell,m)}$}

The main result of this subsection is the following bound.

\begin{proposition}\label{prop:FBB}
Assume that $(a_0,a_1,\ldots,a_{2d})$ satisfy  the symmetry condition \eqref{eq:symai} and $N \geq 3$. There exists a numerical constant $c > 0$ such that if $d \leq \log N / (20 \log \log N)$ and $1 \leq m \leq \ell \leq  \log N$. Then for even integer $2 \leq q \leq \log N / (20 \log \log N)$, 
$$
\dE \| \underline B^{(\ell,m)} \|_{q}^{q} \leq   \ell^{12q}   d^{5q}  ( c qm)^{(4d +1)q} \rho^{q \ell}.
$$ 
\end{proposition}

In a nutshell, the proof of Theorem \ref{th:FKBlm} works in the context of Proposition \ref{prop:FBB} up to a few technical issues, most of which have been dealt with in \cite{MR4203039,MR4024563}. 

The symmetry condition \eqref{eq:symai} guarantees that $\underline B^{(\ell,m)}$ is self-adjoint. Indeed, if $\gamma  =(x_0,i_1,\ldots,x_m) \in F_{m,x,y}$ and $\gamma^*$ is defined as $\gamma^* = (x_m,i_m^*,\ldots,x_1)$ then $\gamma^* \in F_{m,y,x}$ and $a(\ell,\gamma^*) = a(\ell,\gamma)^*$, where $a(\ell,\gamma)$ was defined in \eqref{eq:defapgamma}.  Hence, for even integer $q \geq 1$,
$$
 \| \underline B^{(\ell,m)} \|_{q}^{q} = \tau \PAR{ ( \underline B^{(\ell,m)})^{q}} .
$$ 
Recalling that $\tau = \tau_1 \otimes \tr_N$, we next expanding $\tr_N$. The computation leading to \eqref{eq:tr1} is modified by considering $q$ paths. More precisely,  from \eqref{eq:defdefub}, we get, with the convention that $y_{q+1} = y_1$ at the first line,
\begin{eqnarray*}
N \tau \PAR{ ( \underline B^{(\ell,m)})^{q}} 
 & =  &   \tau_1 \PAR{ \sum_{( y_1,\ldots, y_{q}) \in \INT{N}^{q}}   \prod_{p =1}^{q} ( \underline B^{(\ell,m)})_{y_{p} y_{p+1}} } \\
 & = &  \tau_1 \PAR{ \sum_{\gamma \in W_{m,q}} \prod_{p=1}^{q}  a(\ell,\gamma_p ) \prod_{p=1}^{q} \prod_{t=1}^{m} (\underline U_{i_{p,t}} )_{x_{p,t-1},x_{p,t}}   },
 \end{eqnarray*}
where the above non-commutative product over elements in $\cA_1$ is taken from left to right  and $W_{m,q}$ is the set of $\gamma = (\gamma_1,\ldots, \gamma_q)$ with $\gamma_p = (x_{p,0},i_{p,1},\ldots,i_{p,m},x_{p,m}) \in F_m$ and satisfying the boundary conditions: for all $p \in \INT{q}$
\begin{equation}\label{eq:boundary}
x_{{p+1},0} = x_{p,m} 
\end{equation}
with the convention that $\gamma_{q+1} = \gamma_1$. Taking expectation, we obtain the following analog of \eqref{eq:tr1}
\begin{equation}\label{eq:tr1P}
N \ABS{\dE  \tau \PAR{ ( \underline B^{(\ell,m)})^{q}} } \leq \NRM{\sum_{\gamma \in W_{m,q}}   a(\ell,\gamma)w(\gamma) },
\end{equation}
where for $\gamma \in W_{m,q}$ (or more generally in $Q_m^q$),
\begin{equation}\label{eq:defapgammaP}
a(\ell,\gamma) = \prod_{p=1}^{q}  a(\ell,\gamma_p ) = \prod_{p=1}^q (A_{\star}^\ell)_{g(\gamma_p),\o}  \AND w(\gamma) = \dE \prod_{p=1}^{q} \prod_{t=1}^{m} (\underline U_{i_{p,t}} )_{x_{p,t-1},x_{p,t}}.
\end{equation}
In the above, $g(\gamma_p) = g_{i_{p,m}} \cdots g_{i_{p,1}}$ is a reduced word of length $m$.

The equivalence class $\gamma \sim \gamma'$  on $P_m$ introduced above Definition \ref{def:colorgraph} extends to the larger set of $q$ paths $Q_m^q \supset W_{m,q}$. The colored graph $G_\gamma = (V_\gamma,E_\gamma)$ is also defined similarly. The kernel graph $\widehat G_\gamma$ is defined precisely as above Definition \ref{def:profile} expect that its vertex set is 
$$
\widehat V_\gamma = V_{\geq 3} \cup  \bigcup_{p=1}^{q} \{ x_{p,0} \},
$$
where $V_{\geq 3}$ is the set of vertices of $G_\gamma$ with degree at least $3$. The profile of a path and the finer equivalence class $\gamma \hat \sim \gamma'$ are defined without any change. Beside Lemma \ref{le:newclass} continues to hold: if $\gamma \hat \sim \gamma'$ then  $
w(\gamma) = w(\gamma')$. Indeed in the proof of Lemma \ref{le:newclass}, we have only used the invariance $U $ has the same law as $S U T$ for all permutation matrices $S, T$ (this property obviously continues to hold for uniform permutations $U$ and $\underline U = U - \IND_N \otimes \IND_N$).

Now, we sort the paths in terms of the number of edges and vertices of  $\gamma \in G_\gamma$. Assume that $G_\gamma$ has $v$ vertices, $e$ edges, and $e_1$ edges of multiplicity one. The sum of edge multiplicities is $q m$.
Arguing as in \eqref{eq:bdee1}, we get 
\begin{equation}\label{eq:bdee1P}
e \leq (q m + e_1 )/2. 
\end{equation}
Moreover, from the boundary condition \eqref{eq:boundary} $G_\gamma$ is a connected graph. This implies that
\begin{equation}\label{eq:genusP}
g = e - v +1 \geq 0.
\end{equation}
Note that here the notion of genus $g$ adapted to $W_{m,q}$ differs from the genus $\chi$ defined in \eqref{eq:genus}. For integers $v, e_1$ such that \eqref{eq:genusP} holds. We define $\cW_{m,q} (v,e)$ and $\widehat \cW_{m,q} (v,e)$ as the set of equivalence classes with $v$ vertices and $e$ edges for the equivalence relations $\sim$ and $\hat \sim$ respectively. 

\begin{lemma}\label{le:PlclassP}
With $g  = e - v +1 $, we have, 
$$
|\cW_{m,q} (v,e) |\leq (qm)  ^{ 3q g +4 } \AND |\widehat \cW_{m,q} (v,e) | \leq ((2d) m^{d})^{6g +4q-6} |\cW_{m,q} (v,e) |.
$$
\end{lemma}
\begin{proof}
Up to minor modifications, the first statement is essentially \cite[Lemma 25]{MR4024563}. We start by repeating the proof of Lemma \ref{le:Plclass}. Recall the definition of canonical paths in Lemma \ref{le:Plclass}. The goal is to build an encoding of the canonical paths.  Set $m_0 = m-1$. We equip the set $T = \INT{q} \times \INT{m_0}$ with the lexicographic order (that is $(p,t)\preceq (p+1,t')$ and 
$(p,t)\preceq(p,t+1)$).  We think of an element $s = (p,t) \in T$ as a time. We define $(p, t)^-$ as the largest time smaller than $(p, t)$, i.e. $(p, t)^- = (p, t - 1)$ if 
$t\geq 2$, $ (p, 1)^- = (p- 1, m-1) $ if $j \geq 2$ and, by convention, $(1, 1)^- = (1, 0)$.

Let $T_{(1,0)}$ to be the tree with a single vertex $\{1\}$ and no edge. 
As in Lemma \ref{le:Plclass}, for $s  = (p,t) \in T$, we consider $(T_s)$ the non-increasing sequence of trees obtained by adding iteratively the edge $[x_{p,t-1},i_{p,t},x_{p,t}]$ if its addition does not create a cycle in $T_{s^-}$. Tree edges, excess edges, first time, tree time, and important times are defined as in Lemma \ref{le:Plclass}. The number of tree edges is $v-1$, and the number of excess edges is $g$ by construction.

As in Lemma \ref{le:Plclass}, due to the non-backtracking constraint $i_{t+1} \ne i_t^*$, for each $p \in \INT{q}$, the sequence $\gamma_p$ can thus be decomposed as 
the successive repetitions of: $(i)$ a sequence of first times (possibly empty), $(ii)$ an important time, 
$(iii)$ a sequence of tree times (possibly empty).

For each $p \in \INT{q}$, we mark each important time $s = (p,t)$ by the vector $(i_{p,t},x_{p,t}, x_{p,\tau-1},i_{p,\tau})$ where $\tau > t$ is the next time in $\gamma_p$ which is not a tree time if this time exists. Otherwise, $t$ is the last important time of $\gamma_p$, and we set $\tau = m$). We also add a {\em starting mark} for each $p \geq 2$, given by $(x_{p,0},x_{p,\tau-1},i_{p,\tau})$ where $(p,\tau)$ is the smallest important of $\gamma_p$ (if this time exists). For $p=1$, the starting mark is $i_{q,m}$.

As in Lemma \ref{le:Plclass}, the canonical sequence $\gamma$ is 
uniquely determined by starting marks, the positions of the important times, their marks.  However, thanks to the assumption that $\gamma_p$ is tangle-free for each $p \in \INT{q}$, it is possible to reduce the number of important times that have to be considered. For each $p \in \INT{q}$, we partition important times into three categories, {\em short cycling}, 
{\em long cycling} and {\em superfluous} times. For each $p \in \INT{q}$, we consider the first occurrence of a time $(p,t_1)$ 
such that $x_{p,t_1} \in \{ x_{p,0},  \ldots, x_{p,t_1-1} \}$. If such $t_1$ exists, the last important time $(p,t_s) \preceq (p,t_1)$ 
will be called the short cycling time (this important time exists since we have met a cycle).  Let $1 \leq t_0 \leq t_1$ be such that $x_{j,t_0-1} = x_{j,t_1}$. 
By assumption, $C_p = (x_{p,t_0-1},i_{p,t_0},\ldots, x_{p,t_1})$ will be the unique cycle visited by $\gamma_p$.  
We denote by $(p,t_2)$ the next $t_2 >  t_1$ such that $x_{p,t_2}$ in not in $C_p$ (by convention $t_2 = m$ if 
$\gamma_p$ remains in $C_p$). 
We modify the mark of the short cycling time $(p,t_s)$ as  $(i_{p,t_s},x_{p,t_s}, x_{p,t_1-1}, t_2 ,x_{p, \tau-1},i_{p,\tau})$, where 
$(p,\tau) \succeq (p, t_2)$ is the next time in $\gamma_p$, which is not a tree time (if this time exists, otherwise we set $\tau = m$). 
Important times $(p,t)$ with $1 \leq t < t_s$ or $\tau \leq t < m$ are called long cycling times, they receive the usual mark 
$(i_{p,t},x_{p,t},x_{p,\tau-1},i_{p,\tau})$. The other important times are called superfluous. By convention, if there is no short cycling time, 
we call, anyway, the last important time of $\gamma_p$ the short cycling time. We observe that for each $p$, the number of long cycling times 
of $\gamma_p$ is bounded by  $g-1$ (since there is at most one cycle, no edge of $E_\gamma$ can be seen twice 
outside those of $C_p$, the $-1$ coming from the fact the short cycling time is an excess edge).

We have our encoding. We can reconstruct $\gamma$ from the starting marks, the positions of the long cycling 
and the short cycling times, and their marks. For each $p$, there are at most $1$ short cycling time and $ g-1$ long cycling 
times. There are at most $  m^{qg}$ ways to position them. There are at most $(qm)^2$ different possible marks for a
long cycling time (there are most $q m$ edges and $qm$ vertices in $G_\gamma$) and $m (qm)^3$ possible marks for a short cycling time. Finally, there are $ (q m) ^2 $ possibilities for a 
starting mark. We deduce that    
$$
| \cW _{m,q} (v,e) | \leq    m^{q g} (  q m )^{2q}   ( q m ) ^{2q (g-1)}(m ( qm)^3 )^{q}. 
$$
We obtain the claimed bound  on $| \cW _{m,q} (v,e) |$. 

For the second statement, arguing as in the proof of Lemma \ref{le:Plclass2}, it suffices to prove that for any $\gamma \in W_{m,q}$ with $\hat v$ vertices and $\hat e$ edges in $\widehat G_\gamma$, we have $\hat e \leq 3g +2q -3$.

The proof is the same as in Lemma \ref{le:Plclass2}. The vertices $x_{p,0} \in \widehat V_\gamma$ have a degree at least $1$, all other elements of $\widehat V_\gamma$ have a degree at least $3$, and vertices of $V_\gamma \backslash \widehat V_\gamma$ have degree $2$. Setting  $u = |V_\gamma \backslash \widehat V_\gamma|$, we deduce that
$$\hat v + u  = v \quad \hbox{ and } \quad 3 (\hat v -q) + 2 u  + q \leq \sum_{x \in V_\gamma} \mathrm{deg}(x) = 2 e.$$ 
Canceling  $u$, we get 
 $
 \hat v \leq 2 e - 2 v +2q.
 $ 
Then, using  \eqref{eq:e-v}, we obtain
\begin{equation}\label{eq:bdhateP}
\hat e = \hat v + (e - v) \le 3 e - 3 v + 2 q  = 3 g + 2 q -3 .
\end{equation}
The conclusion follows.
\end{proof}

We next estimate $w(\gamma)$ defined in \eqref{eq:defapgammaP}. The following lemma is  \cite[Lemma 27]{MR4024563}. For real $x$, set $(x)_+ = \max (x,0)$. 

\begin{lemma}\label{le:isopath}
There exists a numerical constant $c >0$ such that for any $\gamma\in W_{m,q}(v,e)$ and $ q m  \leq \sqrt{ N}$, 
$$
 |w(\gamma) | \leq  c^{q + g} N^{-e}  \eta  ^{  (e_1 - 4 g - 2  q)_+}, 
$$
where $\eta = 3 qm / \sqrt N$, $g = e - v +1$ and $e_1$ is the number of edges of $E_\gamma$ with multiplicity one. 
\end{lemma}

The final lemma is the analog of Lemma  \ref{le:sumagamma}.

\begin{lemma}\label{le:sumagammaP}
Let $\gamma \in \widehat \cW_{m,q}(v,e)$, $g = e -v +1$ with $e_1$ edges of multiplicity one. We have 
$$
\rho (\ell,\gamma) = \NRM{ \sum_{ \gamma' : \gamma' \hat \sim \gamma}  a(\ell,\gamma') }  \leq N^{v} \rho^{ q\ell} (2d)^{e_1/2}  \ell^{18 g +12 q -19}.
$$
\end{lemma}

\begin{proof}
We repeat the proof of Lemma \ref{le:sumagamma}. Consider the kernel graph $\widehat G_\gamma$ of $\gamma$ with $\hat e$ edges. Since for each $p \in \INT{q}$, the paths $\gamma_p$ are non-backtracking, we may decompose it into successive visits of the edges of $\widehat E_\gamma$.  We decompose $\gamma = (\gamma_1,\ldots,\gamma_q)$ as $ (p_{1}, p_{2}, \ldots , p_{r})$ with $$p_t =   (y_{t,0},j_{t,1},y_{t,1},\ldots, j_{t,k_t},y_{t,k_t}),$$  where either (a) $p_{t} $ 
follows an edge of $\widehat E_\gamma$ which is visited for the first or last time, or (b) $p_{t}$ follows a sequence  of edges of 
$\widehat E_\gamma$ which are not visited for the first or last time. Note that a path of type $(b)$ may overlap over several successive $\gamma_p$'s. By construction, there are at most $2\hat e$ subpaths $p_{t}$ of type (a), and thus at most $ \hat e$ subpaths of type (b) (any subpath of type (b) is preceded by a subpath of (a)). Hence, $r \leq 3 \hat e$ as in \eqref{eq:ubr}. We also have 
$
\sum_t k_t = q m.
$

We next decompose the value of $a(\ell,\gamma)$ over each subpath. Using Lemma \ref{le:pathdecomFd}, we may write 
\begin{equation}\label{eq:decompalgammaP}
a(\ell,\gamma) =  \sum_{(\lambda_1, \ldots ,\lambda_r) } \tilde a_{r}(\lambda_r,p_r) \cdots  \tilde a_{2}(\lambda_2,p_2) \tilde a(\lambda_1,p_1),
\end{equation}
where the sum is over all $(\lambda_1,\ldots,\lambda_r)$ such that $\sum_{t} \lambda_t = q \ell$ and, for $p = (y_0,j_1,\ldots,y_k)$: $\tilde a_1 (\lambda,p) = a(\lambda,p)$, for $2 \leq t \leq r$ of type (a)   $\tilde a_t(\lambda,p) = a_{1} (\lambda,p)$ defined by \eqref{eq:defAi}. Finally for $2  \leq t \leq r$ of type (b), $\tilde a_t(\lambda,p)$  is a product of the form $a_ {\delta_s} (l_s,q_s)\ldots a_{\delta_1} (l_1,q_1)$ where $(q_1,\ldots,q_s)$ is a decomposition of $p$, $l_1 + \ldots + l_s = \lambda$ and $a_{\delta_j}(l_j,q_j)$ is either $a(l_j,q_j)$ or $a_1(l_j,q_j)$ (corresponding to the successive $\gamma_p$'s intersecting $p_t$). 
Recall that $\| a(\lambda,p) \| \leq \rho^{\lambda}$ and $\| a_1(\lambda,p) \| \leq \rho^{\lambda}$. By Lemma \ref{le:normCk},
\begin{equation*}\label{eq:defAinrmP} 
\| \tilde a_t (\lambda,p) \| \leq  \rho^{\lambda},
\end{equation*} 
which extends \eqref{eq:defAinrm}. Thus, we can repeat the rest of the proof of Lemma \ref{le:sumagamma}. We find:
$$
\rho (\ell,\gamma) \leq N^v \rho^{\ell} (2d)^{e_1/2} \ell^{6 \hat e -1 }.
$$
The claim of the lemma follows then from \eqref{eq:bdhateP}.
\end{proof}

We are finally ready for the proof of Proposition \ref{prop:FBB}.
\begin{proof}[Proof of Proposition \ref{prop:FBB}]
From \eqref{eq:tr1P}, with $\rho(\ell,\gamma)$ as in Lemma \ref{le:sumagammaP} and $\widehat \cW_{m,q} (v,e)$ as in Lemma \ref{le:PlclassP}, we find 
\begin{equation}\label{eq:trBlmfP}
 \ABS{ \dE \SBRA{ \tau (B^{(\ell,m)} )}} \leq \frac 1 N   \sum_{e =1}^{qm} \sum_{v = 1}^{e+1} | \widehat \cW_{m,q} (v,e)| \max_{\gamma  \in \widehat \cW_{m,q}(v,e)}  \rho(\ell,\gamma) | w(\gamma)| ,
\end{equation}
where we have used \eqref{eq:genusP}. Next, we use Lemma \ref{le:isopath} and  Lemma \ref{le:sumagammaP}, for some numerical constant $c >0$,
\begin{equation*}
\rho(\ell,\gamma) |w(\gamma)|  \leq c^{q+g} N^{ -  e}  \eta^{(e_1-4g -2q)_+}  N^{v} (2d)^{e_1/2} (\ell+1)^{18 g + 12q -19 } \rho^{q\ell} ,
\end{equation*}
with $\eta = 3qm/\sqrt N$ and $g = e - v  +1$. From \eqref{eq:bdee1P}, we have $e_1 \geq 2 e - qm$. Assuming $\veps = \eta \sqrt{2d}  \leq 1$, the above expression is maximized in $e_1$ for 
$$e_1 = ( 4 g + 2q) \vee ( 2e - qm).$$ 
From Lemma \ref{le:PlclassP}, we deduce, with $q' = q+2$,
\begin{align*}
& |\widehat \cW_{m,q} (v,e) |  \max_{\gamma  \in \widehat \cW_{m,q}(v,e)}  \rho(\ell,\gamma) | w(\gamma)| \\
& \quad \leq  N^{1-g} (qm)  ^{ 3q g +4 } ( 2d m^{d})^{6g +4q-6} c^{q+g} (2d)^{2g + q} \veps^{(2e - q'm-4g)_+}  \ell^{18 g + 12q -19 } \rho^{q\ell}.
\end{align*}
Making the change of variable $v \to g = e - v +1$, we find, 
$$
\ABS{ \dE \SBRA{ \tau (B^{(\ell,m)} )}} \leq L \rho^{q\ell} \sum_{e= 1}^{qm} \sum_{g \geq 0} \Delta^g \veps^{(2e - q'm-4g)_+},
$$
with, for some new constant $c>0$,
$$
L =   \ell^{12q}   d^{5q}  ( c qm)^{4dq} \AND \Delta = \frac{ \ell^{18 }d^{8} ( c qm)^{6d + 3q} }{N}.
$$
Summing over $e$, if $\veps^2 \leq 1/2$, we obtain 
\begin{eqnarray*}
 \ABS{ \dE \SBRA{ \tau (B^{(\ell,m)} )}} & \leq & L  \rho^{q\ell} \sum_{g \geq 0} \Delta^g   \PAR{ qm + \sum_{k\geq 1} \veps^{2k} }, \\
& \leq &  (qm +1)L  \rho^{q\ell} \sum_{g \geq 0} \Delta^g. 
\end{eqnarray*}
Finally, if in addition $\Delta \leq 1/2$, we get 
$$
 \ABS{ \dE \SBRA{ \tau (B^{(\ell,m)} )}} \leq 2 (qm +1)L  \rho^{q\ell}.
$$
If $N$ is large enough and $\max(\ell,q,d) \leq  \log N / (20 \log \log N)$ then it is immediate to check that $\veps^2 \leq 1/2$ and $\Delta \leq 1/2$. By adjusting the constant $c >0$, Proposition \ref{prop:FBB} follows.
\end{proof}

\subsection{Norm of $R_k^{(\ell,m)}$}

In this subsection, we prove a rough estimate on the norm of $R_k^{(\ell,m)} $. 

\begin{proposition}\label{prop:FBR}
Assume that $N \geq 3$. There exists a numerical constant $c > 0$ such that for any $1 \leq k \leq  m \leq \ell \leq  \log N$ and any even integer $2 \leq q \leq \log N / (20 \log \log N)$, 
$$
\dE \| R_k^{(\ell,m)} \|_{q}^{q} \leq  (c qm)^{17 q}  (2d-1)^{qm} \rho ^{q \ell}.
$$ 
\end{proposition}

The proof is very similar to the proof of Proposition \ref{prop:FBB}. It will, however, be simpler since we do not need a sharp estimate (the factor $(2d -1)^{qm} \rho^{q\ell}$ is sub-optimal). Let $q$ be an even integer. The proof is a repetition of a similar computation found in \cite{MR4203039,MR4024563}. For the sake of completeness, we include a proof.

Expanding $\tr_N$, we may write: 
\begin{eqnarray*}
N \| R_k^{(\ell,m)} \|_{q}^{q}  = N \tau \PAR{ \PAR{ R_k^{(\ell,m)} (R_k^{(\ell,m)})^* }^{q/2}}  = \tau_1 \PAR{ \sum_{( y_1,\ldots, y_{q}) \in \INT{N}^{q}}   \prod_{p =1}^{q} ( (R_k^{(\ell,m)})^{\veps_p} )_{y_{p} y_{p+1}} },
 \end{eqnarray*}
 where  $y_{q+1} = y_1$ and $(R_k^{(\ell,m)})^{\veps_p} = R_k^{(\ell,m)}$  or $(R_k^{(\ell,m)})^{*}$ depending on the parity of $p$.  We next use the expression of $R^{(\ell,m)}_k$ from \eqref{eq:defRk}. We denote by $T^k_{m}$ the disjoint union of $T^k_{m,x,y}$, $x,y \in \INT{N}$ (defined below \eqref{eq:defRk}). We get
 \begin{eqnarray*}
N \| R_k^{(\ell,m)} \|_{q}^{q}  & =  &   \tau_1 \PAR{ \sum_{\gamma \in W^k_{m,q}} \prod_{p=1}^{q}  a(\ell,\gamma_p )^{\veps_p} \prod_{p=1}^{q} \PAR{ \prod_{t=1}^{k-1} (\underline U_{i_{p,t}} )_{x_{p,t-1},x_{p,t}} \prod_{t=k+1}^m (U_{i_{p,t}} )_{x_{p,t-1},x_{p,t}}   }},
 \end{eqnarray*}
where the above non-commutative product over elements in $\cA_1$ is taken from left to right for example and $W^k_{m,q}$ is the set of $\gamma = (\gamma_1,\ldots, \gamma_q)$ with $\gamma_p = (x_{p,0},i_{p,1},\ldots,i_{p,m},x_{p,m}) \in T^k_m$ and satisfying the boundary conditions: for all $2p \in \INT{q}$
\begin{equation}\label{eq:boundaryR}
x_{2p+1,0} = x_{2p,0} \AND x_{{2p-1},m} = x_{2p,m} 
\end{equation}
with the convention that $\gamma_{q+1} = \gamma_1$. We then take expectation and use the rough upper bound  $ \|  a(\ell,\gamma_p \| \leq \rho^\ell$. We arrive at 
\begin{equation}\label{eq:tr1R}
\dE \| R_k^{(\ell,m)} \|_{q}^{q}  \leq \frac {\rho^{q \ell } }{N} \sum_{\gamma \in W^k_{m,q}}   | w_k(\gamma) |,
\end{equation}
where for $\gamma \in W^k_{m,q}$,
\begin{equation}\label{eq:defapgammaR}
w_k(\gamma) = \dE \prod_{p=1}^{q} \PAR{ \prod_{t=1}^{k-1} (\underline U_{i_{p,t}} )_{x_{p,t-1},x_{p,t}} \prod_{t=k+1}^m (U_{i_{p,t}} )_{x_{p,t-1},x_{p,t}}   }.
\end{equation}

To evaluate \eqref{eq:tr1R}, we associate to each $\gamma \in W^k_{m,q}$, the subgraph $G^k_\gamma$ of $G_\gamma$ spanned by colored edges which appear in the product in \eqref{eq:defapgammaR}. More precisely, for each $p \in \INT{q}$, set $\gamma'_p = (x_{p,0}, i_{p,1}, \ldots, x_{p,k-1}) \in F^{k-1}$ and 
$\gamma''_p = (x_{p,k}, \ldots, x_{p,m})\in F^{m -k}$. Then, 
the vertex set of $G^k_\gamma$  is $V_\gamma   =\bigcup_p  V_{\gamma'_p} \cup V_{\gamma''_p}$ and the edge set is 
$E^k_\gamma  =\bigcup_p   E_{\gamma'_p} \cup E_{\gamma''_p}$ (for example, the black edge in  Figure \ref{fig:Gamma3} is not part of $E^k_\gamma$).  Note that the graph $G^k_\gamma$ may not be connected; however, due to the constraint on $\gamma$, it cannot have more vertices than edges. More precisely, let 
$G^k_{\gamma_p}$ denote the colored graph with vertex and edge sets $V_{\gamma'_p} \cup V_{\gamma''_p}$ 
and $E_{\gamma'_p} \cup E_{\gamma''_p}$, by the assumption that $\gamma_p \in T^k_m$, it follows that either 
$G^k_{\gamma_p}$ is a connected graph that contains a cycle, or it has two connected components that both contain a cycle (see Figure \ref{fig:Gamma3}).  
Notably, since $G^k_{\gamma}$ is the union of these graphs, any connected component of $G^k_{\gamma}$  has a cycle, it implies that 
\begin{equation}\label{eq:vehat}
|V_\gamma | \leq |E^k_\gamma |.
\end{equation}

We define $\cW^k_{m,q}(v,e)$ as the set of equivalence classes of $\gamma \in W^k_{m,q}$ such that $G^k_\gamma$ has $v$ vertices and $e$ edges. 

\begin{lemma}\label{le:enumpathR}
With $g = e -v \geq 0$, we have  
$$
| \cW^k _{m,q} (v,e) | \leq   (q m )^{6q g   + 17 q }.
$$
\end{lemma}

\begin{proof}
We argue as in Lemma \ref{le:PlclassP}. Consider the map $\iota : \gamma \mapsto \iota(\gamma)$ on $W_{m,q}^k$ which flips the paths $\gamma_p$ for even $p$: $\gamma_p \mapsto \gamma^*_p = (x_{p,m},i^*_{p,m},\ldots,x_{p,0})$ and leave the paths $\gamma_p$ unchanged for $p$ odd. This map is an involution. We build an encoding for $\check \gamma = \iota(\gamma)$. The advantage of $\check \gamma$ over $\gamma$ is that $\check \gamma$ satisfies the same boundary condition \eqref{eq:boundary} used in the proof of Lemma \ref{le:PlclassP}. The graph $G_{\check \gamma}$ has $v$ vertices and $|E_\gamma | = \check e$ edges. We have $e \leq 
\check e \leq e + q$.
We set 
$$
\hat g = \hat e - v + 1 \leq g + q + 1.
$$ 

We may thus build the same encoding used in Lemma \ref{le:PlclassP} for $\check \gamma$ with important times and the starting marks. The main difference with the situation considered in Lemma \ref{le:PlclassP} is that $\check \gamma_p$ is not tangled-free; however, it is composed of two tangle-free paths $(\check \gamma'_p,\check \gamma''_p)$. We thus introduce short cycling and long cycling times for each tangle-free path $(\check \gamma'_p,\check \gamma''_p)$. Also, for each odd $p$ (resp. even $p$), the time $(p,k)$  (resp. $(p,m-k)$) could be an important time, which we call an extra important time.
 
We can reconstruct $\gamma$ from the starting marks, the possible extra important times, the positions of the long cycling 
and the short cycling times, and their marks. For each $p$, there are at most $2$ short cycling time and $2( \hat g-1)$ long cycling 
times. There are at most $  m^{ 2q \hat g}$ ways to position them. There are at most $(qm)^2$ different possible marks for a
long cycling time and extra important time and $m (qm)^3$ possible marks for a short cycling time. Finally, there are $ (q m) ^2 $ possibilities for a 
starting mark. We deduce that    
$$
| \cW^k _{m,q} (v,e) | \leq    m^{2 q\hat  g} (  q m )^{2q}   ( q m ) ^{4q \hat g}(m ( qm)^3 )^{2q} \leq (qm)^{6q \hat g + 10q}. 
$$
We obtain the claimed bound on $| \cW^k _{m,q} (v,e) |$.  \end{proof}

We need a final lemma on $w_k(\gamma)$ defined in \eqref{eq:defapgammaP}. The following lemma is contained in the proof of \cite[Lemma 27]{MR4024563}. 

\begin{lemma}\label{le:isopathR}
There exists a numerical constant $c >0$ such that for any $\gamma\in W^k_{m,q}(v,e)$ and $ q m  \leq \sqrt{ N}$, with $g = e - v$,
$$
 |w_k(\gamma) | \leq  c^{q + g} N^{-e}. 
$$

\end{lemma}

We are ready to prove Proposition \ref{prop:FBR}.

\begin{proof}[Proof of Proposition \ref{prop:FBR}] Let $q$ be an even integer,  \eqref{eq:vehat} $|V_\gamma| \leq |E^k_\gamma | \leq q m$. We deduce from \eqref{eq:tr1R} that
\begin{eqnarray*}
\dE \| R_k^{(\ell,m)} \|_{q}^{q}  \leq \frac {\rho^{q \ell } }{N} \sum_{e=1}^{qm} \sum_{v =1}^e |\cW^k_{m,q} (v,e) |\max_{ \gamma \in W^k_{m,q} (v,e) } | w_k(\gamma) | \max_{ \gamma  \in W^k_{m,q} (v,e)} C(\gamma), \end{eqnarray*}
where $N(\gamma)$ is the number of $\gamma'$ in $W^k_{m,q}$ such that $\gamma'\sim \gamma$. 
If $\gamma \in \cW_{\ell,m} (v,e)$, the following simple bound holds:
$$
C(\gamma) \leq N^v (2d) (2d-1)^{e-1},
$$
(indeed, $N^v$ bounds the possible choices for the vertices in $V_\gamma$ and $2d(2d-1)^e$ the possible choices for the colors 
of the edges in $E_\gamma$). Hence, using Lemma \ref{le:enumpathR} and Lemma \ref{le:isopathR}, we find, for some numerical constant $c>0$ changing from line to line,
\begin{eqnarray*}
\dE \| R_k^{(\ell,m)} \|_{q}^{q}  & \leq  & \frac {\rho^{q \ell } }{N} \sum_{e=1}^{qm} \sum_{v =1}^e   (q m )^{6q g   + 17 q } c^{q + g} N^{-e} N^v (2d) (2d-1)^{e-1} .
\\
& = &  \frac{\rho^{ q \ell}}{N} (c qm)^{17 q}   \sum_{e=1}^{qm} \sum_{e =1}^{qm}  (2d-1)^e \sum_{ g = 0 } ^{\infty} \PAR{  \frac{ ( c q m) )^{6 q } }{N} }^{g}.
 \end{eqnarray*}
For our choice of $q$ and $m$, the geometric rate of the geometric series is as vanishing in $N$. It follows that for some new constant $ c > 0$, 
$$
\dE \| R_k^{(\ell,m)} \|_{q}^{q}   \leq  N^{-1} (c qm)^{17 q}  \rho^{ q \ell} (2d-1)^{qm}.
$$
It concludes the proof.
\end{proof}

\subsection{Proof of Theorem \ref{th:FKBlmP}}

Fix $0 < \veps < 1$. Let $E$ denote the event that $G_\sigma$ is $\ell$-tangle free with $\ell$ defined in \eqref{eq:defh}. By Lemma \ref{le:probtf}, this event has probability at least $1 - c(\veps) N^{\veps-1}$. Also by Lemma  \ref{le:decomp} and H\"older inequality,
$$
\dE [ \| B^{(\ell,m)} \Pi \|_q^q  \IND_{E} ] \leq (m+1)^{q-1} \PAR{ \dE \| \underline   B^{(\ell,m)} \|_q ^q + \frac{1}{N^q} \sum_{k=1}^m \dE \| R^{(\ell,m)}_k \|_q^q  }. 
$$
Let $2 \leq q \leq \log N / (20 \log \log N)$ be even.  We next use Proposition \ref{prop:FBB} and Proposition \ref{prop:FBR}, adjusting the constant $c >0$, we find
$$
\dE [ \| B^{(\ell,m)} \Pi \|_q^q  \IND_{E} ] \leq \rho^{q\ell} \PAR{ (\ell+1)^{6q}   d^{5q}  ( c qm)^{(4d +2)q}  + \frac{1}{N^{q+1}} \sum_{k=1}^m  (c qm)^{18 q} (2d-1)^{qm} }. 
$$
For our choice of $m\leq \ell$, $(2d-1)^{m} \leq N^{\veps/4}$ and $(cqm)^{18q} \leq N^2$. It follows that for some new $c >0$,
$$
\dE [ \| B^{(\ell,m)} \Pi \|_q^q  \IND_{E} ] \leq \rho^{q\ell}  (\ell+1)^{6q}   d^{5q}  ( c qm)^{(4d +2)q} . 
$$
This is the first claim of Theorem \ref{th:FKBlmP}. For the second claim, from \eqref{eq:normAltrP} and H\"older inequality, we have
$$
\| A \Pi\|_{\ell q }^{q \ell}  \leq (\ell+1)^{q-1} \sum_{m=0}^{\ell} \|  B^{(\ell,m)} \Pi \|^q_q.
$$
For $m=0$, we use \eqref{eq:Bl0P}. For $1 \leq m \leq \ell$, we use the first statement of Theorem \ref{th:FKBlmP}. 

\section{Lower bounds on the operator norm}

\subsection{Haagerup's inequality and exactness of the free group algebra}

\subsubsection{Haagerup's inequality}
\label{subsubsec:haagerup}
In this paragraph, we recall how Haagerup's inequality \cite[Lemma 1.4]{MR520930} classically allows to upper bound the ratio $ \| A _\star \|^p  /  \|A_\star\|_p ^p$. It can be stated as follows. Recall that we equip $M_n(\dC)$ with its normalized tracial state $\tr_n$. 

\begin{lemma}\label{le:Haagerup}
For any operator $T = \sum_{g \in \dF_d} a(g) \otimes \lambda(g)$ in $\cA_1 \otimes \cA_\star$ supported on elements $g \in \dF_d$ of word length $l$, we have
$$
\| T \| \leq (l+1) \sqrt{ \sum_{g \in \dF_d} \| a(g) \|^2}. 
$$
In particular, if $\cA _1 = M_n(\dC)$, we find $\| T \| \leq (l+1) \sqrt n \| T \|_2$. 

\end{lemma}

\begin{proof}The classical Haagerup's inequality \cite[Lemma 1.4]{MR520930} is for $\cA_1 = \dC$ but his proof extends verbatim to this setting. We also refer to Buchholz \cite[Theorem 2.8]{MR1476122} for a refinement.
\end{proof}

As a corollary, we obtain the following. 
\begin{corollary}\label{cor:HaagB}
If $(a_1,\ldots,a_{2d}) \in \cA_1$ then for any even integer $p \geq 2$,
$$
\| A_\star \|^p  \leq p^{3} \sqrt{\sum_{g \in \dF_d}\|(A^{p/2}_\star)_{g \o} \|^2 }. 
$$
In particular, if $\cA_1 = M_n(\dC)$, $\| A_\star \|^p  \leq 2 n p^{3} \|A_\star\|_p ^p$.
\end{corollary}

\begin{proof}
Arguing as in the proof of Theorem \ref{th:main1}, we may assume that the symmetry condition \eqref{eq:symai} holds, up to replacing $\cA_1$ by $\cA_1 \otimes M_2(\dC)$. Then, if $k \geq 1$ is an integer, 
\begin{equation}\label{eq:patch00}
\| A_\star \|^k = \| A_\star^k \| \AND \| A_\star \|_{2k}^{k} = \| A_\star^k \|_2. 
\end{equation}
Next, we write, if $S_l \subset \dF_d$ is the set of elements of word length $l$ and $a(k,g) = (A^{k}_\star)_{g \o}$,
$$
A_\star^k  = \sum_{g \in \dF_d} a(k,g)\otimes \lambda( g)   = \sum_{l=0}^k \PAR{\sum_{ g \in S_l} a(k,g)\otimes \lambda( g)}  = \sum_{l=0}^k  B_\star^{(k,l)},  
$$
and $B_\star^{(k,l)}$ (defined in \eqref{eq:defBkmbis} with $u(g) = \lambda(g)$) is supported on words length $l$ in the sense of Lemma \ref{le:Haagerup}. From the triangle inequality and Cauchy-Schwarz inequality, we get 
$$
\| A_\star^k \|  \leq \sum_{l=0}^k (l+1)  \sqrt{ \sum_{g \in S_l} \| a(k,g) \|^2  }  \leq (k+1)^{3/2}   \sqrt{\sum_{g \in \dF_d}  \| a(k,g) \|^2 }.
$$
For the second claim, it remains to use \eqref{eq:patch00} and take $p= 2k$.
\end{proof}

\subsubsection{Quantitative exactness}

Corollary \ref{cor:HaagB} is not sufficient to get an effective lower bound on $\|A_N\|_p$ for all $C^*$-algebras $\cA_1$. It is not useful for example when $\cA_1 = M_n(\dC) $ and $n \geq N$.  Thankfully, we can use the exactness of the free group algebra $\cA_\star$ to bypass this limitation.   We refer to the monograph of Brown and Ozawa \cite[Chapter 5]{MR2391387} for background on exactness. One key is the following quantitative exactness theorem:
\begin{theorem}\label{th:exactness}
Let $0 < \veps < 1/2$ and $d\geq 2$ be integers, $H$ a Hilbert space and $(a_{0},\ldots ,a_d) \in \cB(H)$.  Consider $A_\star \in \cB(H\otimes \ell^2(\dF_d))$ defined by \eqref{eq:defA*}.
Let  $r = \lfloor (2d)^{4/\varepsilon} \rfloor$, there exists a self-adjoint projection $p\in \cB(\ell^2(\dF_d))$ of rank at most $r$ such that 
\begin{equation}\label{eq:NO}
\|( 1_H\otimes p ) \cdot A_\star \cdot ( 1_H\otimes p) \|\ge (1-\varepsilon ) \| A_\star \|.
\end{equation}
In addition, there exists a self-adjoint projection $q\in \cB(H)$ of rank at most $r$ such that
$$\| ( q \otimes 1_{\ell^2(\dF_d)} ) \cdot A_\star \cdot (q \otimes 1_{\ell^2(\dF_d)}) \|\ge (1-\varepsilon )\|A_\star \|.$$
\end{theorem}
We are indebted to Narutaka Ozawa for suggesting us to use the exactness of the reduced $C^*$-algebra of the free group $\cA_\star$ and for outlining the proof of the theorem. 

\begin{proof}[Proof of Theorem \ref{th:exactness}]
Let $|\cdot |$ be the reduced word length on $\dF_d$ with free generators $(g_1,\ldots,g_d)$ and set $g_{d+i} = g_i^{-1}$ for all $i \in \INT{d}$. Fix $l\geq 1$ an integer and let $B_{l-1} \subset \dF_d$ be the subset of elements $x$ with  $|x| <l$. We introduce a function $\mu: \dF_d\times \dF_d \to [0,1]$ defined as follows. We set $\mu (x,x)=1-(|x|\wedge l)/l$ and $\mu (x,y)=1/l$ if $x = y z$ with $|x|=|y|+|z|$ and $y \in B_{l-1}$. In all other cases, we set $\mu(x,y) =0$.   
This function satisfies the following three properties:
\begin{enumerate}[(i)]
\item \label{it:p1} $\mu (x,y)=0$ if $y \notin B_{l-1}$.
\item \label{it:p2} For all $x \in \dF_d$, $\sum_y \mu (x,y)=1$
\item \label{it:p3}  For $x \in \dF_d$, set $\eta_i(x) = 1-  \sum_{y \in \dF_d}\sqrt{\mu(g_i x  , g_i y )\mu(x ,  y)}$. We have $\eta(\o) = 0$, $\eta_i(x) = 1/l$ if $0 < |x| < l-1$ or \{$|x| =l-1$ and $|g_i x | = l-2$\}, and $\eta_i(x) = 2/l$ otherwise. 
\end{enumerate}
The first two properties follow directly from the definition. 
As for the third property, for $x \ne \o$, we write $x = g_{i_1} \cdots g_{i_n}$ in reduced form. Observe that $\mu (g_i x ,g_i y)=\mu (x,y)$ for all $y$ except when (1) $y = \o$ and $i = i_1^*$ or $y  = g_i^*$ and $i \ne i^*_1$, and (2) if $n \geq l$, $y =  g_{i_1}\cdots g_{i_{l}}$ and $i = i_1^*$ or  if $n \geq l-1$, $y =  g_{i_1}\cdots g_{i_{l-1}}$ and $i \ne i_1^*$. In all theses  cases, $ \{ \mu (g_i x, g_i y),\mu (x,y)\} = \{0,1/l\}$. Property \eqref{it:p3} follows.

The function $\mu$ can be seen as a mass transport map on $\dF_d$. The map $x\to \mu (x,\cdot )$ is a probability measure on $\dF_d$ with support on $B_{l-1}$ and the action of the canonical generators of $\dF_d$ on these measures is $(1/l)$-invariant with respect to the total variation distance. We refer to Brown and Ozawa \cite[Proposition 5.1.8]{MR2391387} for more context.

We next consider the operator
$\Phi : \cB(\ell^2(B_{l-1}))\to \cB(\ell^2(\dF_d))$ given by 
$$\Phi (e_{x,y})=\sum_{ g \in \dF_d}\sqrt{\mu ( x g ,x )\mu (y g,y)}e_{x g , y g},$$ 
where $e_{x,y} = \delta_{x} \otimes \delta_y$. This function is well-defined since the sum can be expressed as a finite sum from property \eqref{it:p1}. We first claim that this function is completely positive. Indeed, for any $g\in \dF_d$, consider the operator $\cB(\ell^2(\dF_d))\to \cB(\ell^2(\dF_d))$ given 
by 
$$e_{x,y} \mapsto \sqrt{\mu (x g ,x)\mu (y g, y )}  e_{x,y}.$$
It is the Schur multiplier with the rank one self-adjoint projection $(\sqrt{\mu (xg,x)\mu (yg,y)})_{x,y}$.  It is thus completely positive. 
It follows, after conjugating by the appropriate unitary translation, that
 the operator 
$\cB(\ell^2(\dF_d))\to \cB(\ell^2(\dF_d))$ defined by 
$$e_{x,y} \mapsto \sqrt{\mu ( xg,x)\mu (yg,y)}e_{x g ,y g}$$
is also completely positive. 
Therefore its restriction  $\cB(\ell^2(B_{l-1})) \to \cB(\ell^2(\dF_d))$ is again completely positive. Finally, since $\Phi$ is a finite sum of such operators, it is completely positive.
Moreover, one checks directly from property \eqref{it:p2} that $\Phi (1_{B_{l-1}})= 1_{\ell^2(\dF_d)}$. Therefore $\Phi$ is a completely positive unital map, and thus, it is completely contractive. 

Next, let $\Psi: \cB(\ell^2(\dF_d)) \to \cB(\ell^2(\dF_d))$ be the corner map defined by, for $T \in \cB(\ell^2(\dF_d)) $,
$$\Psi ( T ) =  p  \cdot T \cdot p,$$ where $p$ is the self-adjoint projection onto $B_{l-1}$. This map is also completely contractive.

We now claim that for $ i \in \INT{2d}$, we have
\begin{equation}\label{eq:clm2}
\| ( 1_H\otimes (\Phi\circ \Psi ) ) A_\star -A_\star \|\le  \frac{2}{l}\|A_\star\|.
\end{equation}
Indeed, since $\lambda(g_i) = \sum_y e_{g_i y , y}$, we have, setting $S = B_{l-1} \cap g_i^{-1} B_{l-1}$,
\begin{eqnarray*}
(\Phi \circ \Psi )\lambda (g_i)  & = & \sum_{y \in S} \Phi ( e_{g_i y,y} )  \nonumber\\
& = & \sum_{y \in \dF_d,g \in \dF_d} \sqrt{\mu(g_i y g , g_i y  )\mu(y g ,  y) } e_{g_i y g  , y g} \nonumber\\ 
& = & \sum_{x \in \dF_d} e_{g_i x ,x} \PAR{ \sum_{y \in \dF_d}\sqrt{\mu(g_i x  , g_i y)\mu(x  ,  y )} }, \label{eq:clm200} 
\end{eqnarray*}
where at the second line we have used property \eqref{it:p1} (the summand is $0$ for all $y \in \dF_d \backslash S$). 

Next, recalling the definition of $\eta_i(x)$ in property \eqref{it:p3}, we find, 
\begin{eqnarray*}
 A_\star -   (1_H \otimes (\Phi  \circ \Psi)) A_\star   &  =&   \sum_{i=1}^{2d} \sum_{x \in \dF_d} a_i \otimes e_{g_i x ,x} \eta_i(x) \\
  & =& \frac 1 l \sum_{x \in \dF_d \backslash \{\o\} } \sum_{i=1}^{2d}  a_i \otimes e_{g_i x ,x} + \frac 1 l \sum_{x \in \dF_d \backslash B_{l-2}} \sum_{i=1}^{2d}  a_i \otimes e_{g_i x ,x} ( 1 -\IND_{|x| = l-1, x_1 = i^{*}}) \\
  & = & \frac{T_1 + T_2}{ l } ,
\end{eqnarray*}
where for $x \ne \o$ written in reduced form as $x =g_{i_1}\ldots g_{i_n}$, we have set $x_1 = i_1$ (if $l=1$, $\IND_{|x| = l-1, x_1 = i^{*}} = 0$ by convention). Now, we observe that $T_1 = A_\star P_1$ with $P_1$ orthogonal projection onto $\ell^2( \dF_d \backslash \{o\})$ and $T_2 = P_2 T_2 P_2$ where $P_2$ orthogonal projection onto $\ell^2( \dF_d \backslash B_{l-2})$. Therefore, $\| T_i \| \leq \| A_\star \|$ and we deduce that \eqref{eq:clm2} holds.

By the triangle inequality, \eqref{eq:clm2} implies
$$\|( 1_H\otimes (\Phi\circ \Psi ) ) A_\star\|\ge  (1-\frac{2}{l}) \|A_\star\|.$$
Complete contractivity of $\Phi$ implies
$$\|(1_H\otimes \Psi ) A_\star\|\ge \|(1_H\otimes (\Phi\circ \Psi ))A_\star \| \geq (1-\frac{2}{l}) \|A_\star\|.$$
Thus, by complete contractivity of $\Psi$ ,
\begin{equation}\label{eq:NO2}
\|A_\star \|\ge \| (1_H\otimes \Psi ) A_\star \|\ge (1- \frac{2}{l} ) \|A_\star\|.
\end{equation}
In other words, this argument shows that the projection onto $B_{l-1}$ is within $2/l$ of a complete isometry. Since  $\Psi (T) = p T p$, this complete the proof of \eqref{eq:NO} by taking $l = \lceil 2 / \veps \rceil$ and using that $ | B_{l-1}| \leq (2d)^l$.

To prove the second assertion, let $\eta >0$ and
consider a unit vector $x\in H\otimes \ell^2(B_{l-1})$ such that  
$$\|(1_H\otimes \Psi ) A_\star  x\|_2\ge (1-\eta )\|(1_H\otimes \Psi ) A_\star\|.$$
Let $x=\sum_{i=1}^r \lambda_i e_i\otimes f_i$ be a singular value decomposition: $\lambda_i \geq 0$, $e_i \in H$, $f_i \in \ell^2(B_{l-1}) $ and $r = |B_{l-1}|$. Similarly, $y = (1_H \otimes \Psi)A_\star x  = (1_H \otimes p ) A_\star x\in H \otimes \ell^2(B_{l-1})$ admits a singular value decomposition $y = \sum_{i=1}^r \mu_i u_i \otimes v_i$. 
Let $q \in \cB(H)$ be the orthogonal projection onto
the span of $e_i$, $u_i$, $1 \leq i \leq r$. 
By the construction of the projection $q$, we have
\begin{eqnarray*}
\| (q \otimes 1_{\ell^{2}(\dF_d)}) \cdot A_\star \cdot (q \otimes 1_{\ell^{2}(\dF_d)}) x \|_2 \geq  \|(1_H\otimes \Psi ) A_\star  x\|_2 \geq (1-\eta )\|(1_H\otimes \Psi ) A_\star\|.
\end{eqnarray*}
By the definition of the operator norm, we have shown
$$\|(q \otimes 1_{\ell^{2}(\dF_d)}) \cdot A \cdot (q \otimes 1_{\ell^{2}(\dF_d)})\|\ge  (1-\eta)\|(1_H\otimes \Psi ) A_\star\|.$$
It remains to use \eqref{eq:NO2} with $l = \lceil 2/ \veps \rceil$ and $(1-\veps)(1-\eta) \geq (1 - 1.1\veps)$ for $\eta =\veps / 10$ and adjust the constant $r$.
\end{proof}

\subsection{Case of Haar unitaries: proof of Theorem \ref{th:LBU1}}

We may use Theorem \ref{th:FKBlm} to get a lower bound on $\| A_N \|_p$ thanks to the concentration of measure phenomenon, which in our context follows from Gromov and Milman \cite{MR708367} or Bakry-Emery curvature condition \cite{MR889476}. 

\begin{lemma}\label{le: GM}
Let $(a_0,\ldots,a_{2d})$ in $\cA_1$ and $A_\star$, $A_N$ be as in \eqref{eq:defA*}-\eqref{eq:defA} with  $U_1,\ldots, U_d$ be independent and Haar-distributed elements on $\dU_N$. There exists a numerical constant $c >0$ such that for any $t >0$ and any $2 \leq p \leq \infty$,
$$
\dP \PAR{| \| A_N \|_p - \dE  \| A_N \|_p  | \geq t  \| A_\star\|} \leq 2  \exp ( - c N^{1 + 2/p} t^2/ d), 
$$
(for $p = \infty$, $\| \cdot \|_\infty = \| \cdot \|$ and $N^{2/p} = 1$).
\end{lemma}

\begin{proof}
 We introduce the Hilbert-Schmidt norm on $M_N^d(\dC)$ defined by
$$
\| (M_1,\ldots,M_d) \|_{2} = \sqrt{ \sum_{i=1}^d \Tr (M_i M_i^*)  },
$$
(note that the trace is not normalized here). This norm coincides with the Euclidean norm when identifying $M^d_N(\dC)$ with $\dR^k$, $k = 2dN^2$. Next, we consider the function  $f : (M_N)^d \to \dR_+$ $f(U_1,\ldots,U_d) = \| A_N \|_p$ with $A_N = A_N(U)$ as in \eqref{eq:defA} and $U_{i+d} = U_i^*$ as usual. This function is Lipshitz. Indeed, for $U= (U_1,\ldots,U_d)$, $V = (V_1,\ldots,V_d)$ in $M_N^d(\dC)$, we have, from the triangle inequality,
$$
| f (U ) -f (V) | \leq \| \sum_{i\in \INT{2d}} a_i \otimes (U_i-V_i ) \|_p \leq  \sum_{i\in \INT{2d}} \| a_i \otimes (U_i-V_i ) \|_p \leq \sqrt{ \sum_{i \in \INT{2d}} \|a_i\|_p^2 } \sqrt { 2\sum_{i \in \INT{d}}  \| U_i-V_i \|^2_p}, 
$$
where we have used that $\| a \otimes x \|_p =\| a \|_p\|x \|_p $ and $\| x^* \|_p = \| x\|_p$. We have $\| a_i\|_p \leq \|a_i \| \leq \| A_\star\| $ (since $(A_\star)_{g_i \o} = a_i$). Moreover, for any matrix $M \in M_N(\dC)$ with singular value vector $s = ( s_1,\ldots s_N)$, we have for $p \geq 2$,
$$
\PAR{\Tr  | M |^p }^{1/p} = \PAR{ \sum_{i=1}^N s_i^p }^{1/p} \leq  \PAR{ \sum_{i=1}^N s_i^2 }^{1/2} = \sqrt{ \Tr (M M^*) }. 
$$
Since $\| M \|_p = N^{-1/p} \PAR{\Tr | M |^p }^{1/p}$, it follows that 
$$
| f (U ) -f (V) | \leq  \frac{2 \sqrt d \| A_\star\| }{ N^{1/p} } \| U - V \|_2.
$$
It remains to apply \cite[Theorem 5.17]{MR3971582} (we can take $c = 96$).
\end{proof}

As a corollary of Theorem \ref{th:FKBlm} and Lemma \ref{le: GM}, we obtain the following. 

\begin{corollary}\label{cor:LBU1}
Let $(a_0,\ldots,a_{2d})$ in $\cA_1$ and $A_\star$, $A_N$ be as in \eqref{eq:defA*}-\eqref{eq:defA} with  $(U_1,\ldots, U_d)$ are independent Haar-distributed elements on $\dU_N$. There exists a numerical constant $c >0$ such that for all even integers $2 \leq p \leq N^{1/(15d + 50)}$, 
$$
\dE \| A_N \|_p \geq \|A_\star\|_p  \PAR{ 1 -   \frac{c \sqrt{d p } \| A _\star \|^p }{\sqrt N \|A_\star\|_p ^p} }_+^{1/p} ,
$$
where $(x)^{1/p}_+ = (\max(x,0))^{1/p}$ and the convention $0/0 =1$.
\end{corollary}

\begin{proof}[Proof of Corollary \ref{cor:LBU1}]
Arguing as in the proof of Theorem \ref{th:main1}, we may assume that the symmetry condition \eqref{eq:symai} holds. Lemma \ref{le: GM} implies that for some new numerical constant $c >0$, for all $p \geq 1$, 
\begin{equation}\label{eq:eop}
\dE [ \| A_N \|^p_p ] - (\dE \| A_N \|_p  )^p \leq  \PAR{ \frac{c \sqrt{d p } \| A _\star \|}{N^{1/2+ 1/p}} }^p,
\end{equation}
We next apply Theorem \ref{th:FKBlm}, we get, for integer $p \leq \ell_0$ even, for some new $c >0$,
$$
(\dE \| A_N \|_p  )^p \geq \|A_\star\|_p ^p  - \PAR{ \frac{c \sqrt{d p } \| A _\star \|}{\sqrt N } }^p.
$$
The conclusion follows.  
\end{proof}

\begin{remark}\label{rk:HaarON}
We have focused on the unitary group $\dU_N$. For the orthogonal group $\dO_N$, the situation is rather different as it has two disconnected components. The subgroup $\dSO_N$ satisfies the Bakry-Emery curvature condition, and Lemma \ref{le: GM} holds for $(U_1,\ldots,U_d)$ Haar distributed on $\dSO_N^d$. We could, however, extend Theorem \ref{th:LBU1} to $\dO_N^d$ by following any of these two strategies. (1) extend Theorem \ref{th:FKBlm} to $\dSO_N^d$, this can be done by justifying that Lemma \ref{cor:WG2} holds for $\dSO_N$ which in turn follows from the fact that the Weingarten calculus is identical for $\dO_N$ and $\dSO_N$ for products of size less than $N$, see \cite{collins2021weingarten}. (2) The proof of Theorem \ref{th:FKBlm} can be adapted to prove that $\dE ( \tau_N (A_N^\ell) - \tau(A_\star^\ell) )^2 $ is bounded by $d^8 (\ell +1)^{8d +16} \rho^{2 \ell} / N^2$ for $\ell$ as in Theorem \ref{th:FKBlm}. We then directly obtain the concentration of $\tau_N (A_N^\ell) $. Note that this last upper bound improves the order of fluctuations on $\| A_N\|_p$.  
\end{remark}

We are now ready to prove Theorem \ref{th:LBU1}. 

\begin{proof}[Proof of Theorem \ref{th:LBU1}]
Up to adjusting the constant $c$, we may assume without loss of generality $d \leq N^{\alpha}$ for any fixed $\alpha >0$.  Let $\veps > 0$. Let $q$ be the orthogonal projection given by Theorem \ref{th:exactness}. Let $\tilde A_N$ and $\tilde A_\star$ be given by \eqref{eq:defA} and \eqref{eq:defA*} with $a_i = q a_i q$. By Theorem \ref{th:exactness}, we may see $\tilde A_N$ as an element of $M_r(\dC) \otimes M_N(\dC)$ and similarly for $\tilde A_\star$ in $M_r(\dC) \otimes \cA_\star$. From the definition of the operator norm, for any $p \geq 1$,
$$
\| A_N \| \geq \| \tilde A_N\| \geq \| \tilde A_N\| _p.
$$
(the norm $\| \cdot \|_p$ is in $M_r(\dC) \otimes M_N(\dC)$). Next, we may use Corollary \ref{cor:LBU1}, for $2 \leq p \leq N^{1/(15d + 50)}$, we find
 $$
 \dE \| A_N \| \geq  \PAR{ \|\tilde A_\star\|^p_p  -  \frac{ c \sqrt{d p } \| \tilde A _\star \|^p }{\sqrt N } }_+^{1/p}.
$$
Since $q a_i q$ has rank at most $r \leq (2d)^{4/ \veps}$,   we may use Haagerup's inequality: from Corollary \ref{cor:HaagB}, 
$$
 \dE \| A_N \| \geq  \| \tilde A_\star \| \PAR{ \frac{(2d)^{-4 / \veps}}{2 p^3}   -  \frac{ c \sqrt{d p }  }{\sqrt N } }_+^{1/p}.
$$
From Theorem \ref{th:exactness} and our choice of $q$, we finally obtain 
$$
 \dE \| A_N \| \geq  \| A_\star \| (1- \veps) \PAR{ \frac{(2d)^{-4 / \veps}}{2 p^3}   -  \frac{ c \sqrt{d p }  }{\sqrt N } }_+^{1/p}.
$$
It remains to take $\veps = 16 \log (2d) / \log N$ (if less than $1$) and $p =  N^{1/(15d + 50)}$. Then the first term in parenthesis is  $N^{1/4}/(2 p^{3})$ which is twice larger that $(dp/N)^{-1/2}$ for all $N$ large enough if $\alpha >0$ is small enough (recall $d \leq N^{\alpha}$).  Adjusting the constant $c$, the theorem follows. 
\end{proof}

\subsection{Case of permutations: proof of Theorem \ref{th:LBP}}

We start with a preliminary lemma.

\begin{lemma}\label{le:LBP}
Let $(a_0,\ldots,a_{2d})$ in $\cA_1$ and $A_\star$, $A_N$ be as in \eqref{eq:defA*}-\eqref{eq:defA} with  $(U_1,\ldots, U_d)$ permutation matrices of $\sigma = (\sigma_1,\ldots, \sigma_d)$. Let $p \geq 2$ be an even integer and assume that there exists $x,y \in \INT{N}$ such that  $(G_\sigma,x)_p$ and $(G_\sigma,y)_p$ are disjoint and without cycle. Then, 
$$
\| A_N \Pi_N \| \geq \| A_\star \|_{p}
$$
\end{lemma}

\begin{proof}
It is sufficient to prove the claim  when the the symmetry condition \eqref{eq:symai} holds.  Consider the homomorphism $U: \dF_d \to \dS_N$ such that $U(g_i) = U_i$ for all $i \in \INT{d}$ (it is also defined in the proof of Theorem  \ref{th:resIB}). For $x \in \INT{N}$, let $V_{\sigma,x}$  be the vector subspace of $\dC^N$ spanned by $U(g)e_x $ for all $g \in \dF_d$ of word length at most $p$. Define similarly, $V_{\star}$ be the vector subspace of $\ell^2 (\dF_d)$ spanned  by $e_{g}$, for all $g \in \dF_d$ of word length at most $p$. We observe that if $(G_\sigma,x)_p$  has no cycle then $V_{\sigma,x}$ and $V_\star$ are isomorphic then for any vector $f \in H_1$
$$
\| A_N ^p f \otimes e_x \|_2 = \| A_\star ^p f \otimes \delta_{\o} \|_2,
$$
(this can also be seen directly from Theorem \ref{th:powerIB}). 

Let $x \ne y \in \INT{N}$, we note that  $\IND_N \otimes \IND_N (\delta_x - \delta_y) = 0$.
Hence, for any unit vector $f \in H_1$, $\Pi_N f \otimes  (\delta_x - \delta_y) = f \otimes  (\delta_x - \delta_y)$ and
$$
\| A_N \Pi_N   \|^p =  \| (A_N \Pi_N)^p   \| = \| A^p _N \Pi_N    \| \geq \| A_N ^p f \otimes \frac{ \delta_x - \delta_{y}}{\sqrt 2} \|_2 .
$$
Moreover,  if $(G_\sigma,x)_p$ and  $(G_\sigma,y)_p$ are disjoint then the vector spaces $A_N ^p H_1 \otimes \delta_x $ and $A_N ^p H_1 \otimes \delta_{y} $ have a trivial intersection.  Hence, for any $f \in H_1$, from Pythagoras' theorem,
$$
\| A_N^p f \otimes ( \delta_x - \delta_{y})\|^2_2 =  \| A_N^p f \otimes \delta_x\|^2_2 + \| A_N^p f \otimes \delta_{y}\|^2_2.
$$ 
We obtain that if $(G_\sigma,x)_p$ and  $(G_\sigma,y)_p$ are disjoint and have no cycle, then for any unit $f \in H_1$,
$$
\| A_N \Pi_N   \|^p  \geq \| A_\star ^p f \otimes \delta_{\o} \|_2,
$$
The conclusion of Lemma \ref{le:LBP} follows from the GNS correspondence. 
\end{proof}

We are ready for the proof of Theorem \ref{th:LBP}. 
\begin{proof}[Proof of Theorem \ref{th:LBP}]
The proof goes as the proof of Theorem \ref{th:LBU1}. Let $\veps > 0$. Let $q$ be the orthogonal projection given by Theorem \ref{th:exactness}.
Let $\tilde A_N$ and $\tilde A_\star$ be given by \eqref{eq:defA} and \eqref{eq:defA*} with $a_i = q a_i q$. From the definition of the operator norm and Lemma \ref{le:LBP},
$$
\| A_N \Pi_N \| \geq \| \tilde A_N \Pi_N \| \geq \| \tilde A_\star \|_p .
$$
Then from Corollary \ref{cor:HaagB}, 
$$
 \| A_N \Pi_N \| \geq  \| \tilde A_\star \|  \PAR{ \frac{(2d)^{-4 /   \veps}}{2 p^3}}^{1/p} \geq\|  A_\star \|  ( 1 - \veps) \PAR{ \frac{(2d)^{-4/ \veps}}{2 p^{3}}}^{1/p}.
$$
It remains to take $\veps =1 \wedge  \sqrt {\log (2 d) / p }$ and adjust the constant $c$. \end{proof}

\subsection{Universal Alon-Boppana lower bound: proof of Theorem \ref{th:AB}}

The proof is a variant of the previous subsection. Arguing as in the proof of Theorem \ref{th:LBP}. It is sufficient to prove the following lemma. 

\begin{lemma}\label{le:AB}
Let $(a_0,\ldots,a_{2d})$ in  $\cA_1$ with non-negative joint moments and $A_\star$, $A_N$ be as in \eqref{eq:defA*}-\eqref{eq:defA} with  $(U_1,\ldots, U_d)$ permutation matrices on $N \geq 2$. Assume that $p =   (1/4) \log (N)/ \log (2d) \geq 1$, 
$$
\| A_N \Pi_N \| \geq \| A_\star \|_{p}
$$
\end{lemma}

\begin{proof} From the linearization trick, we can assume that the symmetry condition \eqref{eq:symai} holds. Let $(\sigma_1,\ldots,\sigma_{d})$ be the associated permutation matrices and $G_\sigma$ the corresponding Schreier graph on $\INT{N}$. We set $h = 2 \lceil p/2 \rceil $. Since the $(G,x)_h$ has at most $(2d)^h$ vertices, their exist two vertices $x,y \in \INT{N}$ such that  $(G,x)_h$ and $(G,y)_h$ are disjoint. 

We then argue as in Lemma \ref{le:LBP}. Since $x \ne y \in \INT{N}$,  $\IND_N \otimes \IND_N (\delta_x - \delta_{y}) = 0$ and for any unit vector $f \in H_1$, $\Pi_N f \otimes  (\delta_x - \delta_{y}) = f \otimes  (\delta_x - \delta_{y})$ and
$$
\| A_N \Pi_N   \|^{2h} = \| A^{h} _N \Pi_N    \|^2 \geq \| A_N ^{h} f \otimes \frac{ \delta_x - \delta_{y}}{\sqrt 2} \|^2_2,
$$
and, since $(G_\sigma,x)_h$ and  $(G_\sigma,y)_h$ are disjoint, from Pythagoras' theorem,
$$
\| A_N^{h} f \otimes ( \delta_x - \delta_{y})\|^2_2 =  \| A_N^{h} f \otimes \delta_x\|^2_2 + \| A_N^{h} f \otimes \delta_{y}\|^2_2.
$$ 
From the GNS correspondence, we deduce that 
$$
\| A_N \Pi_N \|^{2h} \geq \frac{1}{2} \PAR{\tau_1 ( (A_N^{2h} )_{x x} )   + \tau_1 ( (A_N^{2h})_{y y}) }.  
$$ 
Now,  consider the homomorphism $U: \dF_d \to \dS_N$ such that $U(g_i) = U_i$ for all $i \in \INT{d}$ (it is also defined in the proof of Theorem  \ref{th:resIB}). Consider the elements $T_N \in \cA_1 \otimes M_N(\dC)$ and $T \in \cA \otimes \cA_\star$ of the form 
$$
T_N = \sum_{g \in \dF_d} a(g) \otimes U(g) \AND T = \sum_{g \in \dF_d} a(g) \otimes \lambda(g).
$$
where the sum is finite and $a(g)$ is a sum of terms of the form $a_{i_1}\cdots a_{i_n}$.  From the assumption that $(a_0,\ldots, a_{2d})$ have non-negative joint moments, we have $\tau(a(g)) \geq 0$. Also, $U(g)_{xx} \geq 0$ since $U(g)$ is a permutation matrix. Hence, 
$$
\tau_1(  (T_N)_{xx} ) = \sum_g \tau_1(a(g)) U(g)_{xx} \geq  \tau_1 (a(\o)) =\tau (T).  
$$
The conclusion follows by applying this to $T_N = A_N^{2h}$ and $T =A_\star^{2h}$.
\end{proof}

\section{Polynomial of unitaries}
\label{sec:poly}

\subsection{Linearization trick for unitaries}

In this subsection, we consider an extension of our main model defined in Section \ref{sec:model}. Recall that $\cA_\star$ is the left group algebra of the free group $\dF_d$ with free generators $g_1, \ldots ,g_d$ and their inverses.   We denote by $B_l \subset \dF_d$ the set of elements of word length at most $l$. Let $(a_g)_{g \in B_l}$ be a collection of operators in a unital $C^*$-algebra $\cA_1 \subset \cB(H_1)$. We consider the operator in $ \cA_1 \otimes \cA_\star$:
\begin{equation*}\label{eq:defP*}
P_\star   =\sum_{g \in B_l} a_g \otimes \lambda(g).
\end{equation*}
Let $(U_1,\ldots,U_d)$ be   unitaries in $\dU_N$, its corresponding  operator in $\cA_1 \otimes M_N(\dC)$,
\begin{equation*}
P_N  =   \sum_{g \in B_l }a_g \otimes U(g),
\end{equation*}
where $U : \dF_d \to \dU_N$ is the  group homomorphism defined by $U(g_i) = U_i$. The degree of $P_\star$ and $P_N$ is the minimal $l$ such that they can be expressed in this form. The operators $A_\star$ and $A_N$ in \eqref{eq:defA*}-\eqref{eq:defA} are operators of degree (at most) $l= 1$.

The goal of this subsection is to leverage estimates on $\|A_N\| / \| A_\star \|$  into estimates on $\| P_N \| / \| P_\star\|$.  This can be done thanks to the linearization trick. In a nutshell, it is possible to reduce the degree of $P_N$ iteratively at the cost of replacing $\cA_1$ by $\cA_1 \otimes M_k(\dC)$ for some integer $k \geq 1$. For ease of exposition, we generalize the framework. Let $\Gamma$ be a finite generated group and let $S$ be a finite symmetric set of generators of $\Gamma$. We denote by $B_l$ the ball of radius $l$ for the word metric associated to $S$. Let $H_2$ be an Hilbert space and let 
\begin{equation}\label{eq:defP}
P_\Gamma = P_\Gamma ((a_g)_{g \in B_l} ) = \sum_{g \in B_l} a_g \otimes \lambda(g) \AND  P = P((a_g)_{g \in B_l} ) = \sum_{g \in B_l} a_g \otimes u(g),
\end{equation} where $\lambda$ is the left-regular representation $\Gamma$ and $g \mapsto u(g)$ is a group homomorphism from $\Gamma$ to $\dU(H_2)$. 
We have the following lemma, which is a quantitative version of the linearization trick of Pisier for unitaries \cite[Proposition 6]{MR1401692}; see also Lehner \cite[Section 5]{MR1738412} and the appendix in \cite{BC2}.

\begin{lemma}\label{le:linea}
Let $l\geq 2$ be an even integer and a collection $(a_g)_{g \in B_l}$ in $\cA_1$. There exist $(b_g)_{g \in B_{l/2}}$ with $b_g \in \cA_1 \otimes M_{2 |B_{l/2}|}(\dC)$ and $ \theta \geq 0$ such that, if $P = P((a_g)_{g\in B_l})$ and $Q = P ( (b_g)_{g \in B_{l/2}})$ for some group homomorphism $g \mapsto u(g)$ from $\Gamma$ to $\dU(H_2)$, we have 
$$
 \| P \| = \| Q \|^2 - \theta.
$$
Moreover, if $P_\Gamma$ is as in \eqref{eq:defP}, we have $\theta  \leq |B_{l/2}| \| P_\Gamma \|$.
\end{lemma}

\begin{proof} By the standard linearization trick (see proof of Theorem \ref{th:main1}), we can assume that the symmetry condition is met at the cost of replacing $\cA_1$ by $\cA_1 \otimes M_2 (\dC)$ and also that $\|P \|$ is equal to the rightmost point in the spectrum of $P$. We repeat the proof of  \cite[Lemma 29]{BC2} with the present notation. Recall $\cA_1 \subset \cB(H_1)$.

Consider the self-adjoint element $\tilde a \in  M_{B_{l/2}} (\dC) \otimes \cA_1$ defined in matrix form as 
$
\tilde a = (\tilde a_{g,h})_{g,h \in B_{l/2}}
$
with 
$$
\tilde a_{g,h} = a_{g^{-1} h} / |\{(g',h')\in B_{l/2}, (g')^{-1}h'=g^{-1} h\}|.
$$
In particular, for all $w \in B_l$, $$\sum_{g,h\in B_{l/2}, g^{-1} h=w}\tilde a_{g,h}=a_w.$$ 
We claim that $\| \tilde a \| \leq \| P_\Gamma \|$. Indeed, if $\Tr$ denotes the partial trace on $M_{B_{l/2}} (\dC)$ (non-normalized), we have 
$$
 \| \tilde a \|^2 \leq \| \Tr  ( \tilde a \tilde a^*  ) \|  = \NRM{\sum_{g,h \in B_{l/2}} \tilde a_{g,h} \tilde a_{g,h}^*  } \leq \NRM{\sum_{w \in B_{l}}  \tilde a_{w} \tilde a_{w}^*  } = \NRM{ (P_\Gamma P_\Gamma ^*)_{\o \o} } \leq \| P_\Gamma \|^2,
$$ 
where $\o$ is the unit of $\Gamma$. The operator $\tilde a + \| \tilde a \| 1$ is positive semi-definite and let
 $\tilde b \in M_{B_{l/2}} (\dC) \otimes \cA_1$ be a self-adjoint square root of $\tilde a + \| \tilde a \| 1$.
For $g \in B_{l/2}$,  we then set $$b_g = \tilde b \cdot ( e_{g,\o} \otimes 1_{H_1})   \in M_{B_{l/2}} (\dC) \otimes \cA_1,$$ where for $g,h \in B_{l/2}$, $e_{g,h}  = \delta_g \otimes \delta_h \in M_{B_{l/2}} (\dC) $ and $\o$ is the unit of $\Gamma$.  Then, the operator $Q=\sum_{g\in B_{l/2}} b_g\otimes u(g)$ satisfies
\begin{eqnarray*}  Q^*   Q & = &\sum_{g,h  \in B_{l/2}}   (e_{\o,g} \otimes 1_{H_1}) \tilde b ^2 ( e_{h,\o} \otimes 1_{H_1}) \otimes u(g^{-1} h) \\
 &= & \sum_{g,h  \in B_{l/2}} e_{\o,\o} \otimes ( \tilde a_{g,h} + \| \tilde a \| \IND_{g = h} \cdot 1_{H_1}) \otimes u(g^{-1} h)  \\
 &=& e_{\o,\o} \otimes (P + \theta 1),
\end{eqnarray*}
  where we have set $\theta =\| \tilde a \| |B_{l/2}|$.
Finally, since $\theta \geq 0$ and the rightmost point in the spectrum of $P$ is equal to $\|P\|$, we find 
$
\| Q\|^2 = \| P + \theta  1\| = \|P \| + \theta
$. 
 \end{proof}

We note that $|B_l| \leq (2d)^l$ and thus $\prod_{k=0}^{m-1} (  2 |B_{2^k}|) \leq 2^{m} (2d)^{2^m}$. 
By iteration, we deduce the following corollary. 

\begin{corollary}\label{cor:linea}
Let $l \geq 2$ be an integer, $m = \lceil \log_2 l \rceil $ and $P$ as above. For each $k \in \INT{m}$, there  is an integer $n_k \geq 1$, elements $(a^k_g)_{g \in B_{2^{m-k}}}$ with $a^k_g \in \cA_1 \otimes M_{n_k}(\dC)$ and $\theta_k \geq  0$ such that $n_m \leq 2l(2d)^{2l}$ and, setting $Q_0 = P$ and $Q_k = P ( (a^k_g)_{g \in B_{2^{m-k}}})$, we have
$$
 \| Q_{k-1} \| = \| Q_k \|^2 - \theta_k.
$$
Moreover, $\theta_k \leq \| P_\Gamma ((a^{k-1}_g)_{g \in B_{2^{m-k+1}}}) \| (2d)^{2^{m-k}}$ (for $k=1$, $P_\Gamma$ is as in \eqref{eq:defP}).
\end{corollary}

The main result of this section is the following theorem. 

\begin{theorem}\label{le:P2A}
Let $l \geq 2$ be an integer, $m = \lceil \log_2 l \rceil $  and $P_\Gamma, P$ be as in \eqref{eq:defP}. Set $a_0 = a^m_{\o}$ and for $i \in \INT{2d}$, $a_i = a^m_{g_i}$ where $(a^m_g)_{g \in B_1}$ is as in Corollary \ref{cor:linea}. Let $A_\Gamma , A $ be as in \eqref{eq:defP} with $l=1$. For $\veps >0$, if $\| A \| \leq \| A_\Gamma \| ( 1+ \veps)$ and $\veps \leq (2d)^{-l} l ^{-2}$ then 
$$
\| P \| \leq \| P_\Gamma \| ( 1 + 4 l^2 (2d)^{2l} \veps ).
$$ 
Conversely, if $\| A\| \geq \| A_\Gamma \| ( 1- \veps)$ then 
$$
\| P \| \geq \| P_\Gamma \| ( 1 - 4 l^2 (2d)^{2l} \veps ).
$$ 
\end{theorem}
\begin{proof}
For $k \in \INT{m}$, let $(a^k_g)_{g \in B_{2^{m-k}}}$ be as in Corollary \ref{cor:linea}. Let $Q_{k} = P ((a^{k}_g)_{g \in B_{2^{m-k}}}) $ and $Q_{\Gamma,k} = P_\Gamma ( (a^{k}_g)_{g \in B_{2^{m-k}}})) $ be the corresponding operators. We set $Q_0 = P$ and $Q_{\Gamma,0} = P_\Gamma$. With this notation, we have $Q_m = A$ and $Q_{\Gamma,m} = A_\Gamma$.

For some $k \in \INT{m}$, assume that $\| Q_k \| \leq \| Q_{\Gamma,k} \| ( 1+ \veps_k)$ with $0 \leq \veps_k \leq  1$. Then, applying Corollary \ref{cor:linea} twice, we get
\begin{eqnarray*}
\| Q_{k-1} \| - \| Q_{\Gamma,k-1} \|  =  \| Q_{k} \|^2 - \| Q_{\Gamma,k} \|^2 
 \leq   \veps_k (1 + \veps_k) (\|Q_{\Gamma,k-1} \| + \theta_k ) 
 \leq  \veps_{k-1} \| Q_{\Gamma,k-1 } \|,
\end{eqnarray*}
with $\veps_{k-1} = 4 \veps_k (2d)^{2^{m-k}}$. It follows that 
$$
\veps_0 \leq \veps_m \prod_{k=1}^{m} 4 (2d)^{2^{m-k}} \leq  4^m ( 2d)^{2^m} \veps \leq  4 l^2 (2d)^{2l} \veps,
$$
provided that $4^{m-1} (2d)^{2^{m-1}} \veps \leq l^2 (2d)^{l} \veps \leq 1$.

Conversely, assume $\| Q_k \| \geq \| Q_{\Gamma,k} \| ( 1 - \veps_k)$, then we have
$
\| Q_{\Gamma,k-1} \| - \| Q_{k-1} \|  =  \| Q_{\Gamma,k} \|^2 - \| Q_{k} \|^2. 
$
It follows that if $\| Q_{k} \| \geq \| Q_{\Gamma,k}\|$ then $\| Q_{k-1}\| \geq \| Q_{\Gamma,k-1}\|$. Otherwise, $ \| Q_{\Gamma,k} \| ( 1 - \veps_k) \leq \| Q_k \| \leq \| Q_{\Gamma,k} \|$ and we get
$$
\| Q_{\Gamma,k-1} \| - \|Q_{k-1 }\| \leq \veps_k \| Q_{\Gamma,k}\|( \| Q_{\Gamma,k} \| + \| Q_{k}\|) \leq  2 \veps_k (\| Q_{\Gamma,k-1} \| + \theta_k ) \leq \veps_{k-1}\| Q_{\Gamma,k-1} \|,
$$
with $\veps_{k-1} = 4 \veps_k (2d)^{2^{m-k}}$ as above. Hence, in both cases,  $\| Q_{k-1} \| \geq \| Q_{\Gamma,k-1} \| ( 1 - \veps_{k-1})$.
\end{proof}

In conjunction with Theorem \ref{th:main1} and Theorem \ref{th:LBU1}, Theorem \ref{le:P2A} gives quantitative estimates on $\|P_N\| / \|P_\star \|$ for $(U_1,\ldots,U_d)$ independent Haar unitaries on $\dU_N$ or $\dO_N$. For example, in the case where $\cA_1 = \dC$ and $d = O(1)$, we obtain $\| P_N \| \leq ( 1 + o(1) ) \|P_\star\|$ with high probability as soon as  $l \leq c \log N$ for some constant $c >0$. Conversely, we get $\|P_N\| \geq ( 1 - o(1) ) \|P_\star \|$ with high probability if $l \leq c \log (\log N)$. Similarly, if $(U_1,\ldots,U_d)$ are independent random permutation matrices, we may use Theorem \ref{th:main3} and Theorem \ref{th:LBP}. If $\cA_1 = \dC$  and $ d = O(1)$, we find $\| P_N \Pi_N \| = ( 1 + o(1) ) \|P_\star\|$ with high probability as soon as  $l \leq c \log (\log N)$ for some constant $c >0$.

\subsection{Hausdorff distance for spectra}

Let $P$ and $P_\Gamma$ be as in \eqref{eq:defP}-\eqref{eq:defP}. Assume that further that $a(g)^* =  a(g^{-1})$ for all $g \in B_l \subset \dF_d$. Then $P$ and $P_\Gamma$ are self-adjoint, and it turns out that it is straightforward to compare the spectrum (as a set) of $P$ and $P_\Gamma$ in terms of norm estimates of polynomials with degrees doubled, which in turn can be linearized thanks to Theorem \ref{le:linea}.

\begin{proposition}
Let $l \geq 1$ be an integer  and $P_\Gamma, P$ be as in \eqref{eq:defP} such that $a(g^{-1}) = a(g)^*$ for all $g \in B_l$. Let $x\leq  y $ be real numbers. There exists $(b_g)_{g \in B_{2l}}$ in $\cA_1$ such that if $Q = P (  (b_g)_{g \in B_{2l}})$ and $Q_\Gamma = P_\Gamma ( (b_g)_{g \in B_{2l}})$ the following holds. 

For $\veps >0$, set $\eta = 4 \veps \| P_\Gamma \|^2 / (y -x)$ if $\| Q \| \leq (1+ \veps) \|Q_\Gamma\|$,  $(x,y) \subset \sigma (P_\Gamma)^c$ and $y > x + 2\eta$, then $(x- \eta,y-\eta ) \subset \sigma(P)^c$.

Conversely, if $0 < \veps < 1$, $\| Q \| \geq (1-  \veps) \|Q_\Gamma\|$ and $[x,y] \in \sigma(P)$ then $\mathrm{dist}([x,y], \sigma (P) ) \leq \eta$. If $x = y$ we have $\mathrm{dist}(x, \sigma (P) )  \leq \sqrt{2 \veps \| P_\Gamma \| }$.  
\end{proposition}

\begin{proof}
Set $c = \|P_\Gamma\|$. We consider the quadratic function $f(\lambda ) = \theta + ( y - \lambda) ( x -a)$ and set $Q = f(P)$ and $Q_\Gamma = f(P_\Gamma)$, these operators are of the required form. If $\theta = 2 c^2$, for all $\lambda  \in [-c,c] \backslash (a,b)$, $|f(\lambda)| \leq f(a) = f(b) = \theta$. Hence, by functional calculus,  $\| Q_\star \| = \theta$ and $\| Q \| = \max (|f(\lambda)|, \lambda \in \sigma(P))$. Moreover, by concavity, if $\lambda \in [a,(a+b)/2]$ then $f(\lambda ) \geq  \theta + (\lambda - a)(b-a)/2 $. This gives the first claim.   For the second claim, this is similar (with the same function $f$).
\end{proof}

\section{Application to strong convergence}

\subsection{Strong convergence}
\label{subsec:strong}

Fix an integer $d\geq 1$ and let $(x_1,\ldots,x_d)$ be elements of a unital $C^*$-algebra $(\cA,\tau)$ with faithful normal state $\tau$ and associated operator norm $\| \cdot \|$. For integer $N \geq 1$, let $(X_{N,1},\ldots,X_{N,d})$ be elements of $(\cA_N, \tau_N) $ a sequence of unital $C^*$-algebra with faithful normal states $\tau_N$ and operator norm also denoted by $\| \cdot \|$. 

We shall say that $(X_{N,1},\ldots,X_{N,d})$ {\em converges strongly} to $(x_1,\ldots,x_d)$ if for any non-commutative $*$-polynomial $P \in \dC[\dF_d]$,  we have
\begin{eqnarray}\label{eq:strong1}
\lim_{N \to \infty} \| P ( X_{N,1},\ldots, X_{N,d}) \| = \| P ( x_{1},\ldots, x_{d}) \|. 
\end{eqnarray}

In fact, provided that $(\cA,\tau)$ is simple and $\tau$ is the unique tracial state of $\cA$, it is sufficient to prove that 
\begin{eqnarray}\label{eq:strong2}
\limsup_{N \to \infty} \| P ( X_{N,1},\ldots, X_{N,d}) \| \leq \| P ( x_{1},\ldots, x_{d}) \|, 
\end{eqnarray}
(see the proof of Theorem 1.1 in \cite{louder2023strongly}). In fact, in this last case, 
by a compactness argument,
it also implies that for any $P \in \dC[\dF_d]$,
\begin{eqnarray*}
\lim_{N \to \infty} \tau_N (  P ( X_{N,1},\ldots, X_{N,d}) ) = \tau (  P ( x_{1},\ldots, x_{d})) .
\end{eqnarray*}

If $(X_{N,1},\ldots,X_{N,d})$ are random elements of $(\cA_N,\tau_N)$, we will say that $(X_{N,1},\ldots,X_{N,d})$ converge strongly toward $(x_1,\ldots,x_d)$ in probability, if for any $P \in \dC[\dF_d]$, the convergence \eqref{eq:strong1} holds in probability.
Note that if $(X_{N,1},\ldots,X_{N,d})$ and $(x_1,\ldots,x_d)$ are unitaries, then Lemma \ref{le:linea} implies that it is sufficient to prove that for any integer $n\geq 1$ and any $(a_0,\ldots,a_{2d}) \in M_n(\dC)$ with the symmetry condition \eqref{eq:symai}, we have 
$$
\lim_{N \to \infty} \NRM{ a_0 \otimes 1 + \sum_{i=1}^{2d} a_i \otimes X_{N,i} } = \NRM{ a_0 \otimes 1 + \sum_{i=1}^{2d} a_i \otimes x_{i} },
$$
where for $i \in \INT{d}$, we have set $X_{N,i+d} = X_{N,i}^*$ and $x_{i+d} = x_{i}^*$. 
For unitaries, we observe finally  that, by Lemma \ref{le:linea}, to prove that \eqref{eq:strong2} holds (in probability), it is sufficient to prove that for any integer $n \geq 1$ and $(a_0,\ldots,a_{2d}) \in M_n(\dC)$, (in probability) 
\begin{equation}\label{eq:strong2A}
\limsup_{N \to \infty} \NRM{ a_0 \otimes 1 + \sum_{i=1}^{2d} a_i \otimes X_{N,i} } \leq \NRM{ a_0 \otimes 1 + \sum_{i=1}^{2d} a_i \otimes x_{i} }.
\end{equation}

Strong convergence of unitaries is particularly natural and important for unitary representations of a group. More precisely, let $\Gamma$ be a finitely generated group and let $\cA_\Gamma \subset \cB( \ell^2(\Gamma))$ be its reduced $C^*$-algebra generated by the left-regular representation, $\lambda(g)$, $g \in \Gamma$. Let $\rho_N$ be a unitary representation of $\Gamma$ of finite dimension $k_N$, say on $\dC^{k_N}$. Let $(g_1,\ldots,g_d)$ be finite generating set of $\Gamma$. Along a sequence $N \to \infty$, we say that $\rho_N$ converges strongly toward the regular representations of $\Gamma$ if $(\rho_N(g_1),\ldots,\rho_N(g_d)) \in \dU_{k_N}$ converges strongly toward $(\lambda(g_1),\ldots,\lambda(g_d))$. Note that this notion does not depend on the particular choice of the generating set. For example, in this language, \cite{MR3205602} implies that the random unitary representation of the free group $\dF_d$ defined by $\rho_N (g_i) = U_i$ with $(U_1,\ldots,U_d)$ independent Haar distributed on $\dU_N$ converge strongly in probability toward the regular representations of $\dF_d$. Similarly, \cite{MR4024563} implies that the random action of the free group defined by $\rho_N(g_i) = U_i$ with $(U_1,\ldots,U_d)$ independent uniformly distributed on $\dS_N$ converges strongly on $\IND_N^\perp$ toward the regular representations of $\dF_d$. 

\subsection{Absorption principle}

Let $\cA_1,\cA_2$ be unital $C^*$-algebra equipped with tracial faithful normal states.   A well-known property of free Haar unitaries $(u_1,\ldots , u_d)$ in $\cA_1$ is that for any unitaries $(v_1,\ldots v_d)$ in $\cA_2$, $(v_1\otimes u_1,\ldots,v_d \otimes u_d)$ remain free Haar unitaries in $\cA_1 \otimes \cA_2$.  Here, we say $(u_1,\ldots,u_d)$ are free Haar unitaries if they have the same distribution as $(\lambda(g_1), \ldots, \lambda(g_d))$ in the free group reduced $C^*$-algebra $\cA_\star$.
This property implies immediately that $(v_1\otimes U_1,\ldots,v_d\otimes U_N)$ are asymptotically free as $N \to \infty$ if $(U_1,\ldots,U_d)$ are independent Haar unitaries on $\dU_N$. 
For more details about this property (a variant of an absorption principle), we refer to \cite{MR1479116,MR3573218,collins2022asymptotic}. 

Theorem \ref{th:main2} for independent Haar unitaries on $\dU_N$ or $\dO_N$ and Theorem \ref{th:main3} for independent random permutations matrices on $\dS_N$ allow to extend quantitatively this to strong convergence. We obtain, for example, the following result. 

\begin{theorem}
Fix $d \geq 2$ and let $(V_{1},\ldots,V_{d})$ be unitaries in $\dU_n$ and let $N = N_n$  be a diverging integers sequence. We assume that either (i) $(U_1,\ldots,  U_d)$ are independent Haar unitaries on $\dU_N$ or $\dO_N$ and $N \gg (\log n)^{80}$ or (ii) $(U_1,\ldots,  U_d)$ are the restriction to $\IND_N^\perp$ of independent and uniform permutation matrices in $\dS_N$ and $\log N  \gg  \sqrt{(\log \log n) \log n}$.
Then $(U_1 \otimes V_1,\ldots, U_d \otimes V_d)$ converges strongly in probability as $n\to \infty$ toward free Haar unitaries.
\end{theorem}
\begin{proof}
We start with the upper bound \eqref{eq:strong2A}. For $\dO_N$ and $\dU_N$, we apply Theorem \ref{th:main2}  to $\cA_1 = M_k (\dC) \otimes M_n(\dC)$ for fixed $k$. For $\dS_N$, we apply Theorem \ref{th:main3}. Thanks to \eqref{eq:TtoTp}, it implies \eqref{eq:strong2A}. For the converse lower bound, it is guaranteed (as soon as $N\gg1$) by Theorem \ref{th:LBU1} and Theorem \ref{th:LBP}. Note that since $\cA_\star$ is simple \cite{MR374334}, the upper bound was, in fact, sufficient.
\end{proof}

\subsection{Random representations of $\dF_d \times \cdots \times \dF_d$ and a new proof of the Hayes' criterion}
\label{subsec:PT}

The Peterson-Thom conjecture \cite{MR2827095} says that any diffuse, amenable subalgebra of the free group factor $L(F_2)$ is contained in a unique maximal amenable subalgebra.  Hayes proved in \cite{MR4448584} that the following property implies this conjecture:
given a sequence of $2d$-tuples  $(U_{i})_{i \in \INT{2d}}$ of independent and Haar distributed random unitary in $\dU_N$, 
the family $(U_{1}\otimes 1_N,\ldots , U_{d}\otimes 1_N,1_N \otimes U_{d+1}
,\ldots , 1_N\otimes U_{2d})$ in $M_{N^2}(\dC)$ converges strongly in probability, as $N\to\infty$, to $(\lambda(g_1) \otimes 1 ,\ldots, \lambda(g_d) \otimes 1,1 \otimes \lambda(g_{d+1}) , \ldots, 1 \otimes \lambda(g_{d+1}))$ elements of the reduced $C^*$-algebra  of $\dF_d\times \dF_d$ where $(g_1,\ldots,g_d)$ (resp. $(g_{d+1},\ldots,g_{2d})$  are the free generators of the first copy (resp. second) copy of $\dF_d$. In other words, Hayes's criterion amounts to prove that a random unitary representation of dimension $N^2$ of $\dF_d \times \dF_d$ converges strongly in probability toward the regular representations of $\dF_d \times \dF_d$.

As mentioned in the introduction, Hayes found equivalent reformulations of this criterion, and the proof of 
one reformulation involving GUEs was recently done by Belinschi and Capitaine \cite{belinschi2022strong}. The purpose of this subsection is to prove a result that is substantially more general than Hayes' criterion in the sense that more legs of a tensor could be involved and that the Haar unitaries could be replaced by uniform permutation matrices. In other words, we consider random unitary representations and random actions of $\dF_d \times \cdots \times \dF_d$.

We start with the case of Haar unitaries. Let $d \geq 2$ and $k \geq 1$ be integers.  Let $(U_{i,j}), i \in \INT{d}, j \in \INT{k}$ be independent Haar distributed on $\dO_N$ or $\dU_N$. We introduce random unitary matrices in $\dU_N^{\otimes k} \subset \dU_{N^k}$  by setting for $i \in \INT{d}, j \in \INT{k}$,
\begin{equation}\label{eq:Vij}
V_{i,j} = 1_N ^{\otimes(j-1)} \otimes U_{i,j} \otimes 1_N^{\otimes (k-j)},
\end{equation}
that is $V_{i,1} = U_{i,1} \otimes 1_N^{\otimes(k-1)}$, $V_{i,2} = 1_N \otimes U_{i,2} \otimes 1_N^{\otimes (k-2)}$, $\ldots$. By construction, the unitaries $V_{i,j}$ and $V_{i',j'}$ commute if they are on a different leg, that is, $j \ne j'$.

These matrices $(V_{i,j})$ define a random unitary representation of $\dF_d^k$. More precisely, we denote by $(g_{i,j})$ the standard generators of $\dF_d^k$ ($g_{i,j}$ and $g_{i',j}$ free and for $j\ne j'$, $i\ne i'$ $g_{i,j}$ and $g_{i',j}$ Abelian free). The map $\rho_N : g_{i,j} \to V_{i,j}$ defines a unitary representation on $\dU_{N}^{\otimes k}$. Identifying $ \cB(\ell^{2}(\dF^k_d))$ and $\cB(\ell^2(\dF_d)^{\otimes k})$, the left-regular representation of $\dF^k_d$ can be written in $\cA_\star^{\otimes k }$ as
$$
\lambda(g_{ij}) = 1_{\ell^2(\dF_d)}^{\otimes {j-1}} \otimes \lambda(g_i) \otimes 1_{\ell^2(\dF_d)}^{\otimes {k-j}}, 
$$
where as usual $\lambda(g_1), \ldots, \lambda(g_{d})$ in $\cA_\star \subset \cB(\ell^2(\dF_d))$ are the free Haar unitaries associated with the free generators of $\dF_d$. 

The following theorem asserts that the random representation $\rho_N$ converges strongly in probability toward the regular representations of $\dF_d^k$. 
\begin{theorem}\label{th:PTH}
For any fixed integers $d \geq 2$ and $k \geq 1$, as $N\to \infty$, the unitary matrices $(V_{i,j})_{i \in \INT{d}, j \in \INT{k}}$ converges strongly in probability toward $(\lambda(g_{ij}))_{i \in \INT{d}, j \in \INT{k}}$.
\end{theorem}
Before proving Theorem \ref{th:PTH}, let us make a few comments. The analog result for $k= 2$ and for GUE matrices has been proved recently in \cite{belinschi2022strong}. Thanks to \cite{MR4448584}, these results both complete the proof of the Peterson-Thom conjecture. The quantitative version is given in Lemma \ref{le:PTq} below. This result admits multiple variants. For example, we can replace the Haar unitaries with Haar orthogonal matrices or assume that $U_{i,j} \in \dU_{N_j}$, the same result holds, for example, if $\log N_j / \log N_{j'}$ is bounded uniformly from above (and below) for all $j\ne j'$.

 As usual for $i \in \INT{d}$, we set $V_{i+d,j}= V_{i^*,j} = V_{i,j}^*$. Let $n \geq 1$ be an integer and let $(a_0,(a_{i,j})_{i \in \INT{2d}, j \in \INT{k}})$ be matrices in $M_n(\dC)$. We are interested in the matrix in $M_n (\dC) \otimes M_N(\dC)^{\otimes k}$
\begin{equation*}
A_N = a_0 \otimes 1_N^{\otimes k}  + \sum_{i=1}^{2d} \sum_{j=1}^k a_{i,j} \otimes V_{i,j}.
\end{equation*}
In $M_n(\dC) \otimes \cA_\star^{\otimes k}$, we consider the operator 
\begin{equation}\label{eq:defA*PT}
A_{\star^k} = a_0 \otimes 1_N^{\otimes k}  + \sum_{i=1}^{2d} \sum_{j=1}^k a_{i,j} \otimes \lambda(g_{i,j}).
\end{equation}
As explained in Subsection \ref{subsec:strong}, the proof of the theorem amounts to prove that $\|A_N\|$ converges in probability toward $\|A_{\star^k}\|$ (in fact,  since $\cA^{\otimes k}_\star$ is simple by \cite{MR374334}, the upper bound $\limsup_N \|A_N\| \leq \|A_{\star^k}\|$ in probability is sufficient). Therefore, Theorem \ref{th:PTH} follows from the following claim: 

\begin{lemma}\label{le:PTq} Assume $d^{70} \leq N$. For $2 \leq p \leq  N^{1/(16d + 80)}$ even integer and $c >0$ the numerical constant in Theorem \ref{th:main1}, we have
\begin{equation*}\label{eq:PTq}
  \dE \|A_{N} \|  \leq   \| A _{\star^k} \| (2n)^{\frac{k+1}{p}} N^{\frac{k(k+1)}{2p} } p^{\frac{3k(k+1)}{2p}}\PAR{ 1 + \frac{c}{\sqrt  N}}^k.
\end{equation*}
Moreover, if $c >0$ is the numerical constant from Theorem \ref{th:LBU1}, we have 
\begin{equation*} 
  \dE \|A_{N} \|  \geq  \| A _{\star^k} \| \PAR{ 1 - \frac{c \log (2d)}{\log N}}^k.
\end{equation*}
\end{lemma}

We first observe that Haagerup's inequality, Lemma \ref{le:Haagerup}, tensorizes. As usual, for an integer $l\geq 0$, we denote by $S_l \subset \dF_d$ the set of elements of word length $l$. 

\begin{lemma}\label{le:HaagerupFdk}
If $T = \sum_{g \in \dF_d^k} a(g) \otimes \lambda(g) \in M_n(\dC) \otimes \dF_d^k$ is supported on elements $g$ in $S_{l_1} \times \cdots \times S_{l_k}$ then  
$$
\| T \| \leq \sqrt n \PAR{\prod_{j=1}^k (l_j+1) } \| T\|_2.
$$
Moreover, if $A_{\star^k}$ is as in \eqref{eq:defA*PT}, then for any even integer $p \geq 2$, we have
$$
\| A_{\star^k} \|^p \leq 2 n p^{3k} \| A_{\star^k} \|^{p}_p.
$$ 
\end{lemma} 
\begin{proof} 
We check the first claim. The proof is by induction over $k\geq 1$. For $k = 1$, this is Lemma \ref{le:Haagerup}. For general $k\geq 2$, we use that if $g = (x,y) \in \dF_d^k = \dF_d^{k-1} \times \dF_{d}$, we have $\lambda (g) = \lambda(x) \otimes \lambda(y)$ and we may thus write $T$ as 
$$
T =  \sum_{y \in S_{l_k}}  \PAR{ \sum_{x \in \dF_d^{k-1}} a(x,y) \otimes \lambda(x) } \otimes \lambda (y) = \sum_{y \in S_{l_k}}  \tilde a(y) \otimes \lambda (y),
$$ 
We may then apply Lemma \ref{le:Haagerup}.
The second claim is a corollary of the first claim, but it also follows from applying directly Corollary \ref{cor:HaagB} instead of Lemma \ref{le:Haagerup} in the above argument.
\end{proof} 

\begin{proof}[Proof of Lemma \ref{le:PTq}] We are going to make a repeated use of Theorem \ref{th:main1}. To that end, we design an interpolation between $A_N$ and $A_{\star^k}$ which replaces each leg of $\dC^N$ iteratively by $\ell^2(\dF_d)$ and the unitaries $U_{i,j}$ on the corresponding leg by free Haar unitaries $\lambda(g_i)$.   We rename $A_N=A_{N,0}$ and set for $l \in \INT{k}$,  
\begin{eqnarray*}
A_{N,l} & = & a_0 \otimes 1_N^{\otimes k-l}\otimes 1^{\otimes l}_{\ell^2(\dF_d)}  + \sum_{i=1}^{2d}\sum_{j=1}^{k-l}  a_{i,j} \otimes 1_N^{\otimes (j-1)} \otimes U_{i,j} \otimes 1_N^{\otimes (k-j-l)} \otimes 1^l_{\ell^2(\dF_d)}\\
& & \quad +\sum_{i=1}^{2d}\sum_{j=k-l+1}^{k}  a_{i,j} \otimes 1_N^{\otimes (k-l)} \otimes 1^{\otimes (l-k+j-1)}_{\ell^2(\dF_d)}  \otimes \lambda(g_{i}) \otimes 1^{k-j}_{\ell^2(\dF_d)}.
\end{eqnarray*}
The operator $A_{N,l}$ is an element of $M_n(\dC) \otimes M_N(\dC)^{\otimes(k-l)} \otimes \cA_\star^{\otimes l }$. By construction, we have $A_{\star^k} =A_{N,k}$. 

Next, according to Theorem \ref{th:main1}, for our choice of $p$, for each $l \in \INT{k}$,
$$
  \dE_{l} \|A_{N,l-1} \|_p  \leq  \| A_{N,l} \| \PAR{ 1 + \frac{c}{\sqrt  N}},
$$
where $\dE_{l}$ is the expectation conditional to the $\sigma$-algebra generated by the random unitaries  $(U_{i,j}), i \in \INT{d}, j \leq k-l$. In addition, it follows from Lemma \ref{le:HaagerupFdk} that
$$\| A_{N,l} \|\le \|A_{N,l} \|_p \PAR{ 2  n N^{k-l}p^{3l}}^{1/p},$$
so, by combining, we get 
$$
  \dE_{l} \|A_{N,l} \|_p  \leq  \| A_{N,l} \|_p \PAR{ 2 n N^{k-l}p^{3l}}^{1/p} \PAR{ 1 + \frac{c}{\sqrt  N}}.
$$
Iterating and using Fubini's Theorem, we obtain
\begin{eqnarray*}
 \dE \|A_{N} \|_p  & \leq &   \| A _{\star^k} \| \prod_{l=1}^k \PAR{2 n N^{k-l}p^{3l}}^{1/p} \PAR{ 1 + \frac{c}{\sqrt  N}}^k.
\end{eqnarray*}
Since $\| A_N \| \leq (n N^k)^{1/p} \|A_N\|_p$ by \eqref{eq:TtoTp}, we obtain the claimed bound. The proof of the lower bound is obtained by using instead that $
  \dE_{l} \|A_{N,l-1} \|  \geq  \| A_{N,l} \| \PAR{ 1 + c \log (2d) / \log  N},
$ from Theorem \ref{th:LBU1}.
\end{proof}

We conclude with an application to random actions of $\dF_d^k$. Let $(\sigma_{i,j}), i \in \INT{d}, j \in \INT{k},$ be independent and uniform permutations on $\dS_N$. We denote by $U_{ij}$ their permutations matrices and define $V_{ij}$ as in \eqref{eq:Vij}. This defines uniquely a random action of $\dF_d^k$ on $\INT{N}^k$ by setting:
$$
\rho_N (g_{ij}) = V_{ij}.
$$
Identifying $\dC^{N^k}$ with tensors with $k$ legs of dimension $N$, this representation corresponds to permuting the coordinates of each leg separately. 
We denote by $H_N \subset \dC^N$ the hyperplane orthogonal to $\IND_N$, the vector with constant coordinates (that is, sum zero on each leg). The vector space $H_N^{\otimes k}$ is an invariant subspace of $\rho_N$. The following theorem asserts that the representation $\rho^0_N$, the restriction of $\rho_N$ to $H_N^{\otimes k}$, converges strongly in probability toward the regular representations of $\dF_d^k$. 

\begin{theorem}\label{th:randomactionFdk}
For any fixed integers $d \geq 2$ and $k \geq 1$, as $N\to \infty$, the unitary matrices $((V_{i,j})_{|H_N^{\otimes k}})_{i \in \INT{d}, j \in \INT{k}}$ converges strongly in probability toward $(\lambda(g_{ij}))_{i \in \INT{d}, j \in \INT{k}}$.
\end{theorem} 

The proof is the same as the proof of Theorem \ref{th:PTH} except that we now use Theorem \ref{th:main3} and Theorem \ref{th:LBP} (instead of Theorem \ref{th:main1} and Theorem \ref{th:LBU1}). Interestingly, if we consider a uniformly distributed random action of $\dF_d^k$ on $\INT{N}$ then the conclusion of Theorem \ref{th:randomactionFdk} on $\IND_N^\perp$ is not true, there is not even convergence in distribution, see \cite{313910}.

\section*{Appendix: Non-backtracking identities}
\label{sec:appendix-nonback}

In this appendix, we consider the general setting for operators of the form \eqref{eq:defA*}-\eqref{eq:defA}. Let $\cA_1$ and $\cA_2$ be two unital $C^*$-algebras of bounded operators on the Hilbert spaces $H_1$ and $H_2$. We define the operator in the minimal tensor product $\cA_1 \otimes \cA_2$,
\begin{equation}\label{eq:defAH}
A = a_0 \otimes 1_{H_2}  + \sum_{i=1}^{2d} a_i \otimes u_i,
\end{equation}
with $u_i \in \cA_2$ unitaries such that for all $i \in \INT{2d}$, $u_{i^*} = u_i^*$ and $a_i \in \cA_1$. Under the condition  \eqref{eq:symai}, the operator $A$ is self-adjoint.

Similarly, if $b_1,\ldots, b_{2d}$ are elements in $\cA_1$ and $E_{ij} = e_i \otimes e_j \in M_{2d} (\dC)$ denotes the canonical matrix, we consider the operator in $\cA_1 \otimes \cA_2 \otimes M_{2d}(\dC)$: 
\begin{equation}\label{eq:defB}
B = \sum_{(i,j) : i \ne j^*} b_i  \otimes u_i \otimes E_{ji}.
\end{equation}

As above, let $\cA_\star$ be the left group algebra of the free group $\dF_d$. We consider the operator $A_{\star}$ defined in \eqref{eq:defA*}. We denote by $\o$ the unit of $\dF_d$.  If $g \in \dF_d$, we define $\Pi_g$ : $H_1  \otimes \ell^2(\dF_d) \to H_1$ as the orthogonal projection onto $H_1 \otimes \delta_{g}$, where $\delta$ denotes a Dirac vector.  If $T$ is an operator in $\cA_1 \otimes \cA_\star$, we set  $T_{xy} = \Pi_x T \Pi^*_y \in \cA_1$. For example, for $x,y \in \dF_d$, we have  
$(A_\star)_{xy} = a_0 \langle  \delta_x,  \delta_y \rangle + \sum_i a_i \langle  \delta_x, \lambda(g_i)  \delta_y \rangle. 
$
For $i \in \INT{2d}$, we define similarly $\Pi_{g,i} :  H_1 \otimes \ell^2(\dF_d) \otimes \dC^{2d} \to H_1$ as the orthogonal projection onto  $H_1 \otimes \delta_{g} \otimes \delta_i$.

For $z \in \dC \backslash \sigma(A_ \star)$, we set 
$$
\gamma(z) = \left( ( z- A_{\star})^{-1} \right)_{\o \o}.
$$

Recall that $A_\star^{\o}$ is the restriction of the operator $A_\star$ to $H_1 \otimes \ell^2 (V)$ where $V =  \dF_d\backslash \{\o\}$. For $z \in \dC \backslash  \sigma (A^{\o}_\star) $, we set 
$$
\gamma_i (z) = \left( (z - A^{\o}_{\star} )^{-1} \right)_{g_i g_i}.
$$
We also define
$$\hat \sigma(A_\star) =  \sigma(A_\star)  \cup \sigma(A_\star^{\o}).$$
 Note that if the symmetry condition \eqref{eq:symai} holds then, from Courant-Fisher min-max theorem, $\hat \sigma(A_\star)$ is contained in the convex hull of $\sigma(A_\star)$ and, if $\dim(H_1) = n$, $\sigma(A^{\o}_\star)\backslash  \sigma(A_\star)$ is a finite set of at most $2dn$ points.

Coming back to $\cA_2$, for $z \notin \hat\sigma(A_\star)$, we set $B(z)$ as the non-backtracking operator \eqref{eq:defB} in $\cA_1 \otimes \cA_2 \otimes M_{2d}(\dC)$:   associated with the elements for all $i \in \INT{2d}$,
$$b_i(z)  = \gamma_{i}(z) a_i.$$  

 We finally introduce the canonical projection operator $ P : H_1 \otimes H_2 \otimes \dC^{2d} \to H_1 \otimes H_2 $ defined by, for all $\psi \in H_1 \otimes H_2$,  $ i \in \INT{2d}$,
 $$
P \psi \otimes \delta_i =  \psi.
$$
Since $P P^*  = 2 d 1$, it convenient to introduce the isometry 
$$
\hat P = \frac{P}{\sqrt{2d}}.
$$
Let $D$ be a bounded set in $\dC$, we define $\FULL( D ) = \dC \backslash U$ where $U$ is the unique infinite component of 
$\dC \backslash D$ (in loose words, $\FULL(D)$ fills the holes of $D$). For example, if $D \subset \dR$ 
or $D$ is simply connected, then $\FULL(D) = D$. The following formula holds. 
\begin{theorem}\label{th:resIB}
There exists $K>0$ such that for all $z, |z|>K$,
$$
(z - A)^{-1} =  \hat P \PAR{ \Bigm(1 - B(z) \Bigm)^{-1} \Bigm(  1 + \frac{B(z)}{2d-1} \Bigm) }  \hat P^*   ( \gamma(z) \otimes 1_{H_2}).
$$
\end{theorem}

Let us remark that we believe that this formula is true for all $ z  \notin  \FULL(\hat \sigma(A_\star)) \cup \sigma(A)$, where the left hand side is well defined and uniquely determined by analytic continuation by its values for $z, |z|>K$, however this statement requires estimates on the spectrum of the operators on the left hand side which we don't have in infinite dimension. 

This formula is closely related to various spectral identities known as Ihara-Bass identities; in our context, see notably \cite{MR1749978,NIPS2009_0420,anantharaman2017relations,MR4024563,MR3961083} and references therein. An interesting aspect of the present theorem is that it is in an identity in its strongest form, i.e., between two operators.

\begin{proof}[Proof of Theorem \ref{th:resIB}]
We start with a simple alternative reformulation of the right-hand side of the above identity. Recall that $P = \sqrt {2d} \hat P$.  Using the resolvent identity, we have as long as $1 -B(z)$ is invertible
$(1 -B(z) )^{-1}  - (2d)^{-1} 1 = (1 - B(z) )^{-1} ( 2d - 1 + B(z) ) (2d)^{-1}$ and thus
\begin{equation}\label{eq:resid1}
\hat P \PAR{ \Bigm(1 - B(z) \Bigm)^{-1} \Bigm(  1 + \frac{B(z)}{2d-1} \Bigm) }  \hat P^*   = \frac{1}{2d-1} P \PAR{ (1 - B(z) )^{-1} - \frac{1}{2d}  }  P^*.
\end{equation}

From the morphism property $u ( S T) = u(S) u(T)$ for $S, T \in \cM_\star$, it follows that it is sufficient to prove the claim of the theorem when $\cA_2 = \cA_\star$ and $u_i = \lambda(g_i)$.

We thus assume from now on that $H_2 = \ell^2 (\dF_d)$ and $u_i = \lambda(g_i)$. We set $A_\star = A$ and $B_\star (z)= B$ for ease of notation. Since $\gamma(z)$ and $(z- A)^{-1}$ are analytic in $z \in \dC \backslash \FULL(\hat \sigma(A_\star))$, it is sufficient to prove the claimed identity for $|z|$ large enough. It will ultimately follow from basic considerations on generating a series of walks in $\dF_d$. If $|z|$ is large enough,  we have the converging series expansion:
$$
(z - A) ^{-1} = z^{-1} \sum_{k\geq 0} \frac{A^k}{z^k}, \quad  \hbox{ and } \quad   P(1 - B(z)) ^{-1} P^* =   \sum_{k \geq 0} P B^k P^*.
$$
We set  $$\hat \gamma = z \gamma(z) = \sum_k \frac{(A^k)_{\o \o}}{z^k} \quad  \hbox{ and } \quad  \hat \gamma_i  =  z \gamma_i (z) =   \sum_k \frac{((A^{\o})^k)_{g_i g_i}}{z^k} $$  
We fix any $g$ in $\dF_d$ of length $|g| = l$ written in reduced form as  $g= g_{i_l} \cdots g_{i_1}$ (if $l \geq 1$). Hence, from \eqref{eq:resid1}, by translation invariance, to prove the theorem it suffices to prove that 
\begin{equation}\label{eq:formula1}
\sum_{k\geq 0} \frac{(A^k)_{ g \o} }{z^k} = \hat \gamma \IND_{g = \o} + \frac{1}{2d-1}\sum_{k\geq 1}  \frac{(P \hat B^k P^*) _{g \o}}{z^k}  \hat \gamma,
\end{equation}
where $\hat B$ is defined in \eqref{eq:defB} with $b_i = \hat \gamma_i a_i$ and we have used that  $PP^* = 2d 1$. 

By construction, we have for all $i \in \INT{2d}$ and $\psi \in H_1$, for integer $k \geq 1$.
\begin{equation}\label{eq:Bkdeta0}
(\hat B^k ) \psi \otimes \delta_{\o} \otimes \delta_i = \sum_{h = (g_{j_{k}}\cdots g_{j_1}), j_{k+1} \ne j_k^*} \hat \gamma_{j_k} a_{j_k} \cdots \hat \gamma_{j_1} a_{j_1}\psi \otimes \delta_{h} \otimes \delta_{j_{k+1}}, 
\end{equation}
where the sum is over all $h \in \dF_d$ of length $k$ written in reduced form $h = (g_{j_{k}}\cdots g_{j_1})$ with $j_1 = i$.

If follows that the summand on the right-hand side of \eqref{eq:formula1} is $0$ unless $k = l  = |g|$ is the length of $g =  (g_{i_{l}}\cdots g_{i_1})$. We thus deduce that \eqref{eq:formula1} is true for $g = e$. For $g \ne e$,  $l \geq 1$, we find that the right-hand side of \eqref{eq:formula1} is equal to 
$$
  \hat \gamma \IND_{g = \o} + \frac{1}{2d-1}\sum_{k\geq 1}  \frac{(P \hat B^k P^*) _{g \o}}{z^k}  \hat \gamma = \frac{1}{2d-1} \sum_{k \geq 1}  \frac{(P \hat B^k P^*) _{g \o}}{z^k}  \hat \gamma = \frac{\hat \gamma_{i_{l}} a_{i_l}\cdots \hat \gamma_{i_1 } a_{i_1} \hat \gamma}{z^{l}},
$$
where the factor $1/(2d-1)$ compensate the $2d-1$ possibilities for $j_{k+1}$ in \eqref{eq:Bkdeta0}.
On the other hand, Equation \eqref{eq:Akgo0} reads 
\begin{equation*}
(A^k)_{g \o} = \sum_{(j_1,\ldots,j_k) :  g = g_{j_k}\cdots g_{j_1} }a_{j_k} \cdots a_{j_1},
\end{equation*}
where the sum is over $(j_1,\ldots,j_k) \in \{0,\ldots,2d\}^k$ and $g_0 = \o$ by convention.

Set $V_{\o} = \dF_d$ and recall that, for $i \in \INT{2d}$, $V_i \subset \dF_d$ is the subset of group elements $h \in \dF_d$ written in reduced form $h = g_{j_k} \cdots g_{j_2} g_{j_1}$ with $j_1 = i$. 
A sequence $(j_1, \ldots, j_k)$ such that $g  = (g_{j_k}\cdots g_{j_1}) $  determines a walk of length $k$ in $\dF_d$ starting from $\o$ and ending at $g$. Assume that $g = g_{i_l}\cdots g_{i_1}$ in reduced form. By decomposing this walk at the last passage times at $\o$, $g_{i_1}$, $g_{i_2} g_{i_1}$, $\ldots$, $g_{i_{l}}\cdots g_{i_1} = g$, the sequence $(j_1, \ldots, j_k)$ is uniquely decomposed into $l +1 $ sequences $u_0, \ldots u_l$ possibly empty such that for $0 \leq t \leq l$ and $u_t = (j_{t,1}, \ldots , j_{t,n_t})$, $(g_{i_t}, g_{j_{t,1}} g_{i_t}, \cdots , g_{j_{t,n_t}} \cdots g_{j_{t,1}} g_{i_t} )$ is a walk from $g_{i_{t}}$ to  $g_{i_{t}}$ in $V_{i_{t}}$ and $(j_1,\ldots,j_k) = (u_0,i_1, u_1 , \ldots, i_l, u_l)$. We thus have 
$$ \sum_k \frac{(A^k)_{g \o}}{z^k} = \hspace{-0.4cm}\sum_{(j_1,\ldots,j_k) :  g = (g_{j_k}\cdots g_{j_1}) } \frac{a_{j_k} \cdots a_{j_1}}{z^k} = \sum_{u_0,\ldots,u_l}  \frac{a_{j_{l,n_l}} \cdots a_{j_{l,1}}}{z^{n_l}}\frac{a_{i_l}}{z} \cdots \frac{a_{j_{1,n_1}} \cdots a_{j_{1,1}}}{z^{n_1}} \frac{a_{i_1}}{z} \frac{a_{j_{0,n_0}} \cdots a_{j_{0,1}}}{z^{n_0}},$$
where the sum is over all sequences $(u_0,\ldots,u_l)$ as above.  The identity \eqref{eq:formula1} follows
for $z$ of modulus large enough. \end{proof}

\bibliographystyle{abbrv}
\bibliography{bib}

\vline

\noindent 
Charles Bordenave \\
{\em Institut de Math\'ematiques de Marseille \\
CNRS \& Aix-Marseille University \\
39, rue F. Joliot Curie, 13453 Marseille Cedex 13, France.} \\
Email: \href{mailto:charles.bordenave@univ-amu.fr}{charles.bordenave@univ-amu.fr}

\vline

\noindent
Beno\^it Collins \\
{\em Department of Mathematics, Graduate School of Science \\
Kyoto University \\
 Kyoto 606-8502, Japan.} \\
Email: \href{mailto:collins@math.kyoto-u.ac.jp}{collins@math.kyoto-u.ac.jp}
\end{document}